\documentclass{amsart}
\usepackage{amsmath}
\usepackage{amssymb}
\usepackage[all]{xy}
\usepackage{graphicx}
\usepackage{mathrsfs}

\newcommand{\C}{\mathbb{C}}
\newcommand{\F}{\mathbb{F}}

\newcommand{\Qlb}{\overline{\mathbb{Q}}_\ell}
\newcommand{\fK}{\mathfrak{K}}
\newcommand{\fO}{\mathfrak{O}}
\newcommand{\Z}{\mathbb{Z}}

\newcommand{\cO}{\mathcal{O}}
\newcommand{\Fr}{\mathsf{Fr}}
\newcommand{\Gm}{\mathbb{G}_m}

\newcommand{\bP}{\mathbb{P}}

\newcommand{\pt}{\mathrm{pt}}

\newcommand{\scA}{\mathscr{A}}
\newcommand{\scD}{\mathscr{D}}

\newcommand{\Kb}{K^{\mathrm{b}}}
\newcommand{\cone}{\mathrm{cone}}

\newcommand{\Proj}{\mathsf{Proj}}

\newcommand{\Pure}{\mathsf{Pure}}

\newcommand{\Mod}{\mathsf{Mod}}

\newcommand{\Gv}{{\check G}}
\newcommand{\Tv}{{\check T}}
\newcommand{\Bv}{{\check B}}
\newcommand{\Lv}{{\check L}}
\newcommand{\Pv}{{\check P}}
\newcommand{\Iv}{{\check I}}

\newcommand{\Uv}{{\check U}}
\newcommand{\GvO}{\Gv(\fO)}
\newcommand{\LvO}{\Lv(\fO)}
\newcommand{\TvO}{\Tv(\fO)}
\newcommand{\fg}{\mathfrak{g}}
\newcommand{\ft}{\mathfrak{t}}
\newcommand{\ftv}{{\check \ft}}
\newcommand{\fgv}{{\check \fg}}

\newcommand{\bX}{\mathbf{X}}
\newcommand{\bXv}{{\check \bX}}

\newcommand{\bXp}{\mathbf{X}^+}

\newcommand{\Rv}{{\check R}}

\newcommand{\Fl}{\mathsf{Fl}}
\newcommand{\Gr}{\mathsf{Gr}}

\newcommand{\cN}{{\mathcal{N}}}

\newcommand{\cB}{\mathcal{B}}
\newcommand{\wcN}{\widetilde{\mathcal{N}}}

\newcommand{\la}{\lambda}

\newcommand{\rhov}{{\check \rho}}
\newcommand{\alv}{{\check \alpha}}

\newcommand{\Weil}{\mathrm{Weil}}
\newcommand{\mix}{\mathrm{mix}}
\newcommand{\geom}{\mathrm{geom}}

\newcommand{\dm}{\cD^\mix}

\newcommand{\cDb}{\cD^{\mathrm{b}}}
\newcommand{\cCb}{\cC^{\mathrm{b}}}

\newcommand{\Perv}{\mathsf{P}}
\newcommand{\Pervm}{\mathsf{P}^\mix}
\newcommand{\Pervw}{\mathsf{P}^\Weil}

\newcommand{\IC}{\mathrm{IC}}
\newcommand{\ICm}{\mathrm{IC}^\mix}

\newcommand{\uQlb}{\underline{\overline{\mathbb{Q}}}_\ell}

\newcommand{\free}{{\mathrm{free}}}
\newcommand{\DGM}{\mathsf{DGMod}}
\newcommand{\DGMG}{\DGM^G}
\newcommand{\DGMGm}{\DGM^{G \times \Gm}}
\newcommand{\DGC}{\mathsf{DGCoh}}

\newcommand{\DGCGf}{\DGC^G_\free}
\newcommand{\DGCGmf}{\DGC^{G \times \Gm}_\free}

\newcommand{\sA}{\mathsf{A}}
\newcommand{\KPG}{\mathsf{Kproj}^G}

\newcommand{\cF}{\mathcal{F}}

\newcommand{\cK}{\mathcal{K}}

\newcommand{\cH}{\mathcal{H}}
\newcommand{\cS}{\mathcal{S}}
\newcommand{\cM}{\mathcal{M}}
\newcommand{\bO}{\mathbb{O}}
\newcommand{\wbO}{\widetilde{\mathbb{O}}}

\newcommand{\fR}{\mathfrak{R}}

\newcommand{\cL}{\mathcal{L}}
\newcommand{\cR}{\mathcal{R}}

\newcommand{\lotimes}{{\stackrel{_L}{\otimes}}}

\newcommand{\tsqtimes}{\mathrel{\widetilde{\boxtimes}}}

\newcommand{\lan}{\langle}
\newcommand{\ran}{\rangle}

\newcommand{\D}{\mathbb{D}}
\newcommand{\uC}{\underline{\C}}

\newcommand{\For}{\mathrm{For}}

\newcommand{\sfG}{\mathsf{G}}
\newcommand{\sfL}{\mathsf{L}}

\newcommand{\cD}{\mathcal{D}}

\newcommand{\Coh}{\mathsf{Coh}}
\newcommand{\QCoh}{\mathsf{QCoh}}
\newcommand{\Rep}{\mathsf{Rep}}
\newcommand{\Vect}{\mathsf{Vect}}
\newcommand{\cW}{\mathcal{W}}

\newcommand{\id}{\mathrm{id}}

\newcommand{\sE}{\mathsf{E}}
\newcommand{\fr}{\mathrm{free}}
\newcommand{\prj}{\mathrm{proj}}
\newcommand{\sm}{\mathsf{m}}

\newcommand{\scB}{\mathscr{B}}
\newcommand{\scC}{\mathscr{C}}
\newcommand{\scE}{\mathscr{E}}

\newcommand{\scS}{\mathscr{S}}
\newcommand{\cC}{\mathcal{C}}

\newcommand{\fS}{\mathfrak{S}}
\newcommand{\fT}{\mathfrak{T}}
\newcommand{\fX}{\mathfrak{X}}

\DeclareMathOperator{\Hom}{Hom}
\DeclareMathOperator{\Ext}{Ext}
\DeclareMathOperator{\Ind}{Ind}

\DeclareMathOperator{\rad}{rad}

\newcommand{\aff}{\mathrm{aff}}
\newcommand{\lf}{\mathrm{lf}}
\newcommand{\qis}{\mathrm{qis}}
\newcommand{\pr}{\mathsf{prim}}
\newcommand{\FBK}{\mathrm{F}^{\mathrm{BK}}}
\newcommand{\prt}{\mathsf{pretr}}

\def\mon#1{\text{\rm $#1$-mon}}
\def\eq#1{\text{\rm $#1$-eq}}

\newtheorem{thm*}{Theorem}
\numberwithin{equation}{section}
\newtheorem{thm}{Theorem}[section]
\newtheorem{lem}[thm]{Lemma}
\newtheorem{prop}[thm]{Proposition}
\newtheorem{cor}[thm]{Corollary}
\theoremstyle{definition}

\theoremstyle{remark}
\newtheorem{rmk}[thm]{Remark}

\title[Constructible sheaves on affine Grassmannians]{Constructible sheaves on affine Grassmannians and geometry of the dual nilpotent cone}

\author{Pramod N. Achar}
\address{Department of Mathematics, Louisiana State University, Baton Rouge, LA 70803, USA}
\email{pramod@math.lsu.edu}

\author{Simon Riche}
\address{Clermont Universit{\'e}, Universit{\'e} Blaise Pascal, Laboratoire de  
Math{\'e}ma\-tiques, BP 10448, F-63000 Clermont-Ferrand. \newline
\indent CNRS, UMR 6620, Laboratoire de Math{\'e}matiques, F-63177 Aubi{\`e}re.
}
\email{simon.riche@math.univ-bpclermont.fr}

\begin{document}

\begin{abstract}
In this paper we study the derived category of sheaves on the affine Grassmannian of a complex reductive group $\Gv$, contructible with respect to the stratification by $\Gv(\C[[x]])$-orbits. Following ideas of Ginzburg and Arkhipov--Bezrukavnikov--Ginzburg, we describe this category (and a mixed version) in terms of coherent sheaves on the nilpotent cone of the Langlands dual reductive group $G$. We also show, in the mixed case, that restriction to the nilpotent cone of a Levi subgroup corresponds to hyperbolic localization on affine Grassmannians.
\end{abstract}

\maketitle


\section{Introduction}
\label{sect:intro}

\subsection{}

Let $\Gv$ be a complex connected reductive group, and let 
\[
\Gr_{\Gv}:=\Gv(\fK)/\GvO
\]
be the associated affine Grassmannian (where $\fK=\C((x))$ and $\fO=\C[[x]]$). The Satake equivalence is an equivalence of tensor categories 
\[
\cS_G : \Rep(G) \ \xrightarrow{\sim} \ \Perv_{\eq{\GvO}}(\Gr_{\Gv})
\]
between the category $\Perv_{\eq{\GvO}}(\Gr_{\Gv})$ of $\GvO$-equivariant perverse sheaves on $\Gr_{\Gv}$ (endowed with the convolution product) and the category $\Rep(G)$ of finite-dimen\-sional $G$-modules (endowed with the tensor product), where $G$ is the complex reductive group which is Langlands dual to $\Gv$. This equivalence is ``functorial with respect to restriction to a Levi'' in the sense that, if $L \subset G$ is a Levi subgroup and $\Lv \subset \Gv$ a dual Levi subgroup, the restriction functor $\Rep(G) \to \Rep(L)$ can be realized geometrically as a (renormalized) hyperbolic localization functor 
\[
\fR^G_L : \Perv_{\eq{\GvO}}(\Gr_{\Gv}) \ \to \ \Perv_{\eq{\LvO}}(\Gr_{\Lv})
\]
in the sense of Braden \cite{br}.

The forgetful functor 
\[
\Perv_{\eq{\GvO}}(\Gr_{\Gv}) \ \to \ \Perv_{\mon{\GvO}}(\Gr_{\Gv}),
\]
where $\Perv_{\mon{\GvO}}(\Gr_{\Gv})$ is the category of perverse sheaves constructible with respect to the stratification by $\GvO$-orbits, is an equivalence of categories. Hence the category $\Perv_{\eq{\GvO}}(\Gr_{\Gv})$ is naturally the heart of a $t$-structure on the full subcategory $\cDb_{\mon{\GvO}}(\Gr_{\Gv})$ of the derived category of constructible sheaves on $\Gr_{\Gv}$ whose objects have their cohomology sheaves constructible with respect to the stratification by $\GvO$-orbits. Therefore, one can ask two natural questions:
\begin{enumerate}

\item Is it possible to describe the category $\cDb_{\mon{\GvO}}(\Gr_{\Gv})$ in terms of the geometry of the group $G$?

\item Is this description functorial with respect to restriction to a Levi subgroup?

\end{enumerate}

\subsection{}

First, consider question (1). In \cite{g2}, Ginzburg has explained how to describe morphisms in $\cDb_{\mon{\GvO}}(\Gr_{\Gv})$ between shifts of simple objects of $\Perv_{\mon{\GvO}}(\Gr_{\Gv})$, in terms of coherent sheaves on the nilpotent cone $\cN_G$ of $G$. The next step towards answering the question is to construct a functor from the category $\cDb_{\mon{\GvO}}(\Gr_{\Gv})$ to a certain category related to coherent sheaves on $\cN_G$. This question is rather subtle, due to the lack of extra structure on the category $\cDb_{\mon{\GvO}}(\Gr_{\Gv})$ (such as a convolution product). However, one can adapt the constructions of Arkhipov--Bezrukavnikov--Ginzburg in \cite{abg} to construct a functor from $\cDb_{\mon{\GvO}}(\Gr_{\Gv})$ to $\DGCGf(\cN_G)$, where $\DGCGf(\cN_G)$ is a certain modified version of the derived category of $G$-equi\-variant coherent sheaves on $\cN_G$, where $G$ acts on $\cN_G$ by conjugation (see \S \ref{ss:equivalence} below for a precise definition). Then, it follows from Ginzburg's result that 
\begin{equation}
\label{eqn:equivalence-intro}
F_G : \cDb_{\mon{\GvO}}(\Gr_{\Gv}) \ \xrightarrow{\sim} \ \DGCGf(\cN_G)
\end{equation}
is an equivalence of triangulated categories.

Note that this result can be deduced directly from the results of \cite{abg} in the case $G$ is semisimple of adjoint type. Instead, we give two direct proofs of equivalence \eqref{eqn:equivalence-intro}. The first one is based on the same construction as in \cite{abg}, but is slightly simpler. This construction uses a refinement of the stratification by $\GvO$-orbits, namely the stratification by orbits of an Iwahori subgroup of $\GvO$. The latter is better-behaved than the former, e.g.~because, due to results of Beilinson--Ginzburg--Soergel in \cite{bgs}, the category of perverse sheaves for this stratification has enough projectives, and its derived category is equivalent to the associated constructible derived category. Another central argument is the formality of a certain dg-algebra, see \S \ref{ss:formality}.

Our second proof of equivalence \eqref{eqn:equivalence-intro}, inspired by the methods of \cite{bf}, is completely formal, and based on the notion of \emph{enhanced triangulated category}. The main argument is again a formality result (at the categorical level), similar to the one used in the first proof.

\subsection{}

Now, consider question (2). Here our answer is less satisfactory. The functor $\fR^G_L$ induces a triangulated functor 
\[
\fR^G_L : \cDb_{\mon{\GvO}}(\Gr_{\Gv}) \ \to \ \cDb_{\mon{\LvO}}(\Gr_{\Lv}).
\]
On the coherent side of the picture, the natural functor to consider is the inverse image functor
\[
(i_L^G)^* : \DGCGf(\cN_G) \ \to \ \DGC^L_{\fr}(\cN_L)
\]
for the inclusion $\cN_L \hookrightarrow \cN_G$. It would be natural to expect that there exists an isomorphism of functors
\begin{equation}
\label{eqn:isom-functors-intro}
(i_L^G)^* \circ F_G \ \cong \ F_L \circ \fR^G_L.
\end{equation}
However, we were not able to prove this fact. More precisely, it is easy to check that the images of perverse sheaves under both functors appearing in \eqref{eqn:isom-functors-intro} coincide. It can also be checked (see Proposition \ref{prop:action-functors-morphisms}) that the action of both functors on morphisms between shifts of perverse sheaves can be identified. However, we were not able to construct a morphism of functors between $(i_L^G)^* \circ F_G$ and $F_L \circ \fR^G_L$. The main reason is that the functor $\fR^G_L$ is not well-behaved on the category of perverse sheaves constructible with respect to orbits of an Iwahori subgroup, and hence is ``not compatible with the construction of $F_G$.''

To be able to give an answer to (a variant of) question (2), we have to introduce more rigidity (or more structure) on the category $\cDb_{\mon{\GvO}}(\Gr_{\Gv})$. This rigidity is provided by an additional grading, related to weights of Frobenius. In fact, we replace the category $\cDb_{\mon{\GvO}}(\Gr_{\Gv})$ by its ``mixed version'' $\dm_{\mon{\GvO}}(\Gr_{\Gv})$, defined and studied (in a general context) in \cite{ar}. (This definition is based in an essential way on the results of \cite{bgs}.) An easy generalization of the constructions alluded to above yields an equivalence of categories
\begin{equation}
\label{eqn:equivalence-intro-mix}
F_G^{\mix} : \dm_{\mon{\GvO}}(\Gr_{\Gv}) \ \xrightarrow{\sim} \ \cDb_{\fr} \Coh^{G \times \Gm}(\cN_G),
\end{equation}
where $\cDb_{\fr} \Coh^{G \times \Gm}(\cN_G)$ is a certain subcategory of the derived category of $G \times \Gm$-equivariant coherent sheaves on $\cN_G$, where $\Gm$ acts on $\cN_G$ by dilatation. (Note that a similar construction was already considered in \cite{abg}.)

Using again general constructions from \cite{ar}, one can define a ``mixed version''
\[
\fR^{G,\mix}_L : \dm_{\mon{\GvO}}(\Gr_{\Gv}) \ \to \ \dm_{\mon{\LvO}}(\Gr_{\Lv})
\]
of the functor $\fR^G_L$. On the coherent side of the picture, one again has an inverse image functor
\[
(i_L^G)^*_{\mix} : \cDb_{\fr} \Coh^{G \times \Gm}(\cN_G) \ \to \ \cDb_{\fr} \Coh^{L \times \Gm}(\cN_L).
\]
Our main result is the proof of an isomorphism of functors
\begin{equation}
\label{eqn:isom-functors-intro-mix}
(i_L^G)^*_{\mix} \circ F_G^{\mix} \ \cong \ F_L^{\mix} \circ \fR^{G,\mix}_L.
\end{equation}
This proof is based on the observation that the category $\cDb_{\fr} \Coh^{G \times \Gm}(\cN_G)$ is the bounded homotopy category of an \emph{Orlov category} in the sense of \cite{ar}. Then \eqref{eqn:isom-functors-intro-mix} is a consequence of a general uniqueness result on functors between bounded homotopy categories of Orlov categories.

\subsection{Contents of the paper}

In Section \ref{sect:statements} we state precisely the main results of this paper. In Section \ref{sect:proof-1} we give a first proof of equivalences \eqref{eqn:equivalence-intro} and \eqref{eqn:equivalence-intro-mix}. In Section \ref{sect:proof-2} we give a second proof of these equivalences. In Section \ref{sect:relation-ABG} we explain the relation between our results and the main results of \cite[Part II]{abg}. In Section \ref{sect:hl-restriction}, we prove isomorphism \eqref{eqn:isom-functors-intro-mix}. This section also contains a new proof of a result of Ginzburg \cite{g2} on the geometric realization of the Brylinski--Kostant filtration in terms of perverse sheaves, which may be of independent interest. In Sections  \ref{sect:mix-convolution} and \ref{sect:hl-convolution} we study two related questions: the compatibility of our equivalence \eqref{eqn:equivalence-intro-mix} with convolution of (mixed) perverse sheaves, and compatibility of hyperbolic localization with convolution. Finally, in Section \ref{sect:example} we describe some of our objects of study more concretely in the case $\Gv=\mathrm{SL}(2)$.

\subsection{Conventions}

If $X$ is a complex algebraic variety endowed with an action of an algebraic group $H$, recall that an object $\cF$ of the derived category of sheaves on $X$ is said to be \emph{$H$-monodromic} if for any $i \in \Z$, the sheaf $\cH^i(\cF)$ is constructible with respect to a stratification whose strata are $H$-stable. We denote by $\cDb_{\mon{H}}(X)$ the subcategory of the derived category of sheaves on $X$ whose objects are $H$-monodromic, and by $\Perv_{\mon{H}}(X)$ the subcategory of perverse sheaves. We use the same terminology and notation for $\Qlb$-sheaves on varieties defined over an algebraically closed field of characteristic $p \neq l$. We also denote by $\cDb_{\eq{H}}(X)$ the equivariant derived category of sheaves on $X$ (see \cite{bl}), and by $\Perv_{\eq{H}}(X)$ the subcategory of perverse sheaves.

We will often work with $\Qlb$-sheaves on varieties defined over a finite field $\F_p$, where $p \neq l$ is prime. For our considerations, the choice of $l$ and $p$ is not important; we fix it once and for all. We fix a square root of $q$ in $\Qlb$, which defines a square root of the Tate sheaf on any such variety. We denote for any $i \in \Z$ by $\lan i \ran$ the Tate twist $(-\frac{i}{2})$. We also fix an isomorphism $\Qlb \cong \C$. If $X$ is a variety over $\F_p$, endowed with an action of an algebraic group $H$, we say that a perverse sheaf on $X$ is $H$-monodromic if its pull-back to $X \times_{\mathrm{Spec}(\F_p)} \mathrm{Spec}(\overline{\F_p})$ is $H \times_{\mathrm{Spec}(\F_p)} \mathrm{Spec}(\overline{\F_p})$-monodromic.

For any dg-algebra $A$ endowed with an action of an algebraic group $H$, we denote by $\DGM^H(A)$ the derived category of $H$-equivariant $A$-dg-modules. Recall that a \emph{dgg-algebra} is a bigraded algebra endowed with a differential of bidegree $(1,0)$, which satisfies the usual Leibniz rule with respect to the first grading. (Here, ``dgg" stands for ``differential graded graded.") Similarly, a \emph{dgg-module} over a dgg-algebra is a dg-module over the underlying dg-algebra endowed with a compatible additional $\Z$-grading. The first grading will be called ``cohomological," and the second one ``internal." If $A$ is a dgg-algebra, endowed with a compatible $H$-equivariant structure, we denote by $\DGM^{H \times \Gm}(A)$ the derived category of $H$-equivariant $A$-dgg-modules. We denote by $\lan n \ran$ the ``internal shift'' defined by the formula
\[
(M \lan n \ran)^i_j \ = \ M^i_{j-n},
\]
where subscripts indicate the internal grading, and superscripts indicate the cohomological grading.

For any algebraic group $H$, we denote by $Z(H)$ the center of $H$.

\subsection{Acknowledgements}

The first author is grateful to the Universit\'e Clermont-Fer\-rand II for its hospitality during a visit in June 2010, when much of the work in this paper was carried out.  This visit was supported by the CNRS and the ANR.  In addition, P.A. received support from NSA Grant No.~H98230-09-1-0024 and NSF Grant No.~DMS-1001594, and S.R. is supported by ANR Grant No.~ANR-09-JCJC-0102-01.

\section{Notation and statement of the main results}
\label{sect:statements}

\subsection{Reminder on the Satake equivalence}
\label{ss:reminder-satake}

Let $\Gv$ be a complex connected reductive algebraic group. Let $\fO:=\C[[x]]$ be the ring of formal power series in an indeterminate $x$, and let $\fK:=\C((x))$ be its quotient field. We will be interested in the affine Grassmannian \[ \Gr_{\Gv} := \Gv(\fK)/\GvO. \] This space has a natural structure of ind-variety, see \cite{g2, mv, bd, ga, np}, and it is endowed with an action of the group-scheme $\GvO$. We consider the \emph{reduced} ind-scheme structure on $\Gr_{\Gv}$. Consider the category \[ \Perv_{\eq{\GvO}}(\Gr_{\Gv}) \] of $\GvO$-equivariant perverse sheaves on $\Gr_{\Gv}$, with coefficients in $\C$. As usual, an object of this category is understood to be an equivariant perverse sheaf on a closed finite union of $\GvO$-orbits (which is a finite-dimensional algebraic variety).

It is a well-known fundamental result (see \cite{g2, mv}) that this category is endowed with a natural convolution product $\star$, which makes it a (rigid) tensor category, and that it is also equipped with a fiber functor
\[
H^{\bullet}(-) \, := \, H^{\bullet}(\Gr_{\Gv},-) \, : (\Perv_{\eq{\GvO}}(\Gr_{\Gv}),\star) \to (\Vect_{\C},\otimes).
\]
(Here, $\Vect_{\C}$ is the category of finite-dimensional $\C$-vector spaces.)

Fix a maximal torus $\Tv \subset \Gv$, and a Borel subgroup $\Bv \subset \Gv$ containing $\Tv$. The constructions of \cite[Sections 11--12]{mv} provide a canonical complex connected reductive algebraic group $G$ which is dual to $\Gv$ in the sense of Langlands, and a canonical equivalence $\cS_G$ of tensor categories which makes the following diagram commutative:
\[
\xymatrix{
(\Rep(G), \otimes) \ar[rr]^-{\cS_G} \ar[rd]_-{\For} & & (\Perv_{\eq{\GvO}}(\Gr_{\Gv}),\star) \ar[ld]^-{H^{\bullet}(-)} \\
& (\Vect_{\C},\otimes) &
}
\]
where $\Rep(G)$ is the category of finite-dimensional $G$-modules (endowed with the natural tensor product), and $\For$ is the natural fiber functor (which forgets the action of $G$).

Let $\Uv$ be the unipotent radical of $\Bv$. Let $\bX$ be the cocharacter lattice of $\Tv$. Then there is an inclusion
\[
\bX = \Gr_{\Tv} \hookrightarrow \Gr_{\Gv}.
\]
We denote by $L_{\la}$ the image of $\la \in \bX$. For $\la \in \bX$, we denote by $\fS_{\la}$ the $\Uv(\fK)$-orbit of $L_{\la}$. Let also $\bXv$ be the character lattice of $\Tv$, and consider the (complex) torus $T:=\Hom_{\Z}(\bXv,\C^{\times})$, so that we have a canonical isomorphism $X^*(T) \cong \bX$. (In other words, $T$ is dual to $\Tv$ in the sense of Langlands.) We have a tautological equivalence of tensor categories
\[
\cS_T : \, (\Rep(T),\otimes) \to (\Perv_{\eq{\TvO}}(\Gr_{\Tv}),\star).
\]

Let $\Rv \subset \bXv$ be the root system of $\Gv$. The choice of $\Bv \subset \Gv$ determines a system of positive roots $\Rv^+$ in $\Rv$ (chosen as the roots of $\mathrm{Lie}(\Bv)$). Let $\rhov$ be the half sum of positive roots. By \cite[Proposition 6.4]{mv}, the functor
\[
\fR^G_T \, := \, \bigoplus_{\la \in \bX} H^{\lan \lambda,2\rhov \ran}_c(\fS_{\la},-) : \Perv_{\eq{\GvO}}(\Gr_{\Gv}) \to \Perv_{\eq{\TvO}}(\Gr_{\Tv})
\]
is a tensor functor. And, by \cite[Theorem 3.6]{mv}, there is a natural isomorphism of tensor functors which makes the following diagram commutative:
\[
\xymatrix{
(\Perv_{\eq{\GvO}}(\Gr_{\Gv}),\star) \ar[rr]^-{\fR^G_T} \ar[rd]_-{H^{\bullet}(-)} & & (\Perv_{\eq{\TvO}}(\Gr_{\Tv}),\star) \ar[ld]^-{H^{\bullet}(-)} \\
& (\Vect_{\C},\otimes) &
}
\]
By Tannakian formalism (\cite[Corollary 2.9]{dm}), one obtains a morphism of algebraic groups $T \to G$. By \cite[Section 7]{mv}, this morphism is injective, and identifies $T$ with a maximal torus of $G$.

Let $R \subset \bX$ the root system of $G$ (which is canonically the dual of $\Rv$), and let $R^+$ be the system of positive roots in $R$ corresponding to the choice of $\Rv^+$ in $\Rv$. This choice determines a canonical Borel subgroup $B \subset G$ containing $T$. Let $\Delta$ be the basis of $R$ associated to the choice of $R^+$, and let $\bXp \subset \bX$ be the set of dominant weights of $T$.

Let $\Iv := (\mathbf{ev}_0)^{-1} (\Bv)$ be the Iwahori subgroup of $\GvO$ determined by $\Bv$, where $\mathbf{ev}_0 : \GvO \to \Gv$ is the evaluation at $x=0$. Then $\{L_{\la}, \la \in \bX\}$ is a set of representatives of $\Iv$-orbits on $\Gr_{\Gv}$. Similarly, $\{L_{\la}, \la \in \bXp\}$ is a set a representatives of $\GvO$-orbits. We denote by $\Gr_{\Gv}^{\la}$ the orbit of $\la \in \bXp$, and by $\IC_{\la}$ the associated simple perverse sheaf, an object of $\Perv_{\eq{\GvO}}(\Gr_{\Gv})$. For $\la \in \bXv$ we define
\[
V_{\lambda} \ := \ (\cS_G)^{-1}(\IC_{\lambda}).
\]
Then $V_{\lambda}$ is a simple $G$-module with highest weight $\lambda$ (see \cite[Proposition 13.1]{mv}).

For any $\lambda \in \bX$, we denote by $i_{\lambda} : \{L_{\lambda}\} \hookrightarrow \Gr_{\Gv}$ the inclusion.

\subsection{Equivalence}
\label{ss:equivalence}

Recall that the forgetful functor
\[
\Perv_{\eq{\GvO}}(\Gr_{\Gv}) \to \Perv_{\mon{\GvO}}(\Gr_{\Gv})
\]
from the category of \emph{$\GvO$-equivariant} perverse sheaves on $\Gr_{\Gv}$ to that of \emph{$\GvO$-monodromic} perverse sheaves is an equivalence of categories. (In our case this follows easily from the semisimplicity of the category $\Perv_{\mon{\GvO}}(\Gr_{\Gv})$ proved e.g.~in \cite[Lemma 7.1]{mv}.) Hence the abelian category $\Perv_{\eq{\GvO}}(\Gr_{\Gv})$ is naturally the heart of a $t$-structure on the derived category $\cDb_{\mon{\GvO}}(\Gr_{\Gv})$ of $\GvO$-monodromic sheaves on $\Gr_{\Gv}$. Our first result is a description of this triangulated category.

Recall that there exists a right action of the tensor category $\Perv_{\eq{\GvO}}(\Gr_{\Gv})$ on the category $\cDb_{\mon{\GvO}}(\Gr_{\Gv})$, which extends the convolution product. We denote it by 
\[
\left\{ \begin{array}{ccc} \cDb_{\mon{\GvO}}(\Gr_{\Gv}) \times \Perv_{\eq{\GvO}}(\Gr_{\Gv}) & \to & \cDb_{\mon{\GvO}}(\Gr_{\Gv}) \\ (M,P) & \mapsto & M \star P \end{array} \right. .
\]

On the other hand, consider the nilpotent cone $\cN_G \subset \fg$ of the Lie algebra $\fg$ of $G$. It is endowed with an action of $G \times \Gm$, defined by the formula 
\[
(g,t) \cdot X := t^{-2} \mathrm{Ad}_g(X), \quad \text{for } (g,t) \in G \times \C^{\times}, \ X \in \cN_G.
\]
Hence, the algebra $\C[\cN_G]$ is a graded $G$-equivariant algebra, concentrated in even degrees. We use this grading to consider it as a $G$-equivariant dg-algebra, endowed with the trivial differential. We denote by $\DGCGf(\cN_G)$ the triangulated subcategory of the derived category of $G$-equivariant dg-modules over this $G$-equivariant dg-algebra which is generated by the ``free'' objects, i.e.~the objects of the form $V \otimes_{\C} \C[\cN_G]$, where $V$ is a finite dimensional $G$-module. The differential on $V \otimes_{\C} \C[\cN_G]$ is the trivial one, the module structure and the grading are the natural ones, and the $G$-action is diagonal. The tensor product with $G$-modules induces a right action of the tensor category $\Rep(G)$ on the category $\DGCGf(\cN_G)$, which we denote simply by 
\[
\left\{ \begin{array}{ccc} \DGCGf(\cN_G) \times \Rep(G) & \to & \DGCGf(\cN_G) \\ (M,V) & \mapsto & M \otimes V \end{array} \right. .
\]

\begin{thm} \label{thm:equivalence}

There exists an equivalence of triangulated categories 
\[ 
F_G : \cDb_{\mon{\GvO}}(\Gr_{\Gv}) \ \xrightarrow{\sim} \ \DGCGf(\cN_G)
\]
and a natural bifunctorial isomorphism
\begin{equation}
\label{eqn:compatibility-F-S}
F_G( M \star \cS_G(V)) \ \cong \ F_G(M) \otimes V
\end{equation}
for $M$ in $\cDb_{\mon{\GvO}}(\Gr_{\Gv})$ and $V$ in $\Rep(G)$.

\end{thm}

This equivalence is proved at the level of morphisms between objects of the form $\cS_G(V)[n]$ (in the case $G$ is semisimple) in \cite[Proposition 1.10.4]{g2}. It is also suggested in \cite[Sections 7 and 10]{abg} (in the case $G$ is semisimple and adjoint), though it is not explicitly stated. In fact, under this assumption one can deduce Theorem \ref{thm:equivalence} from \cite[Theorem 9.1.4]{abg}, see \cite[\S 11.2]{ar}.

In this paper we give two direct proofs of this result. In Section \ref{sect:proof-1}, we provide a rather explicit construction of the functor $F^G$, and prove that it is an equivalence in \S \ref{ss:FG-equivalence}. These arguments are a simplified version of those of \cite[Part II]{abg}. In Section \ref{sect:relation-ABG} we prove that the equivalence constructed this way is isomorphic to the one which can be deduced from \cite[Theorem 9.1.4]{abg} (in case $G$ is semisimple and adjoint).

Then, in Section \ref{sect:proof-2} we give a second (shorter) proof of this equivalence, inspired by the arguments of the proof of the main result of \cite{bf}. This second proof does not provide an explicit description of $F^G$. It is based on the notion of \emph{enhanced triangulated category} (see \cite{bk, dr, bll}).

\subsection{Equivalence: mixed version}
\label{ss:equivalence-mix}

As in \cite{abg}, the equivalence of Theorem \ref{thm:equivalence} comes together with a ``mixed version.'' To explain this, we first have to consider some generalities. 

On several occasions in this paper we will use the following easy lemma on triangulated categories. (This lemma is stated e.g.~in \cite[Lemma 3.9.3]{abg}.)

\begin{lem}
\label{lem:triangulated-categories}

Let $\scD, \scD'$ be triangulated categories, and let $F: \scD \to \scD'$ be a triangulated functor. Assume given a set $S$ of objects of $\scD$ such that

\begin{enumerate}
\item $S$, respectively $F(S)$, generates $\scD$, respectively $\scD'$, as a triangulated category;
\item for any $M,N$ in $S$ and $i \in \Z$ the functor $F$ induces an isomorphism
\[
\Hom_{\scD}(M,N[i]) \ \xrightarrow{\sim} \ \Hom_{\scD'}(F(M),F(N)[i]).
\]
\end{enumerate}
Then $F$ is an equivalence of categories.\qed

\end{lem}

Let $X_{\Z}$ be a scheme over $\Z$, endowed with a finite (algebraic) stratification $\scS_{\Z}=\{X_{\Z,s}\}_{s \in S}$ by affine spaces. One can consider the version $X_{\C}:=X \times_{\mathrm{Spec}(\Z)} \mathrm{Spec}(\C)$ of $X_{\Z}$ over $\C$, endowed with the stratification $\scS_{\C}=\{X_{\C,s}\}_{s \in S}$, and the version $X_{\F}:=X \times_{\mathrm{Spec}(\Z)} \mathrm{Spec}(\F)$ of $X_{\Z}$ over $\F:=\overline{\F_p}$, endowed with the stratification $\scS_{\F}=\{X_{\F,s}\}_{s \in S}$. We assume that $\scS_{\C}$ is a Whitney stratification, and that the stratification $\scS_{\F}$ satisfies the usual condition \cite[Equation (6.1)]{ar}.

Consider the subcategory
\[
\cDb_{\scS_{\C}}(X_{\C}), \quad \text{respectively} \quad \cDb_{\scS_{\F}}(X_{\F})
\]
of the derived category of sheaves on $X_{\C}$ (for the complex topology), respectively of $\Qlb$-sheaves on $X_{\F}$ (for the {\'e}tale topology), consisting of objects whose cohomology sheaves are constructible with respect to the stratification $\scS_{\C}$, respectively $\scS_{\F}$. (Here, ``constructible'' amounts to requiring that the cohomology sheaves of our complexes are constant on each stratum.) Consider also the abelian subcategory
\[
\Perv_{\scS_{\C}}(X_{\C}), \quad \text{respectively} \quad \Perv_{\scS_{\F}}(X_{\F})
\]
of perverse sheaves. The following result is well known, and is used e.g.~in \cite{abg}. We include a proof for completeness.

\begin{lem} \label{lem:change-of-field}

There exists an equivalence of triangulated, respectively abelian, categories
\[
\cDb_{\scS_{\C}}(X_{\C}) \ \cong \ \cDb_{\scS_{\F}}(X_{\F}) \quad \text{respectively} \quad \Perv_{\scS_{\C}}(X_{\C}) \ \cong \ \Perv_{\scS_{\F}}(X_{\F}).
\]

\end{lem}

\begin{proof} It is enough to construct the first equivalence in such a way that it is $t$-exact. Then, restricting to the hearts gives the second equivalence.

First, by general arguments (see \cite[\S 6.1]{bbd} or \cite[Proposition 5]{bf}), one can replace the category $\cDb_{\scS_{\C}}(X_{\C})$ by the analogous category $\cDb_{\scS_{\C}}(X_{\C,\mathrm{et}})$ where the complex topology is replaced by the {\'e}tale topology, and the coefficients are $\Qlb$ instead of $\C$. Then, choose a strictly henselian discrete valuation ring $R \subset \C$ whose residue field is $\F$, and consider the corresponding constructible category $\cDb_{\scS_{R}}(X_{R})$. Then there are natural functors
\[
\cDb_{\scS_{\C}}(X_{\C,\mathrm{et}}) \ \leftarrow \ \cDb_{\scS_{R}}(X_{R}) \ \to \ \cDb_{\scS_{\F}}(X_{\F})
\]
(see \cite[\S 6.1.8]{bbd}).

We claim that these functors are equivalences. Indeed, one can consider the standard and costandard objects $\Delta_s:=(j_s)_! \uQlb{}_{X_s} [\dim X_s]$ and $\nabla_s:=(j_s)_* \uQlb{}_{X_s}[\dim X_s]$ in all three of these categories (where $j_s$ is the inclusion of the stratum labelled by $s$). These families of objects each generate the categories under consideration (as triangulated categories). Moreover, it is easy to check that in all these categories we have
\[
\Hom(\Delta_s,\nabla_t[n]) \ = \ \Qlb^{\delta_{s,t} \delta_{n,0}}
\]
(see \cite[VI.4.20]{mi}). The result follows by Lemma \ref{lem:triangulated-categories}.\end{proof}

The group $\Gv$ and its subgroups $\Tv$, $\Bv$ can be defined over $\Z$. Hence the ind-scheme $\Gr_{\Gv}$ together with its stratification by $\Iv$-orbits has a version over $\Z$, and we are in the situation of Lemma \ref{lem:change-of-field}. We obtain an equivalence of abelian categories 
\begin{equation} \label{eqn:equiv-Gr-change-of-field}
\Perv_{\mon{\Iv}}(\Gr_{\Gv}) \ \cong \ \Perv_{\mon{{\Iv}}}(\Gr_{\Gv,\F}).
\end{equation}
(On the right-hand side, we have simplified the notation: ``$\mon{\Iv}$'' means monodromic for the action of the version of $\Iv$ over $\F$.)

The category $\Perv_{\mon{\Iv}}(\Gr_{\Gv,\F})$ has a natural mixed version $\Perv_{\mon{\Iv}}^{\mix}(\Gr_{\Gv})$ in the sense of \cite{bgs} or \cite{ar}, constructed as follows. Consider the versions $\Gr_{\Gv,\F_p}$, $\Iv_{\F_p}$ of $\Gr_{\Gv}$, $\Iv$ over the finite field $\F_p$. With the notation of \cite[\S 6.1]{ar}, one can consider the Serre subcategory
\[
\Pervw_{\mon{\Iv}}(\Gr_{\Gv,\F_p})
\]
of the category of $\Qlb$-perverse sheaves on $\Gr_{\Gv,\F_p}$ which is generated by the simple objects $\IC(Y,\Qlb) \lan j \ran$ where $j \in \Z$ and $Y$ is an $\Iv_{\F_p}$-orbit on $\Gr_{\Gv,\F_p}$. By \cite[Th{\'e}or{\`e}me 5.3.5]{bbd}, every object $P$ of this category is equipped with a canonical \emph{weight filtration}, denoted $W_{\bullet} P$. Then we define 
\[
\Pervm_{\mon{\Iv}}(\Gr_{\Gv}) := \{ P \in \Pervw_{\mon{\Iv}}(\Gr_{\Gv,\F_p}) \mid \text{for all } i \in \Z, \ \mathrm{gr}^W_i P \text{ is semisimple} \}.
\]
This is an abelian subcategory of $\Pervw_{\mon{\Iv}}(\Gr_{\Gv,\F_p})$. By \cite[Theorem 4.4.4]{bgs}, the derived category $\cDb \Pervm_{\mon{\Iv}}(\Gr_{\Gv})$ is a mixed version\footnote{Equivalently, in the terminology of \cite[Definition 4.3.1]{bgs}, $\Pervm_{\mon{\Iv}}(\Gr_{\Gv})$ is a grading on $\Perv_{\mon{{\Iv}}}(\Gr_{\Gv,\F})$.} of $\cDb \Perv_{\mon{{\Iv}}}(\Gr_{\Gv,\F})$ in the sense of \cite[\S 2.3]{ar}, for the shift functor $\lan 1 \ran$. (In particular, the abelian category $\Pervm_{\mon{\Iv}}(\Gr_{\Gv})$ is a mixed version of $\Perv_{\mon{{\Iv}}}(\Gr_{\Gv,\F})$ in the sense of \cite[\S 2.3]{ar}, but this notion is weaker.) As the notation suggests, we will mainly forget about the field $\F$ and, using equivalence \eqref{eqn:equiv-Gr-change-of-field}, we will consider $\Pervm_{\mon{\Iv}}(\Gr_{\Gv})$ as a mixed version of $\Perv_{\mon{\Iv}}(\Gr_{\Gv})$. In particular, we have a forgetful functor
\[
\For : \Pervm_{\mon{\Iv}}(\Gr_{\Gv}) \to \Perv_{\mon{\Iv}}(\Gr_{\Gv}).
\]

By \cite[Corollary 3.3.2]{bgs}, the realization functor 
\begin{equation}
\label{eqn:realization-functor-equivalence}
\cDb \Perv_{\mon{\Iv}}(\Gr_{\Gv}) \ \to \ \cDb_{\mon{\Iv}}(\Gr_{\Gv})
\end{equation}
is an equivalence of categories. We define the category
\[
\dm_{\mon{\Iv}}(\Gr_{\Gv}) \ := \ \cDb \Pervm_{\mon{\Iv}}(\Gr_{\Gv}).
\]
By the remarks above, this triangulated category is a mixed version of $\cDb_{\mon{\Iv}}(\Gr_{\Gv})$ in the sense of \cite[\S 2.3]{ar}. (Note that this definition is indeed a particular case of the general definition given in \cite[Equation (7.1)]{ar} by \cite[Corollary 7.10]{ar}.)

In this paper we are mainly not interested in the stratification by $\Iv$-orbits, but rather in the stratification by $\GvO$-orbits. We denote by 
\[
\dm_{\mon{\GvO}}(\Gr_{\Gv})
\]
the triangulated subcategory of $\dm_{\mon{\Iv}}(\Gr_{\Gv})$ generated by the simple objects associated with the $\Gv(\F_p[[x]])$-orbits $\Gr^{\lambda}_{\Gv,\F_p}$, $\lambda \in \bX^+$ (and their shifts). By construction, this triangulated category is a mixed version of the category $\cDb_{\mon{\GvO}}(\Gr_{\Gv})$. In particular, there is a forgetful functor
\begin{equation}
\label{eqn:for-const}
\For : \dm_{\mon{\GvO}}(\Gr_{\Gv}) \to \cDb_{\mon{\GvO}}(\Gr_{\Gv}).
\end{equation}

Recall the dg-algebra $\C[\cN_G]$ considered in \S \ref{ss:equivalence}. Now we consider this algebra as a dgg-algebra. Here, the differential on $\C[\cN_G]$ is again trivial, and the bigrading is chosen so that the natural generators of this algebra are in bidegree $(2,2)$. We denote by $\DGCGmf(\cN_G)$ the subcategory of the derived category $\DGM^{G \times \Gm}(\C[\cN_G])$ generated by the objects of the form $V \otimes_{\C} \C[\cN_G] \lan i \ran$ for $V$ a finite dimensional $G$-module. This triangulated category is clearly a graded version of the category $\DGCGf(\cN_G)$ in the sense of \cite[\S 2.3]{ar}.

As in \cite[\S 9.6]{abg}, one can play the ``regrading trick'': the functor which sends the $\Z^2$-graded vector space $M=\oplus_{(i,j) \in \Z^2} M^i_j$ to the $\Z^2$-graded vector space $N=\oplus_{(i,j) \in \Z^2} N^i_j$ defined by
\[
N^i_j := M^{i+j}_j
\]
induces an equivalence of triangulated categories
\begin{equation}
\label{eqn:equivalence-regrading}
\DGCGmf(\cN_G) \ \xrightarrow{\sim} \ \cDb_{\free}\Coh^{G \times \Gm}(\cN_G),
\end{equation}
where $\cDb_{\free}\Coh^{G \times \Gm}(\cN_G)$ is the subcategory of the bounded derived category of $G \times \Gm$-equivariant coherent sheaves on $\cN_G$ (with respect to the $G \times \Gm$-action defined in \S \ref{ss:equivalence}) generated by the ``free'' objects of the form $V \otimes_{\C} \cO_{\cN_G}$ for $V$ a $G \times \Gm$-module. Hence, the category $\cDb_{\free}\Coh^{G \times \Gm}(\cN_G)$ is also a graded version of the category $\DGCGf(\cN_G)$. In particular, there is a forgetful functor
\begin{equation}
\label{eqn:for-coh}
\For : \cDb_{\free}\Coh^{G \times \Gm}(\cN_G) \to \DGCGf(\cN_G).
\end{equation}

The ``mixed version'' of Theorem \ref{thm:equivalence} is the following result.

\begin{thm} 
\label{thm:equivalence-mix}

There exists an equivalence of triangulated categories
\[
F_G^{\mix} : \dm_{\mon{\GvO}}(\Gr_{\Gv}) \ \xrightarrow{\sim} \ \cDb_{\free}\Coh^{G \times \Gm}(\cN_G)
\]
such that the following diagram commutes:
\[
\xymatrix@C=2cm{
\dm_{\mon{\GvO}}(\Gr_{\Gv}) \ar[d]_-{\For}^-{\eqref{eqn:for-const}} \ar[r]_-{\sim}^-{F_G^{\mix}} & \cDb_{\free}\Coh^{G \times \Gm}(\cN_G) \ar[d]^-{\For}_-{\eqref{eqn:for-coh}} \\
\cDb_{\mon{\GvO}}(\Gr_{\Gv}) \ar[r]_-{\sim}^-{F_G} & \DGCGf(\cN_G).
}
\]
This equivalence satisfies:
\[
F_G^{\mix}(M \lan n \ran) \ \cong \ F_G^{\mix}(M) \lan n \ran [n].
\]

\end{thm}

Again, in the case $G$ is semisimple and adjoint, this result can be deduced from \cite[Theorem 9.4.3]{abg}. We give two proofs of this theorem in \S \ref{ss:mixed-version} and \S \ref{ss:mixed-version-enhanced}, which are parallel to those of Theorem \ref{thm:equivalence}. In \S \ref{ss:isom-FG-FGmix} we prove that the two equivalences obtained by these methods are in fact isomorphic. One can also prove that, in the case $G$ is semisimple and adjoint, this equivalence is isomorphic to the one deduced from \cite[Theorem 9.4.3]{abg}.

Note that, in contrast to Theorem \ref{thm:equivalence}, there is no ``compatibility'' statement with respect to convolution of perverse sheaves in Theorem \ref{thm:equivalence-mix}. We will address this problem in \S \ref{ss:Satake-mix} below.

\subsection{Hyperbolic localization and restriction to a Levi subgroup}
\label{ss:hl-restriction}

Next we study functoriality properties of Theorems \ref{thm:equivalence} and \ref{thm:equivalence-mix}. Consider a standard parabolic subgroup $\Pv \subset \Gv$, and its Levi factor $\Lv \subset \Pv$ containing $\Tv$. This data is entirely determined by the choice of a subset of $\Delta$, hence it determines a Levi subgroup $L \subset G$ containing $T$.

On the coherent side of the picture, one can consider the inclusion 
\[
i_L^G : \cN_L \hookrightarrow \cN_G
\]
and the associated (derived) pull-back functors
\begin{align*}
(i_L^G)^* : \DGCGf(\cN_G) & \to \DGCGf(\cN_L), \\ (i_L^G)^*_{\mix} : \cDb_{\free} \Coh^{G \times \Gm}(\cN_G) & \to \cDb_{\free} \Coh^{G \times \Gm}(\cN_L).
\end{align*}

Now, consider the constructible side of the picture. Let $\Pv^-$ be the parabolic subgroup of $\Gv$ opposite to $\Pv$ (relative to $\Tv$). The inclusions $\Pv \hookrightarrow \Gv$, $\Pv^- \hookrightarrow \Gv$ and the projections $\Pv \twoheadrightarrow \Lv$, $\Pv^- \twoheadrightarrow \Lv$ induce morphisms of ind-schemes
\begin{align*}
i : \Gr_{\Pv} \to \Gr_{\Gv}, & \quad j : \Gr_{\Pv^-} \to \Gr_{\Gv}, \\
p : \Gr_{\Pv} \to \Gr_{\Lv}, & \quad q : \Gr_{\Pv^-} \to \Gr_{\Lv}.
\end{align*}
Then, one has functors
\begin{align*}
p_! i^* : \cDb_{\mon{\GvO}}(\Gr_{\Gv}) & \to \cDb_{\mon{\LvO}}(\Gr_{\Lv}), \\
q_* j^! : \cDb_{\mon{\GvO}}(\Gr_{\Gv}) & \to \cDb_{\mon{\LvO}}(\Gr_{\Lv}).
\end{align*}

Let $\lambda_L : \C^{\times} \to \Tv$ be a generic dominant coweight of $\Tv$ with values in the center of $\Lv$. (For example, one may take $\lambda_L=2\rho_G - 2 \rho_L$, where $\rho_G$, respectively $\rho_L$, is the half sum of positive roots of $G$, respectively of $L$.) Then $\Gr_{\Lv}$ is the set of fixed points for the action of $\lambda_L(\C^{\times}) \subset \GvO$ on $\Gr_{\Gv}$. Moreover, $i$ is a locally closed embedding which identifies $\Gr_{\Pv}$ with the attracting set for this action. Similarly, $j$ identifies $\Gr_{\Pv^-}$ with the attracting set for the action of $\lambda_L^{-1}$. Hence we are in the situation of \cite[Theorem 1]{br}, which provides an isomorphism of functors
\[
p_! i^* \ \cong \ q_* j^! : \cDb_{\mon{\GvO}}(\Gr_{\Gv}) \to \cDb_{\mon{\LvO}}(\Gr_{\Lv}).
\]
This functor is called the \emph{hyperbolic localization} functor, and is denoted $h^{!*}_L$.

The connected components of $\Gr_{\Lv}$ are parametrized by the set $X^*(Z(L))$ of characters of the center of $L$. We denote by $\Gr_{\Lv,\chi}$ the connected component associated to $\chi \in X^*(Z(L))$. Any object $M$ of $\cDb_{\mon{\LvO}}(\Gr_{\Lv})$ is the direct sum of (finitely many) objects $M_{\chi}$ supported on $\Gr_{\Lv,\chi}$, $\chi \in X^*(Z(L))$. Let $\Theta_L$ be the functor which sends such an $M$ to
\[
\bigoplus_{\chi \in X^*(Z(L))} M_{\chi} [\lan \chi, 2\rho_{\Gv} - 2\rho_{\Lv} \ran].
\]
Here, $\rho_{\Gv}$ and $\rho_{\Lv}$ are the half sums of positive coroots of $G$ and $L$. Note that $2\rho_{\Gv} - 2\rho_{\Lv}$ is orthogonal to all roots of $L$, hence the pairing $\lan \chi, 2\rho_{\Gv} - 2\rho_{\Lv} \ran$ makes sense. We consider the functor
\[
\fR^G_L \, := \, \Theta_L \circ h^{!*}_L : \ \cDb_{\mon{\GvO}}(\Gr_{\Gv}) \to \cDb_{\mon{\LvO}}(\Gr_{\Lv}).
\]
(Note that this notation is consistent with that of \S \ref{ss:reminder-satake}.)

The importance of this functor is clear from the following result, which is proved in \cite[Proposition 5.3.29 and Lemma 5.3.1]{bd}. (The case $L=T$ is one of the fundamental preliminary results of \cite{mv}.)

\begin{thm}
\label{thm:hl-restriction-classical}

The functor $\fR^G_L$ sends $\Perv_{\eq{\GvO}}(\Gr_{\Gv}) \subset \cDb_{\mon{\GvO}}(\Gr_{\Gv})$ to the subcategory $\Perv_{\eq{\LvO}}(\Gr_{\Lv}) \subset \cDb_{\mon{\LvO}}(\Gr_{\Lv})$. Moreover, the following diagram commutes up to an isomorphism of functors:
\[
\xymatrix@C=2cm{
\Rep(G) \ar[r]^-{\cS_G}_-{\sim} \ar[d]_-{\mathrm{Res}^G_L} & \Perv_{\eq{\GvO}}(\Gr_{\Gv}) \ar[d]^-{\fR^G_L} \\
\Rep(L) \ar[r]^-{\cS_L}_-{\sim} & \Perv_{\eq{\LvO}}(\Gr_{\Lv}),
}
\]
where $\mathrm{Res}^G_L$ is the restriction functor.\qed

\end{thm}

\begin{rmk}
\label{rmk:R-convolution}
One of the main steps in the proof of the commutativity of the diagram of Theorem \ref{thm:hl-restriction-classical} is to show that the functor $\fR^G_L$ commutes with convolution. In Section \ref{sect:hl-convolution} we give a new proof of this result in greater generality, see \S \ref{ss:hl-convolution} for details.
\end{rmk}

Our second main result relates these two pictures. In fact, we were only able to relate mixed versions of these functors. Using general constructions of \cite{ar}, we first prove that the functor $\fR^G_L$ has a ``mixed version''
\[
\fR^{G,\mix}_L : \dm_{\mon{\GvO}}(\Gr_{\Gv}) \to \dm_{\mon{\LvO}}(\Gr_{\Lv})
\]
(see Proposition \ref{prop:mixed-version-RGL} for details). Then, using the notion of \emph{Orlov category} studied in \cite{ar} we obtain the following result, which is proved in \S \ref{ss:proof-thm-hl-restriction}.

\begin{thm}
\label{thm:hl-restriction}

The following diagram commutes up to an isomorphism of functors:
\[
\xymatrix@C=2cm{
\dm_{\mon{\GvO}}(\Gr_{\Gv}) \ar[r]^-{F_G^{\mix}}_-{\sim} \ar[d]_-{\fR^{G,\mix}_L} & \cDb_{\free} \Coh^{G \times \Gm}(\cN_G) \ar[d]^-{(i^G_L)^*_{\mix}} \\
\dm_{\mon{\LvO}}(\Gr_{\Lv}) \ar[r]^-{F_L^{\mix}}_-{\sim} & \cDb_{\free} \Coh^{L \times \Gm}(\cN_L).
}
\]

\end{thm}

\begin{rmk}

We expect that the non-mixed version of the diagram of Theorem \ref{thm:hl-restriction} also commutes. However, we were not able to prove this fact. The problem is the following. One can check that there exist  isomorphisms 
\begin{equation} \label{eqn:isom-objects}
(i^G_L)^* \circ F_G (\cS_G(V)) \ \cong \ F_L \circ \fR^G_L(\cS_G(V))
\end{equation}
for all $V$ in $\Rep(G)$, which are well-behaved with respect to morphisms (see Proposition \ref{prop:action-functors-morphisms}). The objects $\cS_G(V)$ generate the triangulated category $\cDb_{\mon{\GvO}}(\Gr_{\Gv})$. However, we were not able to construct a morphism of \emph{functors} which would induce isomorphisms \eqref{eqn:isom-objects} on objects. In the mixed setting, we use the extra structure given by the grading, and general constructions from homological algebra, to construct such a morphism of functors.

\end{rmk}

In \S\S \ref{ss:Satake-mix}--\ref{ss:hl-convolution}, we present results that are not essential but which enlighten some of the interesting properties of the functors $F_G$ and $\fR^G_L$ and their mixed versions.

\subsection{Satake equivalence and mixed perverse sheaves}
\label{ss:Satake-mix}

Consider the categories 
\[ 
\Perv_{\eq{\GvO}}(\Gr_{\Gv,\F_p}), \quad \text{respectively} \quad \Perv_{\mon{\GvO}}(\Gr_{\Gv,\F_p})
\] of $\Gv(\F_p[[x]])$-equivariant, respectively $\Gv(\F_p[[x]])$-monodromic, $\Qlb$-perverse sheaves on the $\F_p$-version of $\Gr_{\Gv}$. The forgetful functor
\[
\Perv_{\eq{\GvO}}(\Gr_{\Gv,\F_p}) \to \Perv_{\mon{\GvO}}(\Gr_{\Gv,\F_p})
\]
is fully faithful. Indeed, the corresponding fact over $\overline{\F_p}$ is well known (see \cite[Proposition A.1]{mv}). And it is also well known (\cite[Proposition 5.1.2]{bbd}) that the categories of perverse sheaves over $\F_p$ can be described as categories of perverse sheaves over $\overline{\F_p}$ endowed with a Weil structure. This implies our claim. (For this, see also \cite[Proposition 1 and its proof]{ga}.) 

For any $\lambda \in \bX^+$, we denote by $\IC_{\lambda}^{\mix}$ the simple perverse sheaf associated to the constant local system on $\Gr^{\lambda}_{\Gv,\F_p}$, normalized so that it has weight $0$. We let
\[ 
\Pervw_{\mon{\GvO}}(\Gr_{\Gv,\F_p}) 
\]
be the Serre subcategory of $\Perv_{\eq{\GvO}}(\Gr_{\Gv,\F_p})$, or equivalently $\Perv_{\mon{\GvO}}(\Gr_{\Gv,\F_p})$, generated by the simple objects $\IC_{\lambda}^{\mix} \lan j \ran$, where $\lambda \in \bX^+$ and $j \in \Z$.

It is well known that the category $\Perv_{\eq{\GvO}}(\Gr_{\Gv,\F_p})$ can be endowed with a convolution product $(M,N) \mapsto M \star N$, which is associative and commutative (see \cite[\S 1.1.2]{ga}). The following result is well known, but we have not found any explicit proof in this setting in the literature. It can be easily deduced from \cite[Propositions 9.4 and 9.6]{np}. We give a different proof to illustrate the techniques of \cite{ar}.

\begin{prop} 
\label{prop:convolution-mix}

For any $\lambda,\mu \in \bX^+$, the convolution $\IC_{\lambda}^{\mix} \star \IC_{\mu}^{\mix}$ is a direct sum of simple perverse sheaves $\IC_{\nu}^{\mix}$, $\nu \in \bX^+$.

\end{prop}

\begin{proof}
Let $\Fl_{\Gv}:=\Gv(\fK)/\Iv$ be the affine flag variety. It is naturally endowed with the structure of an ind-scheme, and we have a natural smooth and proper morphism $p : \Fl_{\Gv} \to \Gr_{\Gv}$. One can define the category $\cD^{\Weil}_{\mon{\Iv}}(\Fl_{\Gv})$ as for $\Gr_{\Gv}$, see \cite[\S 6.1]{ar}. We already know that $\IC_{\lambda}^{\mix} \star \IC_{\mu}^{\mix}$ is a $\GvO$-equivariant perverse sheaf. Hence it is enough to prove that $p^*(\IC_{\lambda}^{\mix} \star \IC_{\mu}^{\mix})$ is a semisimple object of $\cD^{\Weil}_{\mon{\Iv}}(\Fl_{\Gv})$.

Let us denote by $(- \star^{\Iv} -)$ the convolution of $\Iv$-equivariant complexes on $\Fl_{\Gv}$, or the action of $\Iv$-equivariant complexes on $\Fl_{\Gv}$ on $\Iv$-equivariant complexes on $\Gr_{\Gv}$. Then we have
\begin{align*}
p^*(\IC_{\lambda}^{\mix}) \star^{\Iv} p^*(\IC_{\mu}^{\mix}) \ & \cong \ p^* \bigl( p^*(\IC_{\lambda}^{\mix}) \star^{\Iv} \ICm_{\mu} \bigr) \\
& \cong \ p^* \bigl( p_* (p^*(\IC_{\lambda}^{\mix})) \star \ICm_{\mu} \bigr) \\
& \cong \ p^* \bigl( \IC_{\lambda}^{\mix} \star \IC_{\mu}^{\mix} \otimes H^{\bullet}(\Gv_{\F_p}/\Bv_{\F_p}) \bigr).
\end{align*}
The cohomology $H^{\bullet}(\Gv_{\F_p}/\Bv_{\F_p})$ has a semisimple action of Frobenius. Hence to prove the proposition it is enough to prove that $p^*(\IC_{\lambda}^{\mix}) \star^{\Iv} p^*(\IC_{\mu}^{\mix})$ is a semisimple object of $\cD^{\Weil}_{\mon{\Iv}}(\Fl_{\Gv})$. However, this follows from \cite[Proposition 3.2.5]{by}, see also \cite[Remark 12.3]{ar}.
\end{proof}

It follows from Proposition \ref{prop:convolution-mix} that the subcategory $\Pervw_{\mon{\GvO}}(\Gr_{\Gv,\F_p})$ is stable under the convolution product. Let $\Perv^0_{\mon{\GvO}}(\Gr_{\Gv})$ be the subcategory of $\Pervw_{\mon{\GvO}}(\Gr_{\Gv,\F_p})$ whose objects are direct sums of objects $\IC_{\lambda}^{\mix}$, $\lambda \in \bX^+$. Again by Proposition \ref{prop:convolution-mix}, this subcategory is stable under the convolution product.

Equivalence \eqref{eqn:equiv-Gr-change-of-field} induces a similar equivalence where ``$\mon{\Iv}$'' is replaced by ``$\mon{\GvO}$.'' In particular, it follows that extension of scalars defines a functor
\[
\Phi_G : \Perv^0_{\mon{\GvO}}(\Gr_{\Gv}) \to \Perv_{\mon{\GvO}}(\Gr_{\Gv}),
\]
which commutes with convolution products.

\begin{lem}
\label{lem:Phi_G-equivalence}

The functor $\Phi_G$ is an equivalence of tensor categories.

\end{lem}

\begin{proof}
This functor induces a bijection on isomorphism classes of objects. Hence it is enough to prove that it is fully faithful. As the category $\Perv^0_{\mon{\GvO}}(\Gr_{\Gv})$ is semisimple by definition, it is enough to prove that for any $\lambda,\mu \in \bX^+$ the functor $\Phi_G$ induces an isomorphism
\[
\Hom_{\Perv^0_{\mon{\GvO}}(\Gr_{\Gv})}(\IC^{\mix}_{\lambda},\IC^{\mix}_{\mu}) \xrightarrow{\sim} \Hom_{\Perv_{\mon{\GvO}}(\Gr_{\Gv})}(\IC_{\lambda},\IC_{\mu}).
\]
However, this fact is obvious since both spaces are isomorphic to $\C^{\delta_{\lambda,\mu}}$.
\end{proof}

Using Lemma \ref{lem:Phi_G-equivalence}, we obtain an equivalence of tensor categories
\begin{equation}
\label{eqn:satake-0}
\cS^0_G : \Rep(G) \ \xrightarrow{\sim} \ \Perv^0_{\mon{\GvO}}(\Gr_{\Gv}).
\end{equation}
Using again the general constructions of \cite{ar} we construct for any object $M$ in $\Perv^0_{\mon{\GvO}}(\Gr_{\Gv})$ a functor
\[
(-) \star M : \cD^{\mix}_{\mon{\GvO}}(\Gr_{\Gv}) \to \cD^{\mix}_{\mon{\GvO}}(\Gr_{\Gv})
\]
which is a mixed version of the functor $(-) \star \Phi_G(M)$, see Proposition \ref{prop:mixed-version-convolution}. Then compatibility of the equivalence $F_G^{\mix}$ with convolution can be stated as follows. The proof of this result is given in \S \ref{ss:compatibility}.

\begin{prop}
\label{prop:FGmix-convolution}

For any $V$ in $\Rep(G)$ and $M$ in $\cD^{\mix}_{\mon{\GvO}}(\Gr_{\Gv})$, there exists an isomorphism
\[
F_G^{\mix}(M \star \cS_G^0(V)) \ \cong \ F_G^{\mix}(M) \otimes V,
\]
which is functorial in $M$.

\end{prop}

In the remainder of this subsection we describe a setting in which the abelian category $\Perv^0_{\mon{\GvO}}(\Gr_{\Gv})$ appears more naturally. These results will not be used in the rest of the paper.

Let us define the category
\[
\Pervm_{\mon{\GvO}}(\Gr_{\Gv}) := \{P \in \Pervw_{\mon{\GvO}}(\Gr_{\Gv,\F_p}) \mid \text{for all } i \in \Z, \ \mathrm{gr}^W_i P \text{ is semisimple} \}.
\]
By definition, this category is a full subcategory of the abelian category $\Pervm_{\mon{\Iv}}(\Gr_{\Gv})$, which is stable under extensions. Any object of $\Pervm_{\mon{\GvO}}(\Gr_{\Gv})$ is endowed with a canonical weight filtration. By definition again, $\Perv^0_{\mon{\GvO}}(\Gr_{\Gv})$ is the subcategory of $\Pervm_{\mon{\GvO}}(\Gr_{\Gv})$ whose objects have weight $0$.

\begin{lem}
\label{lem:spherical-perv-mix}

\begin{enumerate}
\item The category $\Pervm_{\mon{\GvO}}(\Gr_{\Gv})$ is semisimple.
\item The weight filtration of any $M$ in $\Pervm_{\mon{\GvO}}(\Gr_{\Gv})$ splits canonically.
\end{enumerate}

\end{lem}

\begin{proof}
To prove (1), it is enough to prove that for any $\lambda,\mu$ in $\bX^+$ and any $j \in \Z$ we have
\[
\Ext^1_{\Pervm_{\mon{\GvO}}(\Gr_{\Gv})}(\IC_{\lambda}^{\mix}, \IC_{\mu}^{\mix}\lan j \ran)=0.
\]
However, as $\Pervm_{\mon{\GvO}}(\Gr_{\Gv})$ is stable under extensions in $\Pervm_{\mon{\Iv}}(\Gr_{\Gv})$, one can compute the $\Ext^1$-group in the latter category. And, as $\cDb \Pervm_{\mon{\Iv}}(\Gr_{\Gv})$ is a mixed version of $\cDb \Perv_{\mon{\Iv}}(\Gr_{\Gv})$, this $\Ext^1$-group is a direct factor of
\[
\Ext^1_{\Perv_{\mon{\Iv}}(\Gr_{\Gv})}(\IC_{\lambda}, \IC_{\mu}).
\]
It is well known that the latter group is trivial, see e.g.~\cite[Proof of Lemma 7.1]{mv}.

Let us consider assertion (2). The fact that the filtration splits follows from (1). This splitting is canonical because there are no non-zero morphisms between pure objects of distinct weights.\end{proof}

By Lemma \ref{lem:spherical-perv-mix} and Proposition \ref{prop:convolution-mix}, the subcategory $\Pervm_{\mon{\GvO}}(\Gr_{\Gv})$ of the category $\Pervw_{\mon{\GvO}}(\Gr_{\Gv})$ is stable under convolution. Hence $(\Pervm_{\mon{\GvO}}(\Gr_{\Gv}),\star)$ is a semisimple tensor category. Using Tannakian formalism (\cite{dm}), it is not difficult to identify this category. The proof of the following proposition is left to the reader.

\begin{prop}

There exists an equivalence of tensor categories
\[
\cS_G^{\mix} : (\Rep(G \times \Gm),\otimes) \ \xrightarrow{\sim} \ (\Pervm_{\mon{\GvO}}(\Gr_{\Gv}),\star)
\]
such that the diagram:
\[
\xymatrix@C=2cm{
\Rep(G \times \Gm) \ar[r]_-{\sim}^-{\cS_G^{\mix}} \ar[d]_-{\mathrm{Res^{G \times \Gm}_G}} & \Pervm_{\mon{\GvO}}(\Gr_{\Gv}) \ar[d]^-{\For} \\ 
\Rep(G) \ar[r]_-{\sim}^-{\cS_G} & \Perv_{\mon{\GvO}}(\Gr_{\Gv})
}
\]
commutes.\qed

\end{prop}

\subsection{Hyperbolic localization and convolution}
\label{ss:hl-convolution}

Consider, as in \S \ref{ss:hl-restriction}, a standard Levi subgroup $\Lv \subset \Gv$, and the corresponding standard Levi $L \subset G$. It follows in particular from Theorems \ref{thm:equivalence-mix} and \ref{thm:hl-restriction} and Proposition \ref{prop:FGmix-convolution} that for any $M \in \cDb_{\mon{\GvO}}(\Gr_{\Gv})$ which is in the image of the forgetful functor $\For : \dm_{\mon{\GvO}}(\Gr_{\Gv}) \to \cDb_{\mon{\GvO}}(\Gr_{\Gv})$ and $N \in \Perv_{\eq{\GvO}}(\Gr_{\Gv})$, there exists an isomorphism
\[
\fR^G_L(M \star N) \ \cong \ \fR^G_L(M) \star \fR^G_L(N).
\]
Indeed, let $M'$ be an object of $\dm_{\mon{\GvO}}(\Gr_{\Gv})$ such that $M \cong \For(M')$. Let $N':=(\Phi_G)^{-1}(N)$. Then we have a chain of isomorphisms
{\small
\begin{eqnarray*}
\fR^G_L(M \star N) & \overset{\textrm{Th. \ref{thm:hl-restriction}}}{\cong} & \For \circ (F_L^{\mix})^{-1} \circ (i^{G}_{L})^*_{\mix} \circ F_G^{\mix}(M' \star N') \\
& \overset{\textrm{Prop. \ref{prop:FGmix-convolution}}}{\cong} & \For \circ (F_L^{\mix})^{-1} \circ (i^{G}_{L})^*_{\mix} \bigl( F_G^{\mix}(M') \otimes (\cS_G^0)^{-1}(N') \bigr) \\
& \overset{(\dag)}{\cong} & \For \circ (F_L^{\mix})^{-1} \Bigl( \bigl((i^{G}_{L})^*_{\mix} F_G^{\mix}(M') \bigr) \otimes  \bigl( (\cS_L^0)^{-1} \fR^{G,\mix}_L(N') \bigr) \Bigr) \\
& \overset{\textrm{Prop. \ref{prop:FGmix-convolution}}}{\cong} & \For \Bigl( (F_L^{\mix})^{-1} \bigl((i^{G}_{L})^*_{\mix} \circ F_G^{\mix}(M') \bigr) \star \fR^{G,\mix}_L(N') \Bigr) \\
& \overset{\textrm{Th. \ref{thm:hl-restriction}}}{\cong} & \For \bigl( \fR_L^{G,\mix}(M') \star \fR_L^{G,\mix}(N') \bigr) \\
& \cong & \fR^G_L(M) \star \fR^G_L(N).
\end{eqnarray*}
}Isomorphism $(\dag)$ uses an isomorphism of functors $\fR^{G,\mix}_L \circ \cS_G^0 \ \cong \ \cS_L^0 \circ \mathrm{Res}^G_L$ which follows from Theorem \ref{thm:hl-restriction-classical} and the construction of $\fR^{G,\mix}_L$; see Remark \ref{rmk:hl-restriction-mix} for details.

In Section \ref{sect:hl-convolution} we give a (topological) proof of the following much more general claim. Recall that the right action of the tensor category $\Perv_{\eq{\GvO}}(\Gr_{\Gv})$ on $\cDb_{\mon{\GvO}}(\Gr_{\Gv})$ extends to a convolution bifunctor
\[
(- \star -) : \ \cDb_c(\Gr_{\Gv}) \times \cDb_{\eq{\GvO}}(\Gr_{\Gv}) \to  \cDb_c(\Gr_{\Gv}),
\]
where $\cDb_c(\Gr_{\Gv})$ is the bounded derived category of constructible sheaves on $\Gr_{\Gv}$. As in \S \ref{ss:hl-restriction}, let $\lambda_L : \C^{\times} \to \Tv$ be a generic dominant cocharacter with values in the center of $\Lv$. We let $\mathrm{Aut}$ denote the pro-algebraic group of automorphisms of $\fO$, see \cite[\S 2.1.2]{ga}. Then we have the following result, whose proof is given in \S \ref{ss:convolution-hl}.

\begin{prop} \label{prop:convolution-hl}

For any $M$ in $\cDb_{\mon{\lambda_L(\C^{\times})}}(\Gr_{\Gv})$ and $N$ in $\cDb_{\eq{\GvO \rtimes \mathrm{Aut}}}(\Gr_{\Gv})$, there is a bifunctorial isomorphism
\[
\fR^G_L(M \star N) \ \cong \ \fR^G_L(M) \star \fR^G_L(N).
\]

\end{prop}

Note that the forgetful functor $\Perv_{\eq{\GvO \rtimes \mathrm{Aut}}}(\Gr_{\Gv}) \to \Perv_{\eq{\GvO}}(\Gr_{\Gv})$ is an equivalence of categories, see \cite[Proposition 1]{ga} or more generally \cite[Proposition A.1]{mv}. Hence Proposition \ref{prop:convolution-hl} applies in particular if $N$ is in $\Perv_{\eq{\GvO}}(\Gr_{\Gv})$. In case we assume in addition that $M$ is a $\GvO$-monodromic perverse sheaf, this result is well known, see Remark \ref{rmk:R-convolution}. The general case may be known to experts; however we have not found any proof in this generality in the literature. Note also that our proof of Proposition \ref{prop:convolution-hl} works similarly in the case of sheaves with coefficients in any commutative ring of finite global dimension\footnote{In this case we need a more general case of Braden's theorem, which holds by \cite[Remark (3) on p.~211]{br}.}. This proof is independent of the rest of the paper.

\section{Constructible sheaves on $\Gr_{\Gv}$ and coherent sheaves on $\cN_G$:\\ first approach}
\label{sect:proof-1}

In this section we give a first proof of Theorems \ref{thm:equivalence} and \ref{thm:equivalence-mix}. Our arguments are a simplified version of the proofs of \cite[Theorems 9.1.4 and 9.4.3]{abg}.

\subsection{Reminder on cohomology of affine Grassmannians}
\label{ss:reminder-cohomology}


Let us recall, following \cite{g2} and \cite{yz}, how one can give information on the cohomology $H^{\bullet}(\Gr_{\Gv}):=H^{\bullet}(\Gr_{\Gv},\C)$. As $\Gr_{\Gv}$ is homeomorphic to the group of polynomial loops in a Lie group (see \cite[\S 1.2]{g2}), this cohomology is a graded-commutative and cocommutative Hopf algebra. Denote by $\pr$ its subspace of primitive elements, i.e.~elements $x$ whose image under the comultiplication is $x\otimes 1 + 1 \otimes x$. This space is graded, by restriction of the grading on $H^{\bullet}(\Gr_{\Gv})$: $\pr=\oplus_n \pr^n$. For any $c \in \pr^n$, the cup product with $c$ induces, for any $M$ in $\Perv_{\eq{\GvO}}(\Gr_{\Gv})$, a functorial morphism
\[
c \cup - : H^{\bullet}(M) \to H^{\bullet + n}(M).
\]
Hence, via $\cS_G$, we obtain an endomorphism $\phi(c)$ of the functor $\For : \Rep(G) \to \Vect_{\C}$. By \cite[Proposition 2.7]{yz}, this endomorphism satisfies $\phi(c)_{V_1 \otimes V_2}=\phi(c)_{V_1} \otimes \id_{V_2} + \id_{V_1} \otimes \phi(c)_{V_2}$. Hence one obtains an automorphism $\widetilde{\phi}(c)$ of the \emph{tensor} functor $\Rep(G) \to \mathrm{Mod}(\C[\varepsilon]/\varepsilon^2)$ which sends $V$ to $V \otimes_{\C} \C[\varepsilon]/\varepsilon^2 = V \oplus V \cdot \varepsilon$ by setting
\[
\widetilde{\phi}(c)_V :=
\left( \begin{array}{cc}
\id_V & 0 \\
\phi(c)_V & \id_V
\end{array} \right) :
V \oplus V \cdot \varepsilon \to V \oplus V \cdot \varepsilon.
\]
By Tannakian formalism (\cite[Proposition 2.8]{dm}), this gives us a point in $G(\C[\varepsilon]/\varepsilon^2)$. Moreover, the image of this point under the morphism $G(\C[\varepsilon]/\varepsilon^2) \to G$ induced by the evaluation at $\varepsilon=0$ is $1$ (because the induced automorphism of the fiber functor $\For$ is trivial). Hence one obtains a point $\psi(c) \in \fg$. This way we have constructed a Lie algebra morphism
\[
\psi : \pr \to \fg.
\]
Note that this morphism is $\C^{\times}$-equivariant, where the action on $\pr$ corresponds to the grading defined above, and the action on $\fg$ is via the adjoint action of the cocharacter $2\rhov : \C^{\times} \to T$.

Let $\cL_{\det}$ be a very ample line bundle\footnote{The choice of $\cL_{\det}$ is not unique; however, it is easily seen that our constructions essentially do not depend on this choice, mainly because we work with sheaves with coefficients in characteristic $0$. If $\Gv$ is simple then there is a natural candidate for $\cL_{\det}$, namely the determinant line bundle, which explains our notation.} on $\Gr_{\Gv}$. Consider the first Chern class $c_1(\cL_{\det}) \in H^2(\Gr_{\Gv})$. We let
\[
e_G:=\psi(c_1(\cL_{\det})) \in \fg.
\]
This element has weight $2$ for the adjoint action of the cocharacter $2\rhov$. Hence $e_G \in \bigoplus_{\alpha \in \Delta} \fg_{\alpha}$. Moreover, $e_G$ is a regular nilpotent element of $\fg$ (see \cite{g2} or \cite[Proposition 5.6]{yz}), or equivalently its component on any $\fg_{\alpha}$ is nonzero. (Note that \cite[Proposition 5.6]{yz} gives a much more explicit description of $e_G$ if $\cL_{\det}$ is the determinant line bundle; we will not need this in this paper.) As $\pr$ is an abelian Lie algebra, the morphism $\psi$ factors through a $\C^{\times}$-equivariant morphism (denoted similarly for simplicity) $\psi:\pr \to \fg^{e_G}$, where $\fg^{e_G}$ is the centralizer of $e_G$ in $\fg$. 

Now, assume for a moment that $\Gv$ is semisimple and simply connected, so that $\Gr_{\Gv}$ is irreducible. By general theory of cocommutative Hopf algebras (see \cite[Theorem 13.0.1]{sw}), $H^{\bullet}(\Gr_{\Gv})$ is isomorphic to the symmetric algebra of the space $\pr$. By \cite[Proposition 1.7.2]{g2} or \cite[Corollary 6.4]{yz}, the morphism $\psi:\pr \to \fg^{e_G}$ is an isomorphism. Hence we obtain an isomorphism of graded Hopf algebras
\begin{equation}
\label{eqn:isom-cohomology}
\psi : H^{\bullet}(\Gr_{\Gv}) \xrightarrow{\sim} \mathrm{S}(\fg^{e_G}).
\end{equation} 
Note that, by construction, if $M$ is in $\Perv_{\eq{\GvO}}(\Gr_{\Gv})$ and $c \in H^{\bullet}(\Gr_{\Gv})$, the endomorphism of $H^{\bullet}(M)$ induced by the cup product with $c$ coincides with the action of $\psi(c)$ on $(\cS_G)^{-1}(M)$.

For a general reductive $\Gv$, the connected component $\Gr_{\Gv}^0$ of $\Gr_{\Gv}$ containing the base point $\GvO/\GvO$ identifies with the affine Grassmannian of the simply-connected cover of the derived subgroup of $\Gv$. Hence \eqref{eqn:isom-cohomology} gives a description of $H^{\bullet}(\Gr_{\Gv}^0):=H^{\bullet}(\Gr_{\Gv}^0,\C)$.

\subsection{An $\mathrm{Ext}$-algebra} 
\label{ss:Ext-algebra}

Let us define the objects 
\[
1_G:=S_G(\C) \quad \text{and} \quad \cR_G:=S_G(\C[G]),
\]
where $\C$ is the trivial $G$-module, and $\C[G]$ is the (left) regular representation of $G$. Then $1_G$ is an object of $\Perv_{\mon{\GvO}}(\Gr_{\Gv})$ (more precisely the skyscraper sheaf at $L_0=\GvO/\GvO \in \Gr_{\Gv}$) and $\cR_G$ is an ind-object in $\Perv_{\mon{\GvO}}(\Gr_{\Gv})$. This object is a ring-object, i.e.~there is a natural associative (and commutative) product 
\begin{equation} \label{eqn:def-m}
\sm : \cR_G \star \cR_G \to \cR_G
\end{equation}
induced by the multiplication in $\C[G]$. Moreover, $\cR_G$ is endowed with an action of $G$ (induced by the right multiplication of $G$ on itself).

More concretely, one can choose an isomorphism
\[
\cR_G \ \cong \ \varinjlim_{k \geq 0} \, \cR_{G,k},
\]
where $\cR_{G,k}$ is an object of $\Perv_{\mon{\GvO}}(\Gr_{\Gv})$ endowed with an action of $G$ for any $k$, and such that the multiplication $\sm$ of \eqref{eqn:def-m} is induced by $G$-equivariant morphisms
\[
\sm_{k,l} : \cR_{G,k} \star \cR_{G,l} \to \cR_{G,k+l}.
\]

Consider the $G$-equivariant graded algebra 
\[
\Ext^{\bullet}_{\mon{\GvO}}(1_{G}, \cR_G)
\]
whose $i$-th component is
\[
\varinjlim_{k \geq 0} \, \Hom_{\cDb_{\mon{\GvO}}(\Gr_{\Gv})}(1_G, \cR_{G,k}[i]).
\]
(Note that this definition is consistent with the usual formula for morphisms with values in an ind-object, see \cite[Equation (2.6.1)]{ks2}. In particular, it does not depend on the choice of the $\cR_{G,k}$'s.) The action of $G$ on this vector space is induced by the action on $\cR_G$. The (graded) algebra structure can be described as follows. Take $\xi \in \Ext^i_{\mon{\GvO}}(1_{G}, \cR_G)$, $\zeta \in \Ext^j_{\mon{\GvO}}(1_{G}, \cR_G)$; then the product $\xi \cdot \zeta$ is the composition 
\[ 
1_G \xrightarrow{\zeta} \cR_G[j] \cong 1_G \star \cR_G[j] \xrightarrow{\xi \star \cR_G[j]} \cR_G \star \cR_G[i+j] \xrightarrow{\sm[i+j]} \cR_G[i+j].
\]
(This structure is also considered in \cite[\S 7.2]{abg}.)

On the other hand, consider the algebra
\[
\C[\cN_G]
\]
of functions on the nilpotent cone $\cN_G$ of $G$. It is $G$-equivariant, and graded by the $\C^{\times}$-action defined in \S \ref{ss:equivalence}.

\begin{prop} 
\label{prop:Ext-algebra}

There exists an isomorphism of $G$-equivariant graded algebras \[ \Ext^{\bullet}_{\mon{\GvO}}(1_{G}, \cR_G) \ \cong \ \C[\cN_G]. \]

\end{prop}

\begin{proof} This result is proved in \cite[Theorem 7.3.1]{abg} in the case where $G$ is semisimple and adjoint. (Part of the arguments for this proof are reproduced in the proof of Proposition \ref{prop:isom-morphisms} below.) The general case follows, as both of these algebras are unchanged under the replacement of $G$ by $G/Z(G)$.\end{proof}

\subsection{A projective resolution of $1_G$}
\label{ss:projective-resolution}

Recall the definition of the category of pro-object in an abelian category (see \cite[Definition 1.11.4]{ks1} or \cite[Definition 6.1.1]{ks2}); recall also that this category is abelian (see \cite[Theorem 8.6.5(i)]{ks2}). In this subsection we construct a projective resolution of the object $1_G$ in the category of pro-objects in $\Perv_{\mon{\Iv}}(\Gr_{\Gv})$. A similar construction is also performed (without much details) in \cite[\S 9.5]{abg}.

Let us fix a collection of closed finite unions of $\Iv$-orbits
\[
\{L_0\}=X_0 \subset X_1 \subset X_2 \subset \cdots
\]
such that $\Gr_{\Gv}=\cup_{n \geq 0} \, X_n$. For any $n \geq 0$, we denote by $i_n : X_n \hookrightarrow \Gr_{\Gv}$ the inclusion.

For any $n \geq 0$, it follows from \cite[Theorem 3.3.1]{bgs} that the category $\Perv_{\mon{\Iv}}(X_n)$ is equivalent to the category of finite dimensional modules over a finite-dimensional $\C$-algebra. Hence one can choose a projective resolution
\[
\cdots \to P^{-2}_n \xrightarrow{d_n^{-2}} P^{-1}_n \xrightarrow{d_n^{-1}} P^0_n \to 1_G
\]
which is minimal, i.e.~such that for any $i \leq 0$, $P_n^i$ is a projective cover of $\ker(d^{i+1}_n)$. (Here, by an abuse of notation, we consider $1_G$ as an object of $\Perv_{\mon{\Iv}}(X_n)$, without mentioning which ``$n$" is considered.)

Consider now the inclusion $\iota_n : X_n \hookrightarrow X_{n+1}$. For any projective object $P$ in $\Perv_{\mon{\Iv}}(X_{n+1})$, $(\iota_n)^* P$ is a perverse sheaf (because $P$ has a filtration with standard objects as subquotients, see \cite[Theorem 3.3.1]{bgs}), and it is even a projective object in $\Perv_{\mon{\Iv}}(X_n)$ (because the functor $(\iota_n)_*$ is exact). More precisely, we have the following.

\begin{lem}
\label{lem:restriction-projective}

For any $j \leq 0$ there is an isomorphism
\[
(\iota_n)^* P^j_{n+1} \ \cong \ P^j_n.
\]

\end{lem}

\begin{proof}
Fix $j \leq 0$. The projective object $P^j_n$, respectively $P^j_{n+1}$, is the direct sum of the projective covers in $\Perv_{\mon{\Iv}}(X_n)$, respectively $\Perv_{\mon{\Iv}}(X_{n+1})$, of the simple objects $L$ in $\Perv_{\mon{\Iv}}(X_n)$, respectively $L'$ in $\Perv_{\mon{\Iv}}(X_{n+1})$, such that $\Hom(P^j_n, L) \neq 0$, respectively $\Hom(P^j_{n+1}, L') \neq 0$, counted with multiplicity. By minimality we have isomorphisms
\[
\Hom(P^j_n, L) \, \cong \, \Ext^{-j}_{\Perv_{\mon{\Iv}}(X_n)}(1_G,L), \quad \Hom(P^j_{n+1}, L') \, \cong \, \Ext^{-j}_{\Perv_{\mon{\Iv}}(X_{n+1})}(1_G,L').
\]
In fact, denoting by $P(n,L)$, respectively $P(n+1,L')$ the projective cover of a simple object $L$ in $\Perv_{\mon{\Iv}}(X_n)$, respectively $L'$ in $\Perv_{\mon{\Iv}}(X_{n+1})$, there are (non-canonical) isomorphisms
\begin{align*}
P_n^j \ \cong \ & \bigoplus_{\genfrac{}{}{0pt}{}{L \text{ simple in}}{\Perv_{\mon{\Iv}}(X_n)}} \bigl( \Ext^{-j}_{\Perv_{\mon{\Iv}}(X_n)}(1_G,L) \bigr)^* \otimes_{\C} P(n,L), \\
P_{n+1}^j \ \cong \ & \bigoplus_{\genfrac{}{}{0pt}{}{L' \text{ simple in}}{\Perv_{\mon{\Iv}}(X_{n+1})}} \bigl( \Ext^{-j}_{\Perv_{\mon{\Iv}}(X_{n+1})}(1_G,L') \bigr)^* \otimes_{\C} P(n+1,L').
\end{align*}

If $L$ is in $\Perv_{\mon{\Iv}}(X_n)$, then 
\[
\Ext^{-j}_{\Perv_{\mon{\Iv}}(X_n)}(1_G,L) \ \cong \ \Ext^{-j}_{\Perv_{\mon{\Iv}}(X_{n+1})}(1_G,(\iota_n)_*L).
\]
Indeed, by \cite[Corollary 3.3.2]{bgs} the left-hand side, respectively right-hand side, is isomorphic to 
\[
\Hom_{\cDb_{\mon{\Iv}}(X_n)}(1_G,L[-j]), \text{ respectively } \Hom_{\cDb_{\mon{\Iv}}(X_{n+1})}(1_G,(\iota_n)_* L[-j]).
\]
Now these spaces coincide since $(\iota_n)^* 1_G \cong 1_G$. Moreover, by construction the projective cover of $L$ in $\Perv_{\mon{\Iv}}(X_n)$ is isomorphic to the restriction to $X_n$ of the projective cover of $(\iota_n)_* L$ in $\Perv_{\mon{\Iv}}(X_{n+1})$ (see \cite[proof of Theorem 3.2.1]{bgs}). 

On the other hand, if $L'$ is a simple object associated to an $\Iv$-orbit included in $X_{n+1} \smallsetminus X_n$, then the projective cover of $L'$ in $\Perv_{\mon{\Iv}}(X_{n+1})$ is supported in $X_{n+1} \smallsetminus X_n$ (see again \cite[proof of Theorem 3.2.1]{bgs}, or use reciprocity, see the arguments in \cite[step 2 of the proof of Theorem 9.5]{ar}). This concludes the proof.
\end{proof}

More precisely, the resolution $P_n^{\bullet}$ being fixed, one can choose the minimal projective resolution $P_{n+1}^{\bullet}$ in such a way that we have an isomorphism of \emph{complexes}
\[
(\iota_n)^* P_{n+1}^{\bullet} \ \cong \ P_n^{\bullet}.
\]
In particular, by adjunction this provides a morphism of complexes
\[
(i_{n+1})_* P^{\bullet}_{n+1} \to (i_n)_* P^{\bullet}_n.
\]
For any $j \leq 0$, we set
\[
P^j\ := \ \varprojlim_{n \geq 0} \, (i_n)_* P^j_n,
\]
a pro-object in $\Pervm_{\mon{\Iv}}(\Gr_{\Gv})$. For $j <0$, the differentials $d^j_n : P^j_n \to P^{j+1}_n$ induce morphisms $d^j : P^j \to P^{j-1}$ such that $d^{j+1} \circ d^j=0$. Hence one can consider the complex of pro-objects
\[
P^{\bullet}:=\bigl( \cdots \to P^{-2} \to P^{-1} \to P^0 \to 0 \to \cdots \bigr),
\]
where $P^i$ is in degree $i$. There is a natural morphism of complexes $P^{\bullet} \to 1_G$.

\begin{lem}
\label{lem:projective-resolution}

\begin{enumerate}
\item For any $k$, the object $P^k$ is projective in the category of pro-objects in $\Perv_{\mon{\Iv}}(\Gr_{\Gv})$.
\item The morphism $P^{\bullet} \to 1_G$ is a quasi-isomorphism.
\end{enumerate}

\end{lem}

\begin{proof}
(1) First, let us prove that the functor
\[
\Hom(P^k,-) : \Perv_{\mon{\Iv}}(\Gr_{\Gv}) \to \{ \C\text{-vector spaces} \}
\]
is exact, and takes values in $\Vect_{\C}$. (Here, we consider morphisms in the category of pro-objects.) By definition of pro-objects, for any $M$ in $\Perv_{\mon{\Iv}}(\Gr_{\Gv})$ we have
\[
\Hom(P^k,M) \ \cong \ \varinjlim_{m \geq 0} \, \Hom((i_m)_* P^k_m,M)
\]
(see \cite[Equation (2.6.2)]{ks2}). By definition of the category $\Perv_{\mon{\Iv}}(\Gr_{\Gv})$, there exists $n \geq 0$ and an object $M'$ of $\Perv_{\mon{\Iv}}(X_n)$ such that $M \cong (i_n)_* M'$. By adjunction, for any $m \geq 0$ we have
\[
\Hom((i_m)_* P^k_m,M) \ \cong \ \Hom((i_n)^* (i_m)_* P^k_m,M').
\]
Using Lemma \ref{lem:restriction-projective}, we deduce that for $m \geq n$,
\[
\Hom((i_m)_* P^k_m,M) \ \cong \ \Hom(P^k_n,M').
\]
The claim follows, using the fact that $P^k_n$ is projective in the category $\Perv_{\mon{\Iv}}(X_n)$.

Now we conclude the proof of (1). Consider a short exact sequence $M \hookrightarrow N \twoheadrightarrow Q$ in the abelian category of pro-objects in $\Perv_{\mon{\Iv}}(\Gr_{\Gv})$. By \cite[Proposition 8.6.6(a)]{ks2}, this exact sequence is the projective limit of a projective system $(M_i \hookrightarrow N_i \twoheadrightarrow Q_i)_{i \in I}$ of short exact sequences in $\Perv_{\mon{\Iv}}(\Gr_{\Gv})$ indexed by a small filtrant category $I$. By our intermediate result, for any $i$ the sequence
\[
0 \to \Hom(P^k,M_i) \to \Hom(P^k,N_i) \to \Hom(P^k,Q_i) \to 0
\]
is an exact sequence of finite-dimensional vector spaces. It is well known that small filtrant inductive limits of short exact sequences of vector spaces are exact. Using the fact that the duality $V \mapsto V^*$ is exact, restricts to an involution on finite-dimensional vector spaces, and transforms inductive limits into projective limits, we deduce that the same property holds for small filtrant projective limits of short exact sequences of \emph{finite-dimensional} vector spaces. In particular, we get a short exact sequence
\[
0 \to \varprojlim_i \, \Hom(P^k,M_i) \to \varprojlim_i \, \Hom(P^k,N_i) \to \varprojlim_i \, \Hom(P^k,Q_i) \to 0.
\]
In other words, the sequence
\[
0 \to \Hom(P^k,M) \to \Hom(P^k,N) \to \Hom(P^k,Q) \to 0
\]
is exact, which proves the projectivity of $P^k$.

The claim (2) is obvious from the description of kernels and images in a category of pro-objects given in \cite[Lemma 8.6.4(ii)]{ks2}.
\end{proof}

\begin{rmk}
\label{rmk:projectives}
Let $\Gr_{\Gv}^0$ be the connected component of $\Gr_{\Gv}$ containing $\GvO/\GvO$, and let $G_{\mathrm{adj}}=G/Z(G)$. By the main results of \cite{abg}, the category $\Perv_{\mon{\Iv}}(\Gr_{\Gv}^0)$ is equivalent to a certain category of finite-dimensional representations of Lusztig's quantum group at a root of unity associated with the group $G_{\mathrm{adj}}$, see \cite[Theorem 9.10.2]{abg}. By \cite[Theorem 9.12]{apw}, the latter category has enough projectives. Hence the category $\Perv_{\mon{\Iv}}(\Gr_{\Gv}^0)$ has enough projectives.

It follows from this remark that in fact, for any $j \leq 0$, the sequence of objects $((i_n)_* P^j_n)_{n \geq 0}$ stabilizes for $n \gg 0$. As this proof is very indirect, and as this fact is not strictly necessary for our arguments, we will not use it.
\end{rmk}

\subsection{Formality}
\label{ss:formality}

Consider the $G$-equivariant\footnote{Here, the term ``$G$-equivariant'' is not really appropriate since for $i \in \Z$, $\sE^i(1_G,\cR_G)$ is not a rational $G$-module, but rather a projective limit of rational $G$-modules. We use this term by an abuse of language. Similarly, by a $G$-equivariant dg-module over this dg-algebra we mean a dg-module $M^{\bullet}$ such that for any $i \in \Z$, $M^i$ is a pro-object in the category of rational $G$-modules, compatibly with the $G$-action on the dg-algebra. Similar comments apply to the dg-algebra $A^{\lf}$ defined in \S \ref{ss:lf-subalgebra} below.} dg-algebra $\sE^{\bullet}(1_G,\cR_G)$ such that
\[
\sE^i(1_G,\cR_G) \ := \ \varinjlim_k \, \Hom^{i}(P^{\bullet},P^{\bullet} \star \cR_{G,k}).
\]
Here, we have set as usual
\[
\Hom^{i}(P^{\bullet},P^{\bullet} \star \cR_{G,k}) \ = \ \prod_{j \in \Z} \, \Hom(P^j,P^{i+j} \star \cR_{G,k}),
\]
where the morphisms are taken in the category of pro-objects in the category $\Perv_{\mon{\Iv}}(\Gr_{\Gv})$, and 
\[ 
P^{i+j} \star \cR_{G,k} \ := \ \varprojlim_n \, \bigl( (i_n)_* P^{i+j}_n \bigr) \star \cR_{G,k}.
\]
The action of $G$ is induced by the action on $\cR_G$, the differential on this dg-algebra is the natural one (induced by the differential of the complex $P^{\bullet}$), and the product is defined as follows. Consider $\xi \in \Hom^{i}(P^{\bullet},P^{\bullet} \star \cR_{G,k})$, $\zeta \in \Hom^{j}(P^{\bullet},P^{\bullet} \star \cR_{G,l})$; their product $\xi \cdot \zeta \in \Hom^{i+j}(P^{\bullet},P^{\bullet} \star \cR_{G,k+l})$ is defined as the composition
\[
P^{\bullet} \xrightarrow{\zeta} P^{\bullet} \star \cR_{G,l}[j] \xrightarrow{\xi \star \cR_{G,l}[j]} P^{\bullet} \star \cR_{G,k} \star \cR_{G,l}[i+j] \xrightarrow{P^{\bullet} \star \sm_{k,l}[i+j]} P^{\bullet} \star \cR_{G,k+l} [i+j].
\]

On the other hand, recall the dg-algebra $\Ext^{\bullet}_{\mon{\GvO}}(1_G,\cR_G)$ defined in \S \ref{ss:Ext-algebra}.

\begin{prop}
\label{prop:morphism-dg-algebras}

There exists a natural quasi-isomorphism of $G$-equivariant dg-algebras
\[
\phi : \sE^{\bullet}(1_G,\cR_G) \to \Ext^{\bullet}_{\mon{\GvO}}(1_G,\cR_G).
\]

\end{prop}

\begin{proof}
The morphism of complexes $P^{\bullet} \to 1_G$ induces a morphism of complexes
\[
\sE^{\bullet}(1_G,\cR_G) \to \Hom^{\bullet}(P^{\bullet},\cR_G),
\]
where we use the same notation as above, i.e.~by definition we have
\[
\Hom^{i}(P^{\bullet},\cR_G) \ = \ \varinjlim_k \, \Hom(P^{-i},\cR_{G,k}).
\]

Any morphism $P^{-i} \to \cR_{G,k}$ factors through the composition
\[
P^{-i} \to P^{-i}_n \to P^{-i}_n / \rad(P^{-i}_n)
\]
for some $n$. (This follows from the definition of morphisms in the category of pro-objects, and from the fact that $\cR_{G,k}$ is a semisimple perverse sheaf.) Hence, as the resolutions we have chosen are minimal, $\Hom(P^{-i},\cR_{G,k})$ is isomorphic to the space of morphisms of chain complexes $P^{\bullet} \to \cR_{G,k}[i]$. Similar arguments show that any such morphism of complexes which is homotopic to zero is in fact zero. Hence $\Hom(P^{-i},\cR_{G,k})$ is also isomorphic to the space of morphisms $P^{\bullet} \to \cR_{G,k}[i]$ in the homotopy category of pro-objects in $\Perv_{\mon{\Iv}}(\Gr_{\Gv})$. By Lemma \ref{lem:projective-resolution}, we deduce that $\Hom(P^{-i},\cR_{G,k})$ is the space of morphisms $1_G \to \cR_{G,k}[i]$ in the derived category of the abelian category of pro-objects in $\Perv_{\mon{\Iv}}(\Gr_{\Gv})$. By \cite[Theorem 15.3.1(i)]{ks2} and equivalence \eqref{eqn:realization-functor-equivalence}, this space is also isomorphic to $\Ext^i_{\mon{\GvO}}(1_G,\cR_G)$. These considerations imply that there is a natural isomorphism of complexes 
\[
\Hom^{\bullet}(P^{\bullet},\cR_G) \ \cong \ \Ext^{\bullet}_{\mon{\GvO}}(1_G,\cR_G),
\]
where the right-hand side has trivial differential.

Hence we have constructed a morphism of complexes
\[
\phi : \sE^{\bullet}(1_G,\cR_G) \to \Ext^{\bullet}_{\mon{\GvO}}(1_G,\cR_G).
\]
It follows directly from the definitions that this morphism is compatible with products, hence is a morphism of dg-algebras. The fact that it is a quasi-isomorphism follows from Lemma \ref{lem:projective-resolution}(1) and from the fact that the morphism $P^{\bullet} \star \cR_{G,k} \to \cR_{G,k}$ is a quasi-isomorphism for any $k \geq 0$ by Lemma \ref{lem:projective-resolution}(2).
\end{proof}

\subsection{Construction of the functor}
\label{ss:construction-functor}

Now we can construct a functor
\[
F_G : \cDb_{\mon{\GvO}}(\Gr_{\Gv}) \to \DGCGf(\cN_G).
\]
First, consider the category $\cCb \Perv_{\mon{\Iv}}(\Gr_{\Gv})$ of bounded chain complexes of objects of $\Perv_{\mon{\Iv}}(\Gr_{\Gv})$. One can define a functor from this category to the category of $G$-equivariant \emph{right} dg-modules over the dg-algebra $\sE^{\bullet}(1_G,\cR_G)$, which sends the complex $M^{\bullet}$ to the dg-module $\sE^{\bullet}(1_G,M^{\bullet} \star \cR_G)$ such that
\[
\sE^{i}(1_G,M^{\bullet} \star \cR_G) \ := \ \varinjlim_k \, \Hom^i(P^{\bullet}, M^{\bullet} \star \cR_{G,k}),
\]
where the differential is the natural one (induced by the differentials of $M^{\bullet}$ and $P^{\bullet}$), the action of $G$ is induced by the action on $\cR_G$, and the action of the dg-algebra is defined as follows. Consider some $\xi \in \Hom^{i}(P^{\bullet},M^{\bullet} \star \cR_{G,k})$ and $\zeta \in \Hom^{j}(P^{\bullet},P^{\bullet} \star \cR_{G,l})$; their product $\xi \cdot \zeta \in \Hom^{i+j}(P^{\bullet},M^{\bullet} \star \cR_{G,k+l})$ is defined as the composition
\[
P^{\bullet} \xrightarrow{\zeta} P^{\bullet} \star \cR_{G,l}[j] \xrightarrow{\xi \star \cR_{G,l}[j]} M^{\bullet} \star \cR_{G,k} \star \cR_{G,l}[i+j] \xrightarrow{M^{\bullet} \star \sm_{k,l}[i+j]} M^{\bullet} \star \cR_{G,k+l} [i+j].
\]
This functor is exact by Lemma \ref{lem:projective-resolution}(1), hence it induces a functor 
\begin{equation}
\label{eqn:construction-F_G-1}
\sE^{\bullet}(1_G,(-) \star \cR_G) : \cDb \Perv_{\mon{\Iv}}(\Gr_{\Gv}) \to \DGMG(\sE^{\bullet}(1_G,\cR_G)^{\mathrm{op}}).
\end{equation}

Then, the quasi-isomorphism $\phi$ induces an equivalence of categories
\begin{multline}
\label{eqn:construction-F_G-2}
\Ext^{\bullet}_{\mon{\GvO}}(1_G,\cR_G) \lotimes_{\sE^{\bullet}(1_G,\cR_G)^{\mathrm{op}}} - : \DGMG(\sE^{\bullet}(1_G,\cR_G)^{\mathrm{op}}) \\ \xrightarrow{\sim} \ \DGMG(\Ext^{\bullet}_{\mon{\GvO}}(1_G,\cR_G)),
\end{multline}
see \cite[Theorem 10.12.5.1]{bl} (or rather a $G$-equivariant analogue, which is easy since $G$ is a complex reductive group). Note that the dg-algebra $\Ext^{\bullet}_{\mon{\GvO}}(1_G,\cR_G)$ is graded-commutative and concentrated in even degrees; hence it is equal to its opposite dg-algebra.

Finally, the isomorphism of Proposition \ref{prop:Ext-algebra} induces an equivalence of categories
\begin{eqnarray}
\label{eqn:construction-F_G-3}
\DGMG(\Ext^{\bullet}_{\mon{\GvO}}(1_G,\cR_G)) \ \xrightarrow{\sim} \ \DGMG(\C[\cN_G]).
\end{eqnarray}

Composing the functors \eqref{eqn:construction-F_G-1}, \eqref{eqn:construction-F_G-2} and \eqref{eqn:construction-F_G-3} with the inclusion 
\[
\cDb_{\mon{\GvO}}(\Gr_{\Gv}) \hookrightarrow \cDb_{\mon{\Iv}}(\Gr_{\Gv}) \overset{\eqref{eqn:realization-functor-equivalence}}{\cong} \cDb \Perv_{\mon{\Iv}}(\Gr_{\Gv}),
\]
one obtains a functor
\[
\widetilde{F}_G : \cDb_{\mon{\GvO}}(\Gr_{\Gv}) \to \DGMG(\C[\cN_G]).
\]

\begin{lem} 
\label{lem:image-IC}

\begin{enumerate}
\item There is a natural isomorphism $\widetilde{F}_G(1_G) \, \cong \, \C[\cN_G]$.
\item For any $V$ in $\Rep(G)$ and $M$ in $\cDb_{\mon{\GvO}}(\Gr_{\Gv})$ there is a functorial isomorphism \[ \widetilde{F}_G(M \star \cS_G(V)) \ \cong \ \widetilde{F}_G(M) \otimes V \] in $\DGMG(\C[\cN_G])$.
\end{enumerate}

\end{lem}

\begin{proof} 
Claim (1) is obvious from definitions and the proof of Proposition \ref{prop:morphism-dg-algebras}.

(2) There exist isomorphisms of $G$-modules 
\[
V \otimes_{\C} \C[G] \, \cong \, V \otimes_{\C} \Ind_{\{1\}}^G(\C) \, \cong \, \Ind_{\{1\}}^G(V) \, \cong \, \C[G]^{\oplus \dim(V)}.
\]
Hence, applying the Satake equivalence $\cS_G$, one obtains an isomorphism of ind-perverse sheaves 
\[
\cS_G(V) \star \cR_G \, \cong \, \cR_G \otimes_{\C} V.
\]
The natural $G$-action on the left-hand side induced by the action on $\cR_G$ corresponds to the diagonal $G$-action on the right-hand side. The isomorphism follows.\end{proof}

It follows in particular from Lemma \ref{lem:image-IC} that the functor $\widetilde{F}_G$ factors through a functor
\[
F_G : \cDb_{\mon{\GvO}}(\Gr_{\Gv}) \to \DGCGf(\cN_G),
\]
as promised.

\subsection{$F_G$ is an equivalence}
\label{ss:FG-equivalence}

The following lemma is well known, see e.g.~\cite[Equation (2.4.1)]{g2}.

\begin{lem}
\label{lem:adjunction}

Let $V$ in $\Rep(G)$. The functors
\[
(-) \star \cS_G(V) : \cDb_c(\Gr_{\Gv}) \to \cDb_c(\Gr_{\Gv})
\quad \text{and} \quad
(-) \star \cS_G(V^*) : \cDb_c(\Gr_{\Gv}) \to \cDb_c(\Gr_{\Gv})
\]
are adjoint.\qed

\end{lem}

The next important step in the proof of Theorem \ref{thm:equivalence} is the following.

\begin{prop}
\label{prop:isom-morphisms}

For any $V_1, V_2$ in $\Rep(G)$, the functor $F_G$ induces an isomorphism of graded vector spaces
\begin{multline*} 
\Hom^{\bullet}_{\cDb_{\mon{\GvO}}(\Gr_{\Gv})}(\cS_G(V_1),\cS_G(V_2)) \\ \xrightarrow{\sim} \ \Hom^{\bullet}_{\DGCGf(\cN_G)}(\C[\cN_G] \otimes_{\C} V_1,\C[\cN_G] \otimes_{\C} V_2). 
\end{multline*}

\end{prop}

\begin{proof} This result is proved is the case when $G$ is semisimple in \cite[Proposition 1.10.4]{g2}. As the details will be needed later, we reproduce the proof. (See also \cite[\S\S 7.4--7.5]{abg} for similar arguments.)

First, Lemma \ref{lem:adjunction} reduces the proof to the case $V_1=\C$. For simplicity, we write $V$ for $V_2$. Consider the projection $V \twoheadrightarrow V^{Z(G)}$ orthogonal to other weight spaces of the diagonalizable group $Z(G)$. The perverse sheaf $\cS_G(V^{Z(G)})$ is the restriction of $\cS_G(V)$ to $\Gr_{\Gv}^0$. Hence the projection above induces an isomorphism
\begin{equation}
\label{eqn:prop-isom-morphisms-1}
\Hom^{\bullet}_{\cDb_{\mon{\GvO}}(\Gr_{\Gv})}(1_G,\cS_G(V)) \ \xrightarrow{\sim} \ \Hom^{\bullet}_{\cDb_{\mon{\GvO}}(\Gr_{\Gv}^0)}(1_G,\cS_G(V^{Z(G)})).
\end{equation}

Now, by the main result of \cite{g1}, the hypercohomology functor $H^{\bullet}(-)$ induces an isomorphism
\[
\Hom^{\bullet}_{\cDb_{\mon{\GvO}}(\Gr_{\Gv}^0)}(1_G,\cS_G(V^{Z(G)})) \ \xrightarrow{\sim} \ \Hom^{\bullet}_{H^{\bullet}(\Gr_{\Gv}^0)}(\C,(V^{\bullet})^{Z(G)}),
\]
where on the right-hand side we mean morphisms of graded modules, and $V^{\bullet}=H^{\bullet}(\cS_G(V))$ is $V$, graded by the action of the cocharacter $2\rhov$. Using also isomorphism \eqref{eqn:isom-cohomology} (for the group $G_{\mathrm{adj}}:=G/Z(G)$, whose Langlangs dual group is $\Gv_{\mathrm{sc}}$, the simply connected cover of the derived subgroup of $\Gv$, so that we have $\Gr_{\Gv_{\mathrm{sc}}} \xrightarrow{\sim} \Gr_{\Gv}^0$), we obtain an isomorphism
\[
\Hom^{\bullet}_{\cDb_{\mon{\GvO}}(\Gr_{\Gv}^0)}(1_G,\cS_G(V^{Z(G)})) \ \xrightarrow{\sim} (V^{\bullet})^{Z(G),\fg^{e_G}}.
\]
By \cite[Theorem 4.11]{spr}, we have $G^{e_G} \cong Z(G) \times U^{e_G}$, and $U^{e_G}$ is a connected unipotent group. Hence we have $V^{Z(G),\fg^{e_G}} = V^{G^{e_G}}$, and we finally obtain an isomorphism of graded vector spaces
\begin{equation}
\label{eqn:prop-isom-morphisms-2}
\Hom^{\bullet}_{\cDb_{\mon{\GvO}}(\Gr_{\Gv}^0)}(1_G,\cS_G(V^{Z(G)})) \ \xrightarrow{\sim} \ (V^{\bullet})^{G^{e_G}}.
\end{equation}

It is well known that the variety $\cN_G$ is normal, and that $G/G^{e_G} \cong G \cdot e_G$ has complement of codimension $2$ in $\cN_G$. Hence restriction induces an isomorphism of graded algebras $\C[\cN_G] \xrightarrow{\sim} \C[G/G^{e_G}]$, where the grading on $\C[\cN_G]$ is the one defined in \S \ref{ss:equivalence}, and the grading on $\C[G/G^{e_G}]$ is induced by the action of the cocharacter $2\rhov$ via the right regular representation of $G$ on $\C[G]$. It follows that restriction to $e_G \in \cN_G$ induces an isomorphism of graded vector spaces
\begin{equation}
\label{eqn:prop-isom-morphisms-3}
(\C[\cN_G] \otimes V)^G \ \xrightarrow{\sim} \ (V^{\bullet})^{G^{e_G}}.
\end{equation}

Finally, there is a natural isomorphism of graded vector spaces
\begin{equation}
\label{eqn:prop-isom-morphisms-4}
\Hom^{\bullet}_{\DGCGf(\cN_G)}(\C[\cN_G],\C[\cN_G] \otimes_{\C} V) \ \cong \ (\C[\cN_G] \otimes V)^G.
\end{equation}
Indeed, by \cite[Lemma 10.12.2.2]{bl} (adapted to the $G$-equivariant setting), the morphisms in the left-hand side of \eqref{eqn:prop-isom-morphisms-4} can be computed in the homotopy category of $G$-equivariant $\C[\cN_G]$-dg-modules. Then the isomorphism \eqref{eqn:prop-isom-morphisms-4} is clear.

Combining isomorphisms \eqref{eqn:prop-isom-morphisms-1}, \eqref{eqn:prop-isom-morphisms-2}, \eqref{eqn:prop-isom-morphisms-3} and \eqref{eqn:prop-isom-morphisms-4} gives the proposition.\end{proof}

Finally we can prove Theorem \ref{thm:equivalence}.

\begin{proof}[Proof of Theorem {\rm \ref{thm:equivalence}}]
Using Lemma \ref{lem:image-IC} and Proposition \ref{prop:isom-morphisms}, the fact that $F_G$ is an equivalence follows from Lemma \ref{lem:triangulated-categories}.

Isomorphism \eqref{eqn:compatibility-F-S} is proved in Lemma \ref{lem:image-IC}(2).\end{proof}

\subsection{A locally-finite subalgebra}
\label{ss:lf-subalgebra}

Now we come to the proof of the mixed version of Theorem \ref{thm:equivalence}, namely Theorem \ref{thm:equivalence-mix}. For any $n \geq 0$ one can consider the category $\Pervm_{\mon{\Iv}}(X_n)$ defined as in \S \ref{ss:equivalence-mix} (replacing $\Gr_{\Gv}$ by $X_n$), and the forgetful functor
\[
\For : \Pervm_{\mon{\Iv}}(X_n) \to \Perv_{\mon{\Iv}}(X_n).
\] 
By \cite[Lemma 4.4.8]{bgs}, every projective object of $\Perv_{\mon{\Iv}}(X_n)$ can be lifted to $\Pervm_{\mon{\Iv}}(X_n)$. The object $1_G$ can also obviously be lifted to an object $1^{\mix}_G$ of $\Pervm_{\mon{\Iv}}(\Gr_{\Gv})$, of weight $0$. Hence one can choose the projective resolutions $P_n^{\bullet}$ of \S \ref{ss:projective-resolution} in such a way that there exists a projective resolution
\[
\cdots \to P_{\mix}^{-2} \to P_{\mix}^{-1} \to P_{\mix}^{0} \to 1^{\mix}_G
\]
in the category of pro-objects in $\Pervm_{\mon{\Iv}}(\Gr_{\Gv})$ such that $\For(P_{\mix}^{j}) \cong P^j$ for any $j \leq 0$, and similarly for the differentials. The pro-object $P_{\mix}^j$ is obtained by taking a formal projective limit of objects $P^j_{n,\mix}$, where $P^j_{n,\mix}$ is a projective object in $\Pervm_{\mon{\Iv}}(X_n)$ such that $P^j_n \cong \For(P^j_{n,\mix})$.

Recall the equivalence $\cS_G^0$ of \eqref{eqn:satake-0}. We set
\[
\cR_G^{\mix} \ := \ \cS_G^0(\C[G]),
\]
an ind-object in the category $\Perv^0_{\mon{\Iv}}(\Gr_{\Gv})$. This object is endowed with an associative and commutative multiplication map
\[
\sm^{\mix} : \cR_G^{\mix} \star \cR_G^{\mix} \to \cR_G^{\mix}.
\]
One can choose the subobjects $\cR_{G,k} \subset \cR_{G}$ of \S \ref{ss:Ext-algebra} in such a way that there exists $\cR_{G,k}^{\mix} \subset \cR_{G}^{\mix}$ such that the isomorphism $\For(\cR_G^{\mix}) \cong \cR_G$ induces an isomorphism $\For(\cR_{G,k}^{\mix}) \cong \cR_{G,k}$ for any $k$, and $\sm_{k,l}$ can be lifted to a morphism $\sm_{k,l}^{\mix} : \cR^{\mix}_{G,k} \star \cR^{\mix}_{G,l} \to \cR^{\mix}_{G,k+l}$.

Then one can also consider, for any $j$, the object $P^j_{\mix} \star \cR_{G,k}^{\mix}$. This object is a pro-object in the category of perverse sheaves on $\Gr_{\Gv,\F_p}$, but a priori not in the category $\Pervm_{\mon{\Iv}}(\Gr_{\Gv})$. Still, the pro-object 
\[
P^j \star \cR_{G,k} \ \cong \ \For\bigl( P^j_{\mix} \star \cR_{G,k}^{\mix} \bigr)
\]
is endowed with an automorphism induced by the Frobenius.

With these choices, by definition the dg-algebra $\sE^{\bullet}(1_G,\cR_G)$ of \S \ref{ss:formality} is equipped with an automorphism
\[
\Fr : \sE^{\bullet}(1_G,\cR_G) \xrightarrow{\sim} \sE^{\bullet}(1_G,\cR_G)
\]
induced by the Frobenius. We would like to decompose the dg-algebra $\sE^{\bullet}(1_G,\cR_G)$ according to the generalized eigenspaces of this automorphism. However, as this dg-algebra has infinite-dimensional graded pieces, we need to be more careful.

First, observe that the Frobenius also induces an automorphism of the dg-algebra $\Ext^{\bullet}_{\mon{\GvO}}(1_G,\cR_G)$, again denoted $\Fr$.

\begin{lem}
\label{lem:action-Fr-cohomology}

\begin{enumerate}
\item The morphism $\phi$ of Proposition {\rm \ref{prop:morphism-dg-algebras}} commutes with the automorphisms $\Fr$.
\item For any $n \geq 0$, $\Fr$ acts on $\Ext^n_{\mon{\GvO}}(1_G,\cR_G)$ as multiplication by $p^{n/2}$.
\end{enumerate}

\end{lem}

\begin{proof}
Statement (1) is clear from definitions. Let us consider (2). By adjunction we have, for any $n \geq 0$,
\[
\Hom_{\cDb_{\mon{\GvO}}(\Gr_{\Gv})}(1_G,\cR_{G,k}[n]) \ \cong \ H^n(i_0^! \cR_{G,k}).
\]
Hence the result follows from condition $(*)$ of \cite[\S 4.4]{bgs}.
\end{proof}

Statement (2) of this lemma, together with Proposition \ref{prop:Ext-algebra}, imply that there exists an isomorphism of $G$-equivariant bigraded algebras
\begin{equation}
\label{eqn:isom-bigraded-algebras}
\Ext^{\bullet}_{\mon{\GvO}}(1_G,\cR_G) \ \cong \ \C[\cN_G],
\end{equation}
where the additional $\Z$-grading on the left-hand side is induced by the action of $\Fr$, while the additional grading on the right-hand side is defined in \S \ref{ss:equivalence-mix}. We will consider this isomorphism as an isomorphism of dgg-algebras, where both sides have trivial differential.

To simplify notation, let us set
\[
A^{\bullet} \ := \ \sE^{\bullet}(1_G,\cR_G).
\]
By definition, we have $A^{\bullet} = \varinjlim_k A^{\bullet}_k$, where
\[
A^n_k \ = \ \prod_{j \in \Z} \, \Hom(P^j,P^{j+n} \star \cR_{G,k}).
\]
Let $A^n_{j,k}:=\Hom(P^j,P^{j+n} \star \cR_{G,k})$. Then we have 
\[
A^n_{j,k} = \varprojlim_{m \geq 0} \Hom(P^j,P^{j+n}_m \star \cR_{G,k}).
\]
We set $A^n_{j,k,m}:=\Hom(P^j,P^{j+n}_m \star \cR_{G,k})$. By construction, the automorphism $\Fr$ considered above is induced by automorphisms denoted similarly $\Fr : A^n_{j,k,m} \xrightarrow{\sim} A^n_{j,k,m}$.

\begin{lem}
\label{lem:action-Fr-lf}

The action of $\Fr$ on $A^n_{j,k,m}$ is locally finite, and all its eigenvalues are integral powers of $p^{1/2}$.

\end{lem}

\begin{proof}
By definition we have
\[
A^n_{j,k,m} \ = \ \varinjlim_{l \geq 0} \, \Hom(P^j_l,P^{j+n}_m \star \cR_{G,k}).
\]
Each space $\Hom(P^j_l,P^{j+n}_m \star \cR_{G,k})$ is finite-dimensional, and stable under the action of $\Fr$. Hence it is enough to prove that the eigenvalues of the restriction of $\Fr$ to this space are integral powers of $p^{1/2}$.

By \cite[Lemma 7.8 and its proof]{ar}, the eigenvalues of the Frobenius on the Hom-space between any two objects which are images under $\For$ of objects of $\Pervm_{\mon{\Iv}}(\Gr_{\Gv})$ are integral powers of $p^{1/2}$. Here, as explained above, $P^{j+n}_{\mix,m} \star \cR^{\mix}_{G,k}$ is a priori not necessarily an object of $\Pervm_{\mon{\Iv}}(\Gr_{\Gv})$, but at least it is an extension of objects of $\Pervm_{\mon{\Iv}}(\Gr_{\Gv})$. Hence the same property holds.
\end{proof}

Thanks to Lemma \ref{lem:action-Fr-lf}, one can write
\[
A^n_{j,k,m} \ = \ \bigoplus_{r \in \Z} A^n_{j,k,m,r},
\]
where $A^n_{j,k,m,r}$ is the generalized eigenspace of $\Fr$ for the eigenvalue $p^{r/2}$. Note that for any $m \geq 1$, the morphism $A^n_{j,k,m} \to A^n_{j,k,m-1}$ is surjective (by projectivity of $P^j$, see Lemma \ref{lem:projective-resolution}), and compatible with the direct sum decomposition.

Now we can define our ``locally-finite'' subalgebra of $A^{\bullet}$ as follows. First we set
\[
A^{\lf,n}_{k} \ = \ \bigoplus_{r \in \Z} \, \prod_{j \in \Z} \, \varprojlim_{m \geq 0} \, A^n_{j,k,m,r},
\]
and then
\[
A^{\lf,n} \ = \ \varinjlim_{k \geq 0} \, A^{\lf,n}_{k}.
\]
One easily checks that $A^{\lf,\bullet}$ is a sub-dg-algebra of $A^{\bullet}$, stable under the action of $G$.

\begin{lem}
\label{lem:qis-lf}

The inclusion $A^{\lf,\bullet} \hookrightarrow A^{\bullet}$ is a quasi-isomorphism.

\end{lem}

\begin{proof}
Clearly, it suffices to prove that the inclusion $A^{\lf,\bullet}_k \hookrightarrow A^{\bullet}_k$ is a quasi-isomor\-phism for any $k \geq 0$.

First, we check that the morphism $H^{\bullet}(A^{\lf,\bullet}_k) \to H^{\bullet}(A^{\bullet}_k)$ is surjective. Take some $f \in A^n_k$, annihilated by the differential. By definition, $f$ is a family $(f_j)_{j \in \Z}$, where $f_j$ is in $A^n_{j,k}$. Then, each $f_j$ is a family $(f_{j,m})_{m \geq 0}$, where $f_{j,m} \in A^n_{j,k,m}$, and each $f_{j,m}$ can be written as $f_{j,m} = \sum_r f_{j,m,r}$ with $f_{j,m,r} \in A^n_{j,k,m,r}$. To say that $d(f)=0$ amounts to saying that for any $j \in \Z$, $f_{j+1} \circ d_{P}^j = (-1)^n d_{P \star \cR_{G,k}}^{j+n} \circ f_j$. Then this equation can be rewritten as $f_{j+1,m} \circ d_{P}^j = (-1)^n d_{P_m \star \cR_{G,k}}^{j+n} \circ f_{j,m}$ for any $m$, and the latter equation can finally be rewritten as $f_{j+1,m,r} \circ d_{P}^j = (-1)^n d_{P_m \star \cR_{G,k}}^{j+n} \circ f_{j,m,r}$ for any $r$.

By Lemma \ref{lem:action-Fr-cohomology}, the action of $\Fr$ on the image of $f$ in cohomology is multiplication by $p^{n/2}$. We set $f'=(f'_j)_{j \in \Z} \in A^{\lf,n}$, where $f_j'$ is the family $(f'_{j,m})_{m \geq 0}$, where for any $m$ we have set $f'_{j,m}=f_{j,m,n}$. By the remarks above, $f'$ is annihilated by the differential. We claim that $f$ and $f'$ define the same class in cohomology, which proves surjectivity. Indeed, for any $r \neq n$, the family $(f_{j,m,r})_{j,m}$ is an element of $A^{\bullet}$ annihilated by the differential, and has trivial image in cohomology for reasons of degree. Hence it is in the image of the differential. The claim follows.

Injectivity can be proved similarly, which concludes the proof.
\end{proof}

By definition, the dg-algebra $A^{\lf}$ is endowed with an additional $\Z$-grading, which makes it a $G$-equivariant dgg-algebra. By Lemma \ref{lem:qis-lf} and isomorphism \eqref{eqn:isom-bigraded-algebras}, the morphism $\phi$ of Proposition \ref{prop:morphism-dg-algebras} induces a quasi-isomorphism of dgg-algebras
\begin{equation}
\label{eqn:qis}
A^{\lf} \ \xrightarrow{\qis} \ \C[\cN_G],
\end{equation}
where the right-hand side is endowed with the dgg-algebra structure defined in \S \ref{ss:equivalence-mix}.

\subsection{Mixed version of $F_G$}
\label{ss:mixed-version}

Now we can construct the functor
\[
F_G^{\mix} : \cD^{\mix}_{\mon{\GvO}}(\Gr_{\Gv}) \to \cDb_{\fr} \Coh^{G \times \Gm}(\cN_G).
\]
Using equivalence \eqref{eqn:equivalence-regrading}, it is enough to construct a functor
\[
\underline{F}_G^{\mix} : \cD^{\mix}_{\mon{\GvO}}(\Gr_{\Gv}) \to \DGCGmf(\cN_G).
\]
As in \S \ref{ss:construction-functor}, this functor will be the restriction of a functor
\[
\overline{F}_G^{\mix} : \cDb\Perv^{\mix}_{\mon{\Iv}}(\Gr_{\Gv}) \to \DGC^{G \times \Gm}(\cN_G).
\]
The functor $\overline{F}_G^{\mix}$ is constructed as follows. We use the same notation as in \S \ref{ss:lf-subalgebra}.

For any bounded complex $M^{\bullet}$ of objects of $\Perv^{\mix}_{\mon{\Iv}}(\Gr_{\Gv})$, the $\sE^{\bullet}(1_G,\cR_G)$-dg-module $\sE^{\bullet}(1_G,\For(M^{\bullet}) \star \cR_G)$ of \S \ref{ss:construction-functor} is endowed with an automorphism induced by the Frobenius. Moreover, by the same arguments as in the proof of Lemma \ref{lem:action-Fr-lf}, this action is locally finite, and the eigenvalues are all integral powers of $p^{1/2}$. Hence this action gives rise to an additional $\Z$-grading on $\sE^{\bullet}(1_G,\For(M^{\bullet}) \star \cR_G)$. Restricting the action to $A^{\lf}$, this grading makes $\sE^{\bullet}(1_G,\For(M^{\bullet}) \star \cR_G)$ an $A^{\lf}$-dgg-module, denoted $\sE^{\bullet,\bullet}(1_G,M^{\bullet} \star \cR_G)$. Deriving this (exact) functor, we obtain a functor
\[
\sE^{\bullet,\bullet}(1_G,(-) \star \cR_G) : \cDb\Perv^{\mix}_{\mon{\Iv}}(\Gr_{\Gv}) \to \DGM^{G \times \Gm}(A^{\lf}).
\]

Then, quasi-isomorphism \eqref{eqn:qis} induces an equivalence of categories
\[
\C[\cN_G] \, \lotimes_{A^{\lf}} \, (-) :
\DGMGm(A^{\lf}) \ \xrightarrow{\sim} \ \DGMGm(\C[\cN_G]).
\]
Composing this equivalence with the functor $\sE^{\bullet,\bullet}(1_G,(-) \star \cR_G)$ gives our functor $\overline{F}_G^{\mix}$. We denote by $\widetilde{F}_G^{\mix}$ the composition of this functor with the inclusion $\cD^{\mix}_{\mon{\GvO}}(\Gr_{\Gv}) \hookrightarrow \cDb \Pervm_{\mon{\Iv}}(\Gr_{\Gv})$. 

\begin{lem}
\label{lem:diagram-FG-FGmix}

The following diagram commutes up to an isomorphism of functors:
\[
\xymatrix@C=2cm{
\cD^{\mix}_{\mon{\GvO}}(\Gr_{\Gv}) \ar[d]_-{\For} \ar[r]^-{\widetilde{F}_G^{\mix}} & \DGMGm(\C[\cN_G]) \ar[d]^-{\For} \\
\cDb_{\mon{\GvO}}(\Gr_{\Gv}) \ar[r]^-{\widetilde{F}_G} & \DGMG(\C[\cN_G]). \\
}
\] 

\end{lem}

\begin{proof}
By construction of the equivalences, it is enough to prove that the composition of the restriction of scalars functor
\[
\mathrm{Rest}_1: \DGMG(\sE^{\bullet}(1_G,\cR_G)) \to \DGMG(A^{\lf})
\]
with the equivalence
\[
\C[\cN_G] \, \lotimes_{A^{\lf}} \, (-) : \DGMG(A^{\lf}) \to \DGMG(\C[\cN_G])
\]
is isomorphic to the equivalence
\[
\C[\cN_G] \, \lotimes_{\sE^{\bullet}(1_G,\cR_G)} \, (-) : \DGMG(\sE^{\bullet}(1_G,\cR_G)) \to \DGMG(\C[\cN_G]).
\]
However, the latter equivalences have inverses the restriction of scalars functors
\[
\mathrm{Rest}_2 : \DGMG(\C[\cN_G]) \to \DGMG(A^{\lf})
\]
and
\[
\mathrm{Rest}_3 : \DGMG(\C[\cN_G]) \to \DGMG(\sE^{\bullet}(1_G,\cR_G)).
\]
By construction, the morphism $A^{\lf} \to \C[\cN_G]$ is the composition of the morphisms $A^{\lf} \to \sE^{\bullet}(1_G,\cR_G)$ and $\sE^{\bullet}(1_G,\cR_G) \to \C[\cN_G]$. Hence we have $\mathrm{Rest}_2 = \mathrm{Rest}_1 \circ \mathrm{Rest}_3$, and the result follows.
\end{proof}

The same proof as that of Lemma \ref{lem:image-IC} gives the following result.

\begin{lem}
\label{lem:image-IC-mix}

For $V$ in $\Rep(G)$, there is an isomorphism in $\DGMGm(\C[\cN_G])$
\[
\widetilde{F}_G^{\mix}(\cS_G^0(V)) \ \cong \ V \otimes_{\C} \C[\cN_G],
\]
where, on the right-hand side, the $\Gm$-action on $V$ is trivial. \qed

\end{lem}

It follows from this lemma that $\widetilde{F}^{\mix}_G$  factors through a functor
\[
\underline{F}_G^{\mix} : \cD^{\mix}_{\mon{\GvO}}(\Gr_{\Gv}) \to \DGCGmf(\cN_G).
\]

\begin{prop}
\label{prop:FGmix-fully-faithful}

The functor $\underline{F}_G^{\mix}$ is fully faithful.

\end{prop}

\begin{proof}
The category $\cD^{\mix}_{\mon{\GvO}}(\Gr_{\Gv})$, respectively $\DGCGmf(\cN_G)$, is a graded version of the category $\cDb_{\mon{\GvO}}(\Gr_{\Gv})$, respectively $\DGCGf(\cN_G)$ in the sense of \cite[\S 2.3]{ar}. Hence for two objects $M,N$ in $\cD^{\mix}_{\mon{\GvO}}(\Gr_{\Gv})$ we have isomorphisms
\[
\Hom_{\cDb_{\mon{\GvO}}(\Gr_{\Gv})}(\For(M),\For(N)) \ \cong \ \bigoplus_{i \in \Z} \, \Hom_{\cD^{\mix}_{\mon{\GvO}}(\Gr_{\Gv})}(M,N\langle i \rangle), \]
and
\begin{multline*}
\Hom_{\DGCGf(\cN_G)}(\For \circ \underline{F}_G^{\mix}(M), \For \circ \underline{F}_G^{\mix}(N)) \ \cong \\ \bigoplus_{i \in \Z} \, \Hom_{\DGCGmf(\cN_G)}(\underline{F}_G^{\mix}(M),\underline{F}_G^{\mix}(N) \langle i \rangle).
\end{multline*}
By Theorem \ref{thm:equivalence} and Lemma \ref{lem:diagram-FG-FGmix}, we know that the functor $\underline{F}_G^{\mix}$ induces an isomorphism between the left-hand sides of these equations. Moreover, it respects the direct sum decompositions of the right-hand sides. Hence it also induces an isomorphism between the right-hand sides, and also between the summands associated to $i=0$. The result follows.
\end{proof}

Finally we can finish the proof of Theorem \ref{thm:equivalence-mix}.

\begin{proof}[Proof of Theorem {\rm \ref{thm:equivalence-mix}}]
By Proposition \ref{prop:FGmix-fully-faithful} and Lemma \ref{lem:image-IC-mix}, one can apply Lemma \ref{lem:triangulated-categories} to the functor $\underline{F}_G^{\mix}$, proving that it is an equivalence of categories. Composing with the equivalence \eqref{eqn:equivalence-regrading}, we obtain the equivalence $F_G^{\mix}$. 

The commutativity of the diagram of the theorem follows from Lemma \ref{lem:diagram-FG-FGmix}. By construction, the functor $\underline{F}_G^{\mix}$ commutes with internal shifts $\lan i \ran$. The last assertion of the theorem follows.
\end{proof}

\section{Constructible sheaves on $\Gr_{\Gv}$ and coherent sheaves on $\cN_G$:\\ second approach}
\label{sect:proof-2}

In this section we give a second proof of Theorems \ref{thm:equivalence} and \ref{thm:equivalence-mix}, which was inspired by the methods of \cite[\S 6.5]{bf}.

\subsection{Reminder on pretriangulated categories}

The notion of pretriangulated category was introduced in \cite{bk}. We will rather use \cite{dr, bll} as references. Recall the basic notions of dg-categories, see e.g.~\cite[\S 4.1]{bll}\footnote{In \cite{bll}, the authors assume that dg-categories are additive, but this is not really used in the results we quote. We only assume our dg-categories to be preadditive.} or \cite[\S 2.1]{dr}.

Let $\scA$ be a dg-category over a field. Then one can define the dg-category $\scA^{\prt}$ as follows. First, one defines the dg-category $\overline{\scA}$ with
\begin{itemize}
\item objects: formal expressions $A[n]$ where $A$ is an object of $A$ and $n \in \Z$;
\item morphisms: 
\[
\Hom^{\bullet}_{\overline{\scA}}(A[n],B[m]) \ = \ \Hom^{\bullet}_{\scA}(A,B)[m-n]
\]
as graded vector spaces, and the differential of an element $f \in \Hom^{\bullet}_{\scA}(A,B)$ viewed as an element of $\Hom^{\bullet}_{\overline{\scA}}(A[n],B[m])$ is given by
\[
d_{\overline{\scA}}(f)=(-1)^m \cdot d_{\scA}(f).
\]
Composition is defined in the natural way.
\end{itemize}
Note that this definition is chosen so as to mimic the relations for complexes of morphisms between complexes in an additive category, with the usual sign conventions.

Then, one defines the dg-category $\scA^{\prt}$ with
\begin{itemize}
\item objects: formal expressions 
\[
\left( \bigoplus_{i=1}^n \, C_i [d_i], q \right)
\]
where $n \geq 0$, $C_i$ is an object of $\scA$, $d_i \in \Z$, $q=(q_{ij})_{i,j=1 \cdots n}$ with $q_{ij} \in \Hom^{1}_{\overline{\scA}}(C_i[d_i],C_j[d_j])$ such that $dq+q^2=0$ and $q_{ij}=0$ for $i \geq j$;
\item morphisms: for objects $C=(\oplus_{i=1}^n C_i [d_i], q)$ and $C'=(\oplus_{i=1}^m C_i' [d_i'], q')$, the graded vector space $\Hom^{\bullet}_{\scA^{\prt}}(C,C')$ is the space of matrices $(f_{ij})_{i=1 \cdots m}^{j=1 \cdots n}$ such that $f_{ij} \in \Hom^{\bullet}(C_j[d_j],C'_i[d'_i])$. The differential is defined by the rule
\[
d_{\scA^{\prt}}(f)=(d_{\overline{\scA}}(f_{ij}))_{i,j} + q'f - (-1)^k fq
\]
if $f$ is homogeneous of degree $k$. Composition is defined by matrix multiplication.
\end{itemize}

Note that the assignment $\scA \mapsto \scA^{\prt}$ defines an endofunctor of the category of dg-categories and dg-functors between them.

Recall that for any dg-category $\scA$, the category $\mathsf{Ho}(\scA^{\prt})$ is always triangulated. (Here, for a dg-category $\scB$, we denote by $\mathsf{Ho}(\scB)$ its homotopy category.) For example, the cone of a closed morphism $f \in \Hom^0_{\scA^{\prt}}(C,C')$ is defined as the object
\[
\cone(f) = \left( \bigoplus_i C'_i[d'_i] \oplus \bigoplus_j C_j[d_j+1], \left(
\begin{array}{cc}
q' & f \\
0 & -q \\
\end{array}
\right) \right).
\]

One has a natural fully faithful inclusion $\scA \to \scA^{\prt}$ of dg-categories. One says that $\scA$ is \emph{pretriangulated} if the induced functor
\[
\mathsf{Ho}(\scA) \to \mathsf{Ho}(\scA^{\prt})
\]
is an equivalence of categories. If $\scA$ is pretriangulated, then by the remarks above $\mathsf{Ho}(\scA)$ has a canonical structure of triangulated category.

An \emph{enhanced triangulated category} is a triple $(\scD,\scE,\phi)$ where $\scD$ is a triangulated category, $\scE$ is a pretriangulated category, and $\phi$ is an equivalence of triangulated categories $\scD \cong \mathsf{Ho}(\scE)$.

Let $(\scD,\scE,\phi)$ be an enhanced triangulated category. Consider a set of objects $S=\{E_j, j \in J\}$ of $\scE$ (or equivalently of $\scD$). Let us denote by $\lan S \ran_{\scD}$ the full triangulated subcategory of $\scD$ generated by $S$, i.e.~the smallest strictly full triangulated subcategory of $\scD$ containing $S$. Our aim is to explain how one can concretely construct the category $\lan S \ran_{\scD}$ starting from the datum of $S$ and $\scE$.

Let $\scA$ denote the dg-subcategory of $\scE$ whose objects are the set $S$. This category is determined by the morphism complexes $\Hom^{\bullet}_{\scE}(E_i,E_j)$ and the composition maps. Consider the functor
\[
\Psi : \mathsf{Ho}(\scA^{\prt}) \, \to \, \mathsf{Ho}(\scE^{\prt}) \, \cong \, \mathsf{Ho}(\scE) \, \overset{\phi}{\cong} \, \scD
\]
induced by the inclusion $\scA \hookrightarrow \scE$. The following result is \cite[\S 4, Theorem 1]{bk}. We include the (very easy) proof for the reader's convenience.

\begin{prop}
\label{prop:subcategory-generated}

The functor $\Psi$ induces an equivalence of triangulated categories
\[
\mathsf{Ho}(\scA^{\prt}) \, \xrightarrow{\sim} \, \lan S \ran_{\scD}.
\]

\end{prop}

\begin{proof}
By construction, $\Psi$ is a fully faithful triangulated functor. By \cite[Proposition 4.10(b)]{bll}, the category $\mathsf{Ho}(\scA^{\prt})$ is generated, as a triangulated category, by the set of objects $S$. Hence $\Psi$ factors though $\lan S \ran_{\scD}$, which is its essential image.\end{proof}

Finally, we will need the following result, see \cite[Proposition 2.5]{dr}. Recall that a dg-functor $F:\scA \to \scB$ is a \emph{quasi-equivalence} if it induces an isomorphism
\[
H^{\bullet}(\Hom^{\bullet}_{\scA}(A,B)) \to H^{\bullet}(\Hom^{\bullet}_{\scB}(F(A),F(B)))
\]
for any $A,B$ in $\scA$ and if moreover $\mathsf{Ho}(F) : \mathsf{Ho}(\scA) \to \mathsf{Ho}(\scB)$ is essentially surjective.

\begin{prop}
\label{prop:qis-pretr}

If $F : \scA \to \scB$ is a quasi-equivalence, then the induced functor $F^{\prt} : \scA^{\prt} \to \scB^{\prt}$ is also a quasi-equivalence. \qed

\end{prop}

\subsection{Alternative definition of $F^G$}
\label{ss:alternative-definition}

Now we can give the alternative definition of the equivalence
\[
F^G: \cDb_{\mon{\GvO}}(\Gr_{\Gv}) \ \xrightarrow{\sim} \ \DGCGf(\cN_G).
\]
More precisely, we will construct an equivalence in the opposite direction. 

Consider the triangulated category $\DGMG(\C[\cN_G])$. This category has a natural enhancement. Indeed, let $\KPG(\C[\cN_G])$ be the sub-dg-category of the dg-category of $\C[\cN_G]$-dg-modules whose objects are the K-projective objects (in the sense of \cite[Definition 10.12.2.1]{bl}, adapted to our setting). Then this category is pretriangulated, and there is a natural equivalence of categories
\[
\mathsf{Ho}(\KPG(\C[\cN_G])) \ \xrightarrow{\sim} \ \DGMG(\C[\cN_G]).
\]
Hence we are in the setting of Proposition \ref{prop:subcategory-generated}. We let $\scA_G$ be the sub-dg-category of $\KPG(\C[\cN_G])$ whose objects are the $V \otimes \C[\cN_G]$ for $V \in \Rep(G)$. By Proposition \ref{prop:subcategory-generated}, there is a natural equivalence
\begin{equation}
\label{eqn:enhanced-1}
\mathsf{Ho}(\scA_G^{\prt}) \ \cong \ \DGCGf(\cN_G).
\end{equation}
Note that all morphism spaces in the dg-category $\scA_G$ have trivial differential.

Now, consider the constructible side. The category $\cDb_{\mon{\GvO}}(\Gr_{\Gv})$ is the subcategory of the triangulated category $\cDb_{\mon{\Iv}}(\Gr_{\Gv})$ generated by the objects $\cS_G(V)$ for $V$ in $\Rep(G)$. Moreover, the realization functor 
\[
\cDb \Perv_{\mon{\Iv}}(\Gr_{\Gv}) \ \to \ \cDb_{\mon{\Iv}}(\Gr_{\Gv})
\]
is an equivalence of categories, see \eqref{eqn:realization-functor-equivalence}. As above, the category $\cDb \Perv_{\mon{\Iv}}(\Gr_{\Gv})$ has a natural enhancement. Indeed, let $\Proj(\Gr_{\Gv})$ be the dg-category of bounded above complexes of projective pro-objects in $\Perv_{\mon{\Iv}}(\Gr_{\Gv})$ whose cohomology is bounded, and is in $\Perv_{\mon{\Iv}}(\Gr_{\Gv})$. Then this is a pretriangulated category, and the natural functor
\[
\mathsf{Ho}(\Proj(\Gr_{\Gv})) \to \cDb \Perv_{\mon{\Iv}}(\Gr_{\Gv})
\]
is an equivalence of triangulated categories. Hence we are again in the setting of Proposition \ref{prop:subcategory-generated}. Recall the resolution $P^{\bullet}$ constructed in \S \ref{ss:projective-resolution}. We denote by $\scB_G$ the dg-sub-category of $\Proj(\Gr_{\Gv})$ whose objects are the complexes $P^{\bullet} \star \cS_G(V)$, $V$ in $\Rep(G)$. Then, by Proposition \ref{prop:subcategory-generated}, there is a natural equivalence of triangulated categories
\begin{equation}
\label{eqn:enhanced-2}
\mathsf{Ho}(\scB_G^{\prt}) \ \cong \ \cDb_{\mon{\GvO}}(\Gr_{\Gv}).
\end{equation}

Now, we observe that the dg-categories $\scA_G$ and $\scB_G$ are quasi-equivalent. Indeed, first there is a natural bijection between objects of these categories. Then, consider $V,V'$ in $\Rep(G)$. Using Lemma \ref{lem:adjunction} and the same arguments as in the proof of Proposition \ref{prop:morphism-dg-algebras}, there are natural quasi-isomorphisms of complexes
\[
\Hom^{\bullet}_{\scB_G}(P^{\bullet} \star \cS_G(V), P^{\bullet} \star \cS_G(V')) \ \xrightarrow{\qis} \ \Ext^{\bullet}_{\cDb_{\mon{\GvO}}(\Gr_{\Gv})}(\cS_G(V),\cS_G(V'))
\]
which are compatible with composition, where the differential on the right-hand side is trivial. Then, by Proposition \ref{prop:isom-morphisms}, we have a natural isomorphism
\[
\Ext^{\bullet}_{\cDb_{\mon{\GvO}}(\Gr_{\Gv})}(\cS_G(V),\cS_G(V')) \ \cong \ \Ext^{\bullet}_{\DGCGf(\cN_G)}(V \otimes \C[\cN_G], V' \otimes \C[\cN_G]).
\]
Finally, we observe that the graded vector space on the right-hand side, endowed with the trivial differential, is the complex of morphisms in $\scA_G$ between $V \otimes \C[\cN_G]$ and $V' \otimes \C[\cN_G]$. Combining these morphisms provides the quasi-equivalence of dg-categories
\[
\scB_G \ \xrightarrow{\qis} \ \scA_G.
\]
Using Proposition \ref{prop:qis-pretr}, we deduce an equivalence of triangulated categories
\begin{equation}
\label{eqn:enhanced-3}
\mathsf{Ho}(\scB_G^{\prt}) \ \xrightarrow{\sim} \ \mathsf{Ho}(\scA_G^{\prt}).
\end{equation}

Combining equivalences \eqref{eqn:enhanced-1}, \eqref{eqn:enhanced-2} and \eqref{eqn:enhanced-3}, we obtain a second construction of the equivalence of Theorem \ref{thm:equivalence}. Property \eqref{eqn:compatibility-F-S} is obvious: the category $\Rep(G)$ acts on all the categories we have considered, in particular on the dg-categories $\scA_G$ and $\scB_G$, and all our equivalences commute with the action of $\Rep(G)$.

\begin{rmk}
It would be natural to expect that the equivalence constructed in this subsection is isomorphic to the one constructed in \S \ref{ss:construction-functor}. However, we were not able to prove this fact.
\end{rmk}

\subsection{Mixed version}
\label{ss:mixed-version-enhanced}

One can give a completely parallel proof of Theorem \ref{thm:equivalence-mix}. Namely, for any $V$ in $\Rep(G)$ one can construct a projective resolution in the category of pro-objects in $\Pervm_{\mon{\Iv}}(\Gr_{\Gv})$:
\[
\cdots \to P_V^{\mix,-2} \to P_V^{\mix,-1} \to P_V^0 \twoheadrightarrow \cS_G^0(V)
\]
by choosing for any $n \gg 0$ a minimal projective resolution in the abelian category $\Pervm_{\mon{\Iv}}(X_n)$ and then taking a projective limit over $n$. (See \S \ref{ss:projective-resolution} for details.)

Then one considers the dg-category $\scB_{G,\mix}$ whose objects are the resolutions $P_V^{\bullet} \lan j \ran$ for $V$ in $\Rep(G)$ and $j \in \Z$. Similarly, one defines the dg-category $\scA_{G,\mix}$ whose objects are the $G$-equivariant dgg-modules $V \otimes \C[\cN_G] \lan j \ran$ for $V$ in $\Rep(G)$ and $j \in \Z$. Then, one can construct a quasi-equivalence of dg-categories
\[
\scB_{G,\mix} \ \to \ \scA_{G,\mix},
\]
hence obtain an equivalence of triangulated categories $\mathsf{Ho}(\scB_{G,\mix}^{\prt}) \cong \mathsf{Ho}(\scA_{G,\mix}^{\prt})$. Combining with mixed analogs of equivalences \eqref{eqn:enhanced-1} and \eqref{eqn:enhanced-2}, and equivalence \eqref{eqn:equivalence-regrading}, one obtains an equivalence as in Theorem \ref{thm:equivalence-mix}. One can also show that it is possible to construct this equivalence and that of \S \ref{ss:alternative-definition} in such a way that the diagram of Theorem \ref{thm:equivalence-mix} commutes. Details are left to the reader.

In the remainder of this section, we prove that the equivalence constructed in this subsection, denoted by $'F_G^{\mix}$, is isomorphic to the equivalence $F^G_{\mix}$ constructed in \S \ref{ss:mixed-version}. This result will not be used in this paper; we only include it for completeness. However, the constructions of \S \ref{ss:orlov-category} will be used in Sections \ref{sect:hl-restriction} and \ref{sect:mix-convolution} below.

\subsection{An Orlov category}
\label{ss:orlov-category}

Consider the full subcategory $\Coh_{\free}^{G \times \Gm}(\cN_G)$ of the category $\Coh^{G \times \Gm}(\cN_G)$ which is generated under extensions by the objects $V \otimes_{\C} \cO_{\cN_G} \langle i \rangle$, for $V$ a simple $G$-module. Note that these objects, called \emph{free objects}, are
projective: since $G \times \Gm$ is a reductive group, the functor of
taking $G \times \Gm$-fixed points is exact.  Hence the objects of
$\Coh^{G \times \Gm}_{\free}(\cN_G)$ are in fact direct sums of free
objects. In particular, the indecomposable objects of the category $\Coh_{\free}^{G \times \Gm}(\cN_G)$ are exactly the objects $V \otimes_{\C} \cO_{\cN_G} \langle i \rangle$, for $V$ a simple $G$-module. Recall the notion of \emph{Orlov category} introduced in \cite[Definition 4.1]{ar}.

\begin{lem}
\label{lem:orlov-coherent}

The category $\Coh_{\free}^{G \times \Gm}(\cN_G)$, endowed with the function $\deg$ defined by
\[
\deg(V \otimes_{\C} \cO_{\cN_G} \langle i \rangle) = i,
\]
is an Orlov category. Moreover, there is a natural equivalence of triangulated categories
\begin{equation}
\label{eqn:equivalence-coherent}
\Kb\bigl( \Coh_{\free}^{G \times \Gm}(\cN_G) \bigr) \ \cong \ \cDb_{\free} \Coh^{G \times \Gm}(\cN_G).
\end{equation}

\end{lem}

\begin{proof}
First, it is clear that morphism spaces between objects of $\Coh_{\free}^{G \times \Gm}(\cN_G)$ are finite-dimensional. Then, for $V,V'$ simple $G$-modules and $i,j \in \Z$ we have
\[
\Hom_{\Coh^{G \times \Gm}_{\free}(\cN_G)}\bigl(V \otimes_{\C} \cO_{\cN_G} \langle i \rangle, V' \otimes_{\C} \cO_{\cN_G} \langle j \rangle \bigr) \cong (V' \otimes V^* \otimes \C[\cN_G] \langle j-i \rangle)^{G \times \Gm}.
\]
If $V \cong V'$ and $i=j$, this space has dimension $1$. And, in the general case, this space is zero unless $j-i \in 2\Z_{<0}$, or $j=i$ and $V \cong V'$. This proves the properties of an Orlov category.

Now we prove equivalence \eqref{eqn:equivalence-coherent}. The abelian category $\Coh^{G \times \Gm}(\cN_G)$ has enough projectives. Hence there is an equivalence of categories
\[
K^-\bigl( \Proj(\cN_G) \bigr) \ \cong \ \cD^- \Coh^{G \times \Gm}(\cN_G),
\]
where $\Proj(\cN_G)$ is the additive category of projective objects in $\Coh^{G \times \Gm}(\cN_G)$. We have remarked above that the objects $V \otimes_{\C} \cO_{\cN_G} \langle i \rangle$, $V$ a simple $G$-module, are projective. Hence the category $\cDb_{\free} \Coh^{G \times \Gm}(\cN_G)$ is equivalent to the full triangulated subcategory of $K^-\bigl( \Proj(\cN_G) \bigr)$ generated by these objects. It is clear that this subcategory is the essential image of the fully faithful functor
\[
\Kb\bigl( \Coh_{\free}^{G \times \Gm}(\cN_G) \bigr) \to K^-\bigl( \Proj(\cN_G) \bigr)
\]
induced by the inclusion $\Coh_{\free}^{G \times \Gm}(\cN_G) \hookrightarrow \Proj(\cN_G)$. This proves the equivalence.
\end{proof}

To simplify the task of working with the category $\Coh_{\free}^{G \times \Gm}(\cN_G)$, let us make the following remark. We have explained above that any object of $\Coh_{\free}^{G \times \Gm}(\cN_G)$ is isomorphic to a direct sum of objects of the form $V \otimes \cO_{\cN_G} \lan i \ran$. Hence this category is equivalent to the additive (Orlov) category $\sA_G$ with:
\begin{itemize}
\item objects: finite-dimensional graded $G$-modules $V=\bigoplus_{n \in \Z} V_n$;
\item morphisms:
\[
\Hom_{\sA_G}(V,V') = \Hom_{\Coh^{G \times \Gm}(\cN_G)}( \bigoplus_{n \in \Z} V_n \otimes \C[\cN_G]\lan n \ran, \bigoplus_{n \in \Z} V'_n \otimes \C[\cN_G]\lan n \ran).
\]
\end{itemize}

\subsection{Isomorphism of $F_G^{\mix}$ and $'F_G^{\mix}$}
\label{ss:isom-FG-FGmix}

Consider the functor
\[
{}'F_G^{\mix} \circ (F_G^{\mix})^{-1} : \cDb_{\free} \Coh^{G \times \Gm}(\cN_G) \to \cDb_{\free} \Coh^{G \times \Gm}(\cN_G).
\]
By Lemma \ref{lem:orlov-coherent}, one can consider this functor as a functor from $\Kb\bigl( \Coh_{\free}^{G \times \Gm}(\cN_G) \bigr)$ to itself. As such, this functor stabilizes the subcategory $\Coh_{\free}^{G \times \Gm}(\cN_G)$, and there is a natural isomorphism of functors
\[
\bigl( {}'F_G^{\mix} \circ (F_G^{\mix})^{-1} \bigr)_{|\Coh_{\free}^{G \times \Gm}(\cN_G)} \ \cong \ \id_{\Coh_{\free}^{G \times \Gm}(\cN_G)}.
\]
(Use Lemma \ref{lem:image-IC-mix} and the category $\sA_G$ introduced in \S \ref{ss:orlov-category}.) By \cite[Theorem 4.7]{ar}, we deduce that there exists an isomorphism of functors
\[
{}'F_G^{\mix} \circ (F_G^{\mix})^{-1} \ \cong \ \id_{\cDb_{\free} \Coh^{G \times \Gm}(\cN_G)},
\]
hence an isomorphism of functors
\[
{}'F_G^{\mix} \ \cong \ F_G^{\mix}.
\]

\section{Relation to \cite{abg}}
\label{sect:relation-ABG}

In this section we explain the relationship between Theorems \ref{thm:equivalence} and \ref{thm:equivalence-mix} and \cite[Theorems 9.1.4 and 9.4.3]{abg}. For simplicity, we only treat the non-mixed case. The mixed equivalences can be related similarly. (In this case, this result can also be proved ``abstractly'' using Orlov categories.) The equivalence ``$F_G$'' we consider in this section is the one constructed in Section \ref{sect:proof-1}. The results of this section are not used in the rest of the paper.


\subsection{Sheaves on $\wcN$ and multihomogeneous coordinate algebra}
\label{ss:multihomogenous-coord-alg}

Let $\cB_G:=G/B$ be the flag variety of $G$, and let $\wcN_G:=T^*\cB_G$ be its cotangent bundle. Recall that for every weight $\lambda \in \bX$ there is a natural line bundle $\cO_{\cB_G}(\lambda)$ on $\cB_G$, which is globally generated iff $\lambda$ is dominant. We denote by $\cO_{\wcN_G}(\lambda)$ the pullback of $\cO_{\cB_G}(\lambda)$ to $\wcN_G$. Set
\[
\mathbf{\Gamma}(\wcN_G) \ := \ \bigoplus_{\lambda \in \bX^+} \, \Gamma(\wcN_G,\, \cO_{\wcN_G}(\lambda)).
\]
This is a $G$-equivariant $\bX$-graded algebra, called the \emph{multihomogeneous coordinate algebra} of $\wcN_G$. We denote by $\Mod^G_{\bX}(\mathbf{\Gamma}(\wcN_G))$ the abelian category of $G$-equivariant $\bX$-graded $\mathbf{\Gamma}(\wcN_G)$-modules, and by $\QCoh^G(\wcN_G)$ the abelian category of $G$-equivariant quasi-coherent sheaves on $\wcN_G$. There is a natural functor
\[
\mathbf{\Gamma} : \left\{ 
\begin{array}{ccc}
\QCoh^G(\wcN_G) & \to & \Mod^G_{\bX}(\mathbf{\Gamma}(\wcN_G)) \\
\cM & \mapsto & \bigoplus_{\lambda \in \bX^+} \, \Gamma(\wcN_G, \cM \otimes_{\cO_{\wcN_G}} \cO_{\wcN_G}(\lambda))
\end{array} .
\right.
\]

Now consider the $G$-equivariant $\bX$-graded sheaf of algebras
\[
\bO_{\wcN_G} \, := \, \bigoplus_{\lambda \in \bX} \, \cO_{\wcN_G}(\lambda).
\]
We denote by $\QCoh^G_{\bX}(\wcN_G, \bO_{\wcN_G})$ the abelian category of $G$-equivariant $\bX$-graded sheaves of modules over the algebra $\bO_{\wcN_G}$, which are quasi-coherent over $\cO_{\wcN_G}$. There are natural adjoint functors
\[
\sfL : \left\{
\begin{array}{ccc}
\Mod^G_{\bX}(\mathbf{\Gamma}(\wcN_G)) & \to & \QCoh^G_{\bX}(\wcN_G, \bO_{\wcN_G}) \\
M & \mapsto & \bO_{\wcN_G} \otimes_{\mathbf{\Gamma}(\wcN_G)} M
\end{array}
\right.
\]
and
\[
\sfG : \left\{
\begin{array}{ccc}
\QCoh^G_{\bX}(\wcN_G, \bO_{\wcN_G}) & \to & \Mod^G_{\bX}(\mathbf{\Gamma}(\wcN_G)) \\
\cM & \mapsto & \Gamma(\wcN_G, \cM)
\end{array} .
\right.
\]
Here, $\sfG$ and $\sfL$ stand for ``global'' and ``local.'' Note that the inclusion $\mathbf{\Gamma}(\wcN_G) \subset \Gamma(\wcN_G, \bO_{\wcN_G})$ is strict.

There are also natural functors
\[
\mathsf{Ind}_{\wcN_G} : \left\{
\begin{array}{ccc}
\QCoh^G(\wcN_G) & \to & \QCoh^G_{\bX}(\wcN_G, \bO_{\wcN_G}) \\
\cM & \mapsto & \bO_{\wcN_G} \otimes_{\cO_{\wcN_G}} \cM
\end{array}
\right.
\]
and
\[
\mathsf{Res}_{\wcN_G} : \left\{
\begin{array}{ccc}
\QCoh^G_{\bX}(\wcN_G, \bO_{\wcN_G}) & \to & \QCoh^G(\wcN_G) \\
\cM & \mapsto & [\cM]_0
\end{array} ,
\right.
\]
where $[\cM]_{\lambda}$ is the component of $\cM$ of degree $\lambda \in \bX$.

Finally, we let $\mathsf{Tor}^G_{\bX}(\mathbf{\Gamma}(\wcN_G)) \subset \Mod^G_{\bX}(\mathbf{\Gamma}(\wcN_G))$ be the subcategory whose objects are inductive limits of objects $M$ such that there exists $\lambda \in \bX$ (depending on $M$) such that $[M]_{\mu}=0$ for $\mu \in \lambda + \bX^+$. We denote by
\[
\mathsf{Q} : \Mod^G_{\bX}(\mathbf{\Gamma}(\wcN_G)) \to \Mod^G_{\bX}(\mathbf{\Gamma}(\wcN_G)) / \mathsf{Tor}^G_{\bX}(\mathbf{\Gamma}(\wcN_G))
\]
the quotient functor.

The following result is a version of Serre's theorem on quasi-coherent sheaves on projective varieties. See \cite{av} for a similar result, whose proof can easily be adapted to our setting.

\begin{prop}
\label{prop:serre-thm}

\begin{enumerate}

\item 
The functor $\sfL$ is exact and vanishes on the subcategory $\mathsf{Tor}^G_{\bX}(\mathbf{\Gamma}(\wcN_G))$. The induced functor
\[
\sfL' : \Mod^G_{\bX}(\mathbf{\Gamma}(\wcN_G)) / \mathsf{Tor}^G_{\bX}(\mathbf{\Gamma}(\wcN_G)) \to \QCoh^G_{\bX}(\wcN_G, \bO_{\wcN_G})
\]
is an equivalence of abelian categories, with quasi-inverse $\mathsf{Q} \circ \sfG$.

\item
\label{it:prop-serre-thm-Ind-Res}
The functors $\mathsf{Ind}_{\wcN_G}$ and $\mathsf{Res}_{\wcN_G}$ are quasi-inverse equivalences of categories.

\item
There exists an isomorphism of functors $\mathsf{Q} \circ \mathbf{\Gamma} \ \cong \ \mathsf{Q} \circ \sfG \circ \mathsf{Ind}_{\wcN}$.\qed

\end{enumerate}

\end{prop}

\begin{rmk}
The functor $\mathsf{Res}_{\wcN_G} \circ \sfL$ is a non-$\C^{\times}$-equivariant version of the functor denoted by $\mathscr{F}$ in \cite[\S 8.8]{abg}.
\end{rmk}

In fact, we will not use Proposition \ref{prop:serre-thm} in the sequel. This proposition only serves as a motivation for the definition of the functor $\mathsf{Loc}$ below.

Consider the Springer resolution $\pi : \wcN_G \to \cN_G$, and the associated inverse image functor $\pi^* : \QCoh^G(\cN_G) \to \QCoh^G(\wcN_G)$. As $\cN_G$ is an affine variety, the global sections functor induces an equivalence
\[
\Gamma(\cN_G,-) : \QCoh^G(\cN_G) \to \Mod^G(\C[\cN_G])
\]
(where $\Mod^G(\C[\cN_G])$ is the category of $G$-equivariant $\C[\cN_G]$-modules).

\begin{prop}
\label{prop:inverse-image-multihomogeneous}

The following diagram commutes up to an isomorphism of functors:
\[
\xymatrix@C=1.7cm{
\QCoh^G(\cN_G) \ar[r]^-{\Gamma(\cN_G,-)} \ar[d]_-{\pi^*} & \Mod^G(\C[\cN_G]) \ar[r]^-{\mathbf{\Gamma}(\wcN_G) \otimes_{\C[\cN_G]} -} & \Mod^G_{\bX}(\mathbf{\Gamma}(\wcN_G)) \ar[d]^-{\sfL} \\
\QCoh^G(\wcN_G) \ar[rr]^-{\mathsf{Ind}_{\wcN_G}} & & \QCoh^G_{\bX}(\wcN_G, \bO_{\wcN_G}).
}
\]

\end{prop}

\begin{proof}
This follows immediately from the fact that there exists a natural isomorphism of functors
\[
\pi^* \ \cong \ \cO_{\wcN_G} \otimes_{\C[\cN_G]} \Gamma(\cN_G,-)
\]
and transitivity of the tensor product.
\end{proof}

Now we consider dg-analogues of these constructions. Let $\DGC^G(\wcN_G)$ be the subcategory of the the derived category of $G$-equivariant quasi-coherent sheaves of dg-modules over the sheaf of dg-algebras $\mathrm{S}_{\cO_{\cB_G}}(\mathcal{T}_{\cB_G})$ on $\cB_G$ (where the tangent sheaf $\mathcal{T}_{\cB_G}$ is in degree $2$, and the differential is trivial) whose objects have their cohomology locally finitely generated over $\mathrm{S}_{\cO_{\cB_G}}(\mathcal{T}_{\cB_G})$.

Consider also the $\bX$-graded $G$-equivariant quasi-coherent sheaf of dg-algebras on $\cB_G$
\[
\wbO_{\wcN_G} \ := \ \bigoplus_{\lambda \in \bX} \, \mathrm{S}_{\cO_{\cB_G}}(\mathcal{T}_{\cB_G}) \otimes_{\cO_{\cB_G}} \cO_{\cB_G}(\lambda).
\]
Here the multiplication is the natural one, $\mathcal{T}_{\cB_G}$ is in degree $2$ is each direct summand, and the differential is trivial. (This definition is chosen so that, if we forget about the grading, $\wbO_{\wcN_G}$ is the direct image to $\cB_G$ of $\bO_{\wcN_G}$.) We denote by $\DGC^G_{\bX}(\wcN_G,\wbO_{\wcN_G})$ the subcategory of the derived category of $\bX$-graded $G$-equivariant quasi-coherent sheaves of dg-modules over $\wbO_{\wcN_G}$ whose objects have their cohomology locally finitely generated over $\wbO_{\wcN_G}$. By the same arguments as for Proposition \ref{prop:serre-thm}\eqref{it:prop-serre-thm-Ind-Res}, there is a natural equivalence of triangulated categories
\begin{equation}
\label{eqn:equivalence-Ind-Res}
\DGC^G(\wcN_G) \ \cong \ \DGC^G_{\bX}(\wcN_G,\wbO_{\wcN_G}).
\end{equation}

Finally, we consider $\mathbf{\Gamma}(\wcN_G)$ as an $\bX$-graded $G$-equivariant dg-algebra with trivial differential and the $\bX \times \Z$-grading chosen so that the inclusion $\mathbf{\Gamma}(\wcN_G) \subset \Gamma(\cB_G, \wbO_{\wcN_G})$ is graded. We denote by $\DGM^{G}_{\mathrm{fg},\bX}(\mathbf{\Gamma}(\wcN_G))$ the subcategory of the derived category of $\bX$-graded $G$-equivariant dg-modules over this dg-algebra whose objects have their cohomology finitely generated over $\mathbf{\Gamma}(\wcN_G)$. There is also a natural functor
\[
\left\{ 
\begin{array}{ccc}
\DGM^{G}_{\mathrm{fg},\bX}(\mathbf{\Gamma}(\wcN_G)) & \to & \DGC^G_{\bX}(\wcN_G,\wbO_{\wcN_G}) \\
M & \mapsto & \wbO_{\wcN_G} \, \lotimes_{\mathbf{\Gamma}(\wcN_G)} \, M
\end{array}
\right. .
\]
We denote by
\begin{equation}
\label{eqn:def-Loc}
\mathsf{Loc} : \DGM^{G}_{\mathrm{fg},\bX}(\mathbf{\Gamma}(\wcN_G)) \to \DGC^G(\wcN_G)
\end{equation}
the composition of this functor with the equivalence \eqref{eqn:equivalence-Ind-Res}.

The morphism $\pi$ defined above induces a (derived) inverse image functor
\[
\pi^* : \DGCGf(\cN_G) \to \DGC^G(\wcN_G).
\]
There is also a functor
\[
\mathbf{\Gamma}(\wcN_G) \, \lotimes_{\C[\cN_G]} \, (-) : \DGCGf(\cN_G) \to \DGM^{G}_{\mathrm{fg},\bX}(\mathbf{\Gamma}(\wcN_G))
\]
The same proof as that of Proposition \ref{prop:inverse-image-multihomogeneous} gives the following result.

\begin{prop}
\label{prop:inverse-image-dg}

The diagram
\[
\xymatrix{
& \DGCGf(\cN_G) \ar[ld]|-{\mathbf{\Gamma}(\wcN_G) \, \lotimes_{\C[\cN_G]} \, (-)} \ar[rd]^-{\pi^*} & \\
\DGM^{G}_{\mathrm{fg},\bX}(\mathbf{\Gamma}(\wcN_G)) \ar[rr]^-{\mathsf{Loc}} & & \DGC^G(\wcN_G). 
}
\]
commutes up to an isomorphism of functors.\qed

\end{prop}

\subsection{Reminder on \cite{abg}}
\label{ss:reminder-abg}

For the remainder of Section \ref{sect:relation-ABG} we assume that $G$ is semisimple of adjoint type. By \cite[Theorem 9.1.4]{abg}, there exists an equivalence of categories
\begin{equation}
\label{eqn:equiv-abg}
\mathbf{F}_G : \cDb_{\mon{\Iv}}(\Gr_{\Gv}) \ \xrightarrow{\sim} \ \DGC^G(\wcN_G).
\end{equation}
Let us recall how this equivalence is constructed.

First, we denote by $\cD_{\prj}(\Gr_{\Gv})$ the subcategory of the homotopy category of bounded above complexes of projective pro-objects in $\Perv_{\mon{\Iv}}(\Gr_{\Gv})$ whose objects $C^{\bullet}$ satisfy the following conditions: 
\begin{itemize}
\item $H^i(C) \in \Perv_{\mon{\Iv}}(\Gr_{\Gv})$ for any $i \in \Z$;
\item $H^i(C)=0$ for $i \ll 0$.
\end{itemize}
There exists a natural functor from $\cD_{\prj}(\Gr_{\Gv})$ to the derived category of the abelian category of pro-objects in $\Perv_{\mon{\Iv}}(\Gr_{\Gv})$. The essential image of this functor is the subcategory whose objects have their total cohomology in $\Perv_{\mon{\Iv}}(\Gr_{\Gv})$. By \cite[Theorem 15.3.1(i)]{ks2}, the latter subcategory is equivalent to $\cDb \Perv_{\mon{\Iv}}(\Gr_{\Gv})$. Hence one obtains an equivalence of triangulated categories
\begin{equation}
\label{eqn:equivalence-proj}
\Upsilon : \cD_{\prj}(\Gr_{\Gv}) \xrightarrow{\sim} \cDb \Perv_{\mon{\Iv}}(\Gr_{\Gv}) \cong \cDb_{\mon{\Iv}}(\Gr_{\Gv}).
\end{equation}

As in \S \ref{ss:Satake-mix}, let $\Fl_{\Gv}:=\Gv(\fK)/\Iv$ be the affine flag variety. For every $\lambda \in \bX$ one can define a \emph{Wakimoto sheaf} $\cW_{\lambda} \in \Perv_{\eq{\Iv}}(\Fl_{\Gv})$, see \cite[\S 8.3]{abg}. Recall that there is a natural convolution product on $\cDb_{\eq{\Iv}}(\Fl_{\Gv})$, which we denote by $\star^{\Iv}$. For any $\lambda,\mu \in \bX$ there exists a canonical isomorphism
\begin{equation}
\label{eqn:wakimoto}
\cW_{\lambda} \star^{\Iv} \cW_{\mu} \ \cong \ \cW_{\lambda+\mu}
\end{equation}
(see \cite[Corollary 8.3.2]{abg} or \cite[Corollary 1]{ab}).

There exists also a (convolution) action of the category $\cDb_{\eq{\Iv}}(\Fl_{\Gv})$ on the category $\cDb_{\eq{\Iv}}(\Gr_{\Gv})$. In \cite[\S 8.9]{abg}, the authors explain how to ``extend'' the convolution with $\cW_{\lambda}$ to a functor on $\cDb_{\mon{\Iv}}(\Gr_{\Gv})$. More precisely, they construct for every $\lambda \in \bX^+$ an equivalence of categories $C_{\lambda}: \cDb_{\mon{\Iv}}(\Gr_{\Gv}) \to \cDb_{\mon{\Iv}}(\Gr_{\Gv})$ such that the following diagram commutes up to isomorphism:
\[
\xymatrix@C=1.5cm{
\cDb_{\eq{\Iv}}(\Gr_{\Gv}) \ar[d]_-{\For} \ar[r]^-{\cW_{\lambda} \star (-)} & \cDb_{\eq{\Iv}}(\Gr_{\Gv}) \ar[d]^-{\For} \\
\cDb_{\mon{\Iv}}(\Gr_{\Gv}) \ar[r]^{C_{\lambda}} & \cDb_{\mon{\Iv}}(\Gr_{\Gv}).
}
\]
Moreover, for any $\lambda,\mu \in \bX^+$ there exists a canonical isomorphism of functors
\[
C_{\lambda} \circ C_{\mu} \ \cong \ C_{\lambda+\mu}
\]
compatible (in the obvious sense) with isomorphism \eqref{eqn:wakimoto}. For these reasons, for any $M$ in $\cDb_{\mon{\Iv}}(\Gr_{\Gv})$ one can set  $\cW_{\lambda} \star M := C_{\lambda}(M)$. (Note that this notation may be misleading, as this construction is \emph{not} functorial in the left factor, but only in the right one.)

Consider the $\bX \times \Z$-graded vector space
\[
\Ext^{\bullet}_{\mon{\Iv}}(1_G, \cW_{\bX^+} \star \cR_G),
\]
whose $(\lambda,i)$-component is zero if $\lambda \notin \bX^+$, and otherwise is
\[
\varinjlim_{k \geq 0} \, \Hom_{\cDb_{\mon{\Iv}}(\Gr_{\Gv})}(1_G,\cW_{\lambda} \star \cR_{G,k}[i]).
\]
This graded vector space can be endowed with an algebra structure, where for $\xi \in \Ext^i_{\mon{\Iv}}(1_G,\cW_{\lambda} \star \cR_G)$ and $\zeta \in \Ext^j_{\mon{\Iv}}(1_G,\cW_{\mu} \star \cR_G)$, the product $\xi \cdot \zeta$ is by definition the morphism
\begin{multline*}
1_G \xrightarrow{\zeta} \cW_{\mu} \star \cR_G[j] \cong \cW_{\mu} \star 1_G \star \cR_G[j] \xrightarrow{\cW_{\lambda} \star \xi \star \cR_G[j]} \cW_{\mu} \star \cW_{\lambda} \star \cR_G \star \cR_G[i+j] \\ \cong \cW_{\lambda+\mu} \star \cR_G \star \cR_G[i+j] \xrightarrow{\cW_{\lambda+\mu} \star \sm[i+j]} \cW_{\lambda+\mu} \star \cR_G [i+j].
\end{multline*}
The action of $G$ on $\cR_G$ also induces an action on this algebra, which is compatible with the product. Note that the algebra $\Ext^{\bullet}_{\mon{\GvO}}(1_G,\cR_G)$ of \S \ref{ss:Ext-algebra} is the component of the algebra $\Ext^{\bullet}_{\mon{\Iv}}(1_G, \cW_{\bX^+} \star \cR_G)$ of weight $0 \in \bX$.

By \cite[Theorem 8.5.2]{abg}, there exists an isomorphism of $G$-equivariant $\bX \times \Z$-graded algebras
\begin{equation}
\label{eqn:isom-abg}
\Ext^{\bullet}_{\mon{\Iv}}(1_G, \cW_{\bX^+} \star \cR_G) \ \cong \ \mathbf{\Gamma}(\wcN_G),
\end{equation}
where the $\Z$-grading on the right-hand side is as in \S \ref{ss:multihomogenous-coord-alg}.

The next step is a formality result for some dg-algebra. Consider the resolution $P^{\bullet}$ as in \S \ref{ss:projective-resolution}, and form the $\bX$-graded dg-algebra
\[
\sE_{\bX}^{\bullet}(1_G,\cR_G) \ := \ \bigoplus_{\genfrac{}{}{0pt}{}{\lambda \in \bX^+}{i\in \Z}} \, \varinjlim_{k \geq 0} \, \Hom^i(P^{\bullet},\cW_{\lambda} \star P^{\bullet} \star \cR_{G,k}).
\]
Here the differential is the natural one, and the product is defined as follows. If $\xi \in \Hom^i(P^{\bullet},\cW_{\lambda} \star P^{\bullet} \star \cR_{G,k})$ and $\zeta \in \Hom^j(P^{\bullet},\cW_{\mu} \star P^{\bullet} \star \cR_{G,l})$, then the product $\xi \cdot \zeta$ is the composition
\begin{multline*}
P^{\bullet} \xrightarrow{\zeta} \cW_{\mu} \star P^{\bullet} \star \cR_{G,l}[j] \ \xrightarrow{\cW_{\mu} \star \xi \star \cR_{G,l}[j]} \cW_{\mu} \star \cW_{\lambda} \star P^{\bullet} \star \cR_{G,k} \star \cR_{G,l}[i+j] \cong \\
\cW_{\lambda+\mu} \star P^{\bullet} \star \cR_{G,k} \star \cR_{G,l}[i+j] \xrightarrow{\cW_{\lambda+\mu} \star P^{\bullet} \star \sm_{k,l}[i+j]} \cW_{\lambda+\mu} \star P^{\bullet} \star \cR_{G,k+l} [i+j].
\end{multline*}
Note that the dg-algebra $\sE^{\bullet}(1_G,\cR_G)$ of \S \ref{ss:formality} is the component of $\sE_{\bX}^{\bullet}(1_G,\cR_G)$ of weight $0 \in \bX$.

By the same arguments as in \S \ref{ss:lf-subalgebra}, one can construct a sub-dg-algebra\footnote{This construction is not performed in \cite{abg}. It is necessary for the arguments there to work, however, even in the non-mixed case.} 
\begin{equation}
\label{eqn:lf-subalgebra-X}
\sE_{\bX}^{\bullet}(1_G,\cR_G)^{\lf} \subset \sE_{\bX}^{\bullet}(1_G,\cR_G)
\end{equation}
which is ``the locally finite part for the action of the Frobenius'', such that the inclusion $\sE_{\bX}^{\bullet}(1_G,\cR_G)^{\lf} \hookrightarrow \sE_{\bX}^{\bullet}(1_G,\cR_G)$ is a quasi-isomorphism. Then, using pointwise purity of simple $\GvO$-equivariant perverse sheaves on $\Gr_{\Gv}$ (see \cite{kl,g1}), one can construct an injection and a surjection of $\bX$-graded, $G$-equivariant dg-algebras
\begin{equation}
\label{eqn:qis-pointwise-purity}
\xymatrix{
\sE_{\bX}^{\bullet}(1_G,\cR_G)^{\lf} &
\ \sE_{\bX}^{\bullet}(1_G,\cR_G)^{\lf}_< \ar@{_{(}->}[l] \ar@{->>}[r] & \Ext^{\bullet}_{\mon{\Iv}}(1_G, \cW_{\bX^+} \star \cR_G)
}
\end{equation}
which are both quasi-isomorphisms\footnote{Note that these arguments prove in particular the formality of the dg-algebra $\sE^{\bullet}(1_G,\cR_G)$ of \S \ref{ss:formality}. We believe our argument in the proof of Proposition \ref{prop:morphism-dg-algebras} is more elementary.} (see \cite[\S 9.5]{abg}).

Finally we can review the construction of the equivalence $\mathbf{F}_G$. It is defined as the composition
\begin{multline*}
\cDb_{\mon{\Iv}}(\Gr_{\Gv}) \xrightarrow{\Upsilon^{-1}} \cD_{\prj}(\Gr_{\Gv}) \xrightarrow{(1)} \DGM^G_{\mathrm{fg},\bX}(\sE_{\bX}^{\bullet}(1_G,\cR_G)^{\mathrm{op}}) \\
\xrightarrow{(2)} \DGM^G_{\mathrm{fg},\bX}(\sE_{\bX}^{\bullet}(1_G,\cR_G)^{\lf,\mathrm{op}}) \xrightarrow{(3)} \DGM^G_{\mathrm{fg},\bX}(\sE_{\bX}^{\bullet}(1_G,\cR_G)^{\lf,\mathrm{op}}_<) \\
\xrightarrow{(4)} \DGM^G_{\mathrm{fg},\bX}(\Ext^{\bullet}_{\mon{\Iv}}(1_G, \cW_{\bX^+} \star \cR_G)) \xrightarrow{(5)} \DGMG_{\mathrm{fg},\bX}(\mathbf{\Gamma}(\wcN_G)) \\
\xrightarrow{\mathsf{Loc}} \DGC^G(\wcN_G).
\end{multline*}
Here, the categories $\DGM^G_{\mathrm{fg},\bX}(-)$ are defined as for $\mathbf{\Gamma}(\wcN_G)$ in \S \ref{ss:multihomogenous-coord-alg}. The functor $(1)$ is defined by a formula very similar to that for the functor $\sE^{\bullet}(1_G,(-) \star \cR_G)$ defined in \S \ref{ss:construction-functor}, adding Wakimoto sheaves to the picture. (This functor is defined at the level of homotopy categories; the functor $(1)$ is the composition with the natural functor from the homotopy category to the derived category.) The functor $(2)$ is the equivalence induced by quasi-isomorphism \eqref{eqn:lf-subalgebra-X}. The functors $(3)$ and $(4)$ are similarly induced by quasi-isomorphisms \eqref{eqn:qis-pointwise-purity}. The functor $(5)$ is the equivalence induced by isomorphism \eqref{eqn:isom-abg}. Finally, the equivalence $\Upsilon$ is defined in \eqref{eqn:equivalence-proj}, and the functor $\mathsf{Loc}$ in \eqref{eqn:def-Loc}.

\subsection{Compatibility}
\label{ss:compatibility-abg}

The main result of this section is the following.

\begin{prop}

The following diagram is commutative up to an isomorphism of functors:
\[
\xymatrix@C=2cm{
\cDb_{\mon{\GvO}}(\Gr_{\Gv}) \ar[r]^-{F_G}_-{{\rm Thm.~\ref{thm:equivalence}}} \ar@{^{(}->}[d]_-{i} & \DGCGf(\cN_G) \ar[d]^-{\pi^*} \\
\cDb_{\mon{\Iv}}(\Gr_{\Gv}) \ar[r]^-{\mathbf{F}_G}_-{\eqref{eqn:equiv-abg}} & \DGC^G(\wcN_G).
}
\]

\end{prop}

\begin{proof}
First, using Proposition \ref{prop:inverse-image-dg} and arguments similar to those of the proof of Lemma \ref{lem:diagram-FG-FGmix}, one checks that the following diagram commutes up to an isomorphism of functors:
\[
\xymatrix@C=3cm@R=1.2cm{
\DGMG_{\fr}(\sE^{\bullet}(1_G,\cR_G)) \ar[d]|-{\sE^{\bullet}_{\bX}(1_G,\cR_G) \, \lotimes_{\sE^{\bullet}(1_G,\cR_G)} \, (-)} \ar[r] & \DGCGf(\cN_G) \ar[d]^-{\pi^*} \\
\DGMG_{\mathrm{fg}, \bX}(\sE^{\bullet}_{\bX}(1_G,\cR_G)) \ar[r] & \DGC^G(\wcN_G).
}
\]
Here, $\DGMG_{\fr}(\sE^{\bullet}(1_G,\cR_G))$ is the subcategory of $\DGMG(\sE^{\bullet}(1_G,\cR_G))$ generated by objects of the form $\sE^{\bullet}(1_G,\cR_G) \otimes_{\C} V$ for $V$ in $\Rep(G)$, the functor on the top line is the one appearing in the definition of $F_G$, and the functor on the bottom line is the one appearing in the definition of $\mathbf{F}_G$.

We define $\cH \DGMG(\sE^{\bullet}(1_G,\cR_G))$ and $\cH \DGMG_{\bX}(\sE^{\bullet}_{\bX}(1_G,\cR_G))$ as the homotopy categories whose associated derived categories are $\DGMG(\sE^{\bullet}(1_G,\cR_G))$ and $\DGMG_{\bX}(\sE^{\bullet}_{\bX}(1_G,\cR_G))$. As $\sE^{\bullet}_{\bX}(1_G,\cR_G)$ is K-flat as an $\sE^{\bullet}(1_G,\cR_G)$-dg-module, the following diagram commutes:
\[
\xymatrix@C=0.7cm@R=1.2cm{
\cH \DGMG(\sE^{\bullet}(1_G,\cR_G)) \ar[r] \ar[d]|-{\sE^{\bullet}_{\bX}(1_G,\cR_G) \otimes_{\sE^{\bullet}(1_G,\cR_G)} (-)} & \DGMG(\sE^{\bullet}(1_G,\cR_G)) \ar[d]|-{\sE^{\bullet}_{\bX}(1_G,\cR_G) \, \lotimes_{\sE^{\bullet}(1_G,\cR_G)} \, (-)} \\
\cH \DGMG_{\bX}(\sE^{\bullet}_{\bX}(1_G,\cR_G)) \ar[r] & \DGMG_{\bX}(\sE^{\bullet}_{\bX}(1_G,\cR_G)).
}
\]

Recall that the functor $F_G$ is constructed using an exact functor from the category $\cCb \Perv_{\mon{\Iv}}(\Gr_{\Gv})$ to the category of $G$-equivariant dg-modules over the dg-algebra $\sE^{\bullet}(1_G,\cR_G)$. Hence this functor factors through the composition
\[
\cDb_{\mon{\GvO}}(\Gr_{\Gv}) \xrightarrow{i} \cDb_{\mon{\Iv}}(\Gr_{\Gv}) \xrightarrow{\Upsilon^{-1}} \cD_{\prj}(\Gr_{\Gv}) \xrightarrow{H_G} \cH \DGMG(\sE^{\bullet}(1_G,\cR_G))
\]
which we denote by $I_G$. Here, the functor $H_G$ sends a complex $Q^{\bullet}$ to the dg-module
\[
\Hom^{\bullet}(P^{\bullet},Q^{\bullet} \star \cR_G)
\]
(with obvious notation).

Now, consider the following diagram:
\[
\xymatrix@R=1.2cm{
\cDb_{\mon{\GvO}}(\Gr_{\Gv}) \ar@{^{(}->}[d]_-{\Upsilon^{-1} \circ i} \ar[r]^-{I_G} & \cH \DGMG(\sE^{\bullet}(1_G,\cR_G)) \ar[d]|-{\sE^{\bullet}_{\bX}(1_G,\cR_G) \otimes_{\sE^{\bullet}(1_G,\cR_G)} (-)} \\
\cD_{\prj}(\Gr_{\Gv}) \ar[r]^-{H_G^{\bX}} & \cH \DGMG_{\bX}(\sE^{\bullet}_{\bX}(1_G,\cR_G)),
}
\]
where the functor $H_G^{\bX}$ is the analogue of $H_G$ which appears in the definition of $\mathbf{F}_G$. More precisely, $H_G^{\bX}$ sends a complex $Q^{\bullet}$ to the $\bX$-graded dg-module
\[
\bigoplus_{\lambda \in \bX^+} \, \Hom^{\bullet}(P^{\bullet},\cW_{\lambda} \star Q^{\bullet} \star \cR_G).
\]
There exists a natural morphism of functors
\[
\sE^{\bullet}_{\bX}(1_G,\cR_G) \otimes_{\sE^{\bullet}(1_G,\cR_G)} H_G(-) \to H_G^{\bX}(-)
\]
defined by the natural map
\[
\Hom^i(P^{\bullet},\cW_{\lambda} \star P^{\bullet} \star \cR_G) \otimes_{\C} \Hom^j(P^{\bullet},Q^{\bullet} \star \cR_G) \to \Hom^{i+j}(P^{\bullet},\cW_{\lambda} \star Q^{\bullet} \star \cR_G).
\]
Composing with $\Upsilon^{-1} \circ i$ on the right, one obtains a morphism of functors
\[
\sE^{\bullet}_{\bX}(1_G,\cR_G) \otimes_{\sE^{\bullet}(1_G,\cR_G)} I_G(-) \to H_G^{\bX} \circ \Upsilon^{-1} \circ i (-).
\]

Using the diagram considered earlier in this proof, we obtain a morphism a functors
\[
\nu : \pi^* \circ F_G \to \mathbf{F}_G \circ i.
\]
To prove that $\nu$ is an isomorphism, it is enough to check that for any $V$ in $\Rep(G)$, $\nu(\cS_G(V))$ is an isomorphism. However, there are natural isomorphisms
\[
F_G(\cS_G(V)) \ \cong \ V \otimes_{\C} \C[\cN_G], \quad \mathbf{F}_G(\cS_G(V)) \ \cong \ V \otimes_{\C} \cO_{\wcN_G}
\]
(see Lemma \ref{lem:image-IC} and \cite[Proposition 9.8.1]{abg}), and this claim is obvious.
\end{proof}

\section{Hyperbolic localization and restriction}
\label{sect:hl-restriction}

\subsection{Reminder on the Brylinski--Kostant filtration}
\label{ss:reminder-BK-filtration}

Recall the regular nilpotent element $e_G \in \fg$ introduced in \S \ref{ss:reminder-cohomology}. For any $G$-module $V$, and any subspace $U \subset V$, the Brylinski--Kostant filtration $\FBK_{\bullet} U$ on $U$ associated to $e_G$ (introduced and studied in particular in \cite{bry}) is by definition given by
\[
\FBK_i(U) \ = \ U \cap \ker(e_G^{i+1} : V \to V),
\]
where $e_G$ acts on $V$ via the differential of the $G$-action. In particular, this way we get a filtration on every $T$-weight space $V(\lambda)$ of $V$ ($\lambda \in \bX$). Let us recall a geometric construction of this filtration, due to Ginzburg. For any $\mu$ in $\bX$, we let $\fT_{\mu}$ be the $\Uv^-(\fK)$-orbit through $L_{\mu}$, where $\Uv^-$ is the unipotent radical of the Borel subgroup opposite to $\Bv$ with respect to $\Tv$.

Fix a $G$-module $V$, and a weight $\lambda \in \bX$. Then, by construction of the torus $T$ and \cite[Theorem 3.5]{mv} we have natural isomorphisms
\[
V(\lambda) \ \cong \ H^{\bullet}_c(\fS_{\lambda},\cS_G(V)) \ \cong \ H^{\bullet}_{\fT_{\lambda}}(\cS_G(V)),
\]
and both cohomology groups are concentrated in degree $\lan \lambda,2\rhov\ran$. Let $t_{\lambda} : \fT_{\lambda} \hookrightarrow \Gr_{\Gv}$ be the inclusion. Then by definition we have 
\[
H^{\bullet}_{\fT_{\lambda}}(\cS_G(V)) \ = \ H^{\bullet}(\fT_{\lambda}, t_{\lambda}^! \cS_G(V)).
\]

Now, consider the $\Tv$-action on $\Gr_{\Gv}$ by left multiplication. Recall that we have a natural isomorphism of graded algebras
\[
H^{\bullet}_{\Tv}(\pt) \ \cong \ \mathrm{S}(\ftv^*) \ \cong \ \mathrm{S}(\ft),
\]
where $\ftv^*$ is in degree $2$. In particular, any point $h \in \ftv$ defines a character of the algebra $H^{\bullet}_{\Tv}(\pt)$. We denote by $\C_h$ the corresponding one-dimensional module. The only $\Tv$-fixed point in $\fT_{\lambda}$ is $\{L_{\lambda}\}$. Hence, by the localization theorem in equivariant cohomology, the morphism
\begin{equation}
\label{eqn:localization-map}
H^{\bullet}_{\Tv}(i_{\lambda}^! \cS_G(V)) \ \to \ H^{\bullet}_{\Tv}(\fT_{\lambda}, t_{\lambda}^! \cS_G(V))
\end{equation}
induced by the adjunction for the inclusion $\{L_{\lambda}\} \hookrightarrow \fT_{\lambda}$ becomes an isomorphism after inverting all $\alv \in \Rv \subset \ftv^*$, or equivalently induces an isomorphism
\begin{equation}
\label{eqn:localization-map-specialized}
H^{\bullet}_{\Tv}(i_{\lambda}^! \cS_G(V)) \otimes_{H^{\bullet}_{\Tv}(\pt)} \C_h \ \xrightarrow{\sim} \ H^{\bullet}_{\C^{\times}}(\fT_{\lambda}, t_{\lambda}^! \cS_G(V)) \otimes_{H^{\bullet}_{\Tv}(\pt)} \C_h.
\end{equation}
for any $h \in \ftv \smallsetminus \cup_{\alv \in \Rv} \ker(\alv)$. (See Remark \ref{rmk:localization-thm} below for comments.)

By \cite[Equation (8.3.3)]{g2}, there is a (Leray) spectral sequence which computes $H^{\bullet}_{\Tv}(\fT_{\lambda}, t_{\lambda}^! \cS_G(V))$ and with $E_2$-term 
\[
E_2^{p,q} \ = \ H^p_{\Tv}(\pt) \otimes_{\C} H^{q}(\fT_{\lambda}, t_{\lambda}^! \cS_G(V)).
\]
As recalled above, the cohomology $H^{\bullet}(\fT_{\lambda}, t_{\lambda}^! \cS_G(V))$ is concentrated in one degree (in particular in degrees of constant parity). Hence this spectral sequence degenerates, and $H^{\bullet}_{\Tv}(\fT_{\lambda}, t_{\lambda}^! \cS_G(V))$ is a free $H^{\bullet}_{\Tv}(\pt)$-module, with a canonical isomorphism of graded vector spaces
\[
H^{\bullet}_{\Tv}(\fT_{\lambda}, t_{\lambda}^! \cS_G(V)) \otimes_{H^{\bullet}_{\Tv}(\pt)} \C_0 \ \cong \ H^{\bullet}(\fT_{\lambda}, t_{\lambda}^! \cS_G(V))
\]
In particular, as the right-hand side is concentrated in degree $\lan \lambda,2\rhov \ran$, it follows that the lowest non-zero degree in $H^{\bullet}_{\Tv}(\fT_{\lambda}, t_{\lambda}^! \cS_G(V))$ is $\lan \lambda,2\rhov \ran$, and that we have a natural isomorphism (induced by forgetting the equivariance)
\[
H^{\lan \lambda,2\rhov \ran}_{\Tv}(\fT_{\lambda}, t_{\lambda}^! \cS_G(V)) \ \cong \ H^{\lan \lambda,2\rhov \ran}(\fT_{\lambda}, t_{\lambda}^! \cS_G(V)).
\]
Hence there is a canonical morphism
\begin{equation}
\label{eqn:isom-equivariant-cohomology}
\bigl( H^{\lan \lambda,2\rhov\ran}(\fT_{\lambda}, t_{\lambda}^! \cS_G(V)) \otimes_{\C} H^{\bullet}_{\Tv}(\pt) \bigr) [-\lan \lambda, 2\rhov \ran] \ \to \ H^{\bullet}_{\Tv}(\fT_{\lambda}, t_{\lambda}^! \cS_G(V)),
\end{equation}
which is necessarily an isomorphism. (See also \cite[proof of Lemma 2.2]{yz} for similar arguments.) In particular, this way we get for any $h \in \ftv$ a \emph{canonical} isomorphism
\begin{equation}
\label{eqn:isom-cohomology-T}
H^{\lan \lambda,2\rhov\ran}(\fT_{\lambda}, t_{\lambda}^! \cS_G(V)) \ \xrightarrow{\sim} \ H^{\bullet}_{\Tv}(\fT_{\lambda}, t_{\lambda}^! \cS_G(V)) \otimes_{H^{\bullet}_{\Tv}(\pt)} \C_h.
\end{equation}

Now we come back to the morphism \eqref{eqn:localization-map-specialized}. It is well known that $H^{\bullet}(i_{\lambda}^! \cS_G(V))$ is concentrated in degrees of constant parity (see \cite[Corollaire 2.10]{sp2} or \cite[Theorem 5.5]{kl}). Hence, by the same spectral sequence arguments as above, the equivariant cohomology $H^{\bullet}_{\Tv}(i_{\lambda}^! \cS_G(V))$ is also a free $H^{\bullet}_{\Tv}(\pt)$-module, and there is a canonical isomorphism
\[
H^{\bullet}_{\Tv}(i_{\lambda}^! \cS_G(V)) \otimes_{H^{\bullet}_{\Tv}(\pt)} \C_0 \ \cong \ H^{\bullet}(i_{\lambda}^! \cS_G(V)).
\]
It follows that for any $h \in \ftv$ the canonical filtration on the vector space
\[
H^{\bullet}_{\Tv}(i_{\lambda}^! \cS_G(V)) \otimes_{H^{\bullet}_{\Tv}(\pt)} \C_h
\]
induced by the grading on $H^{\bullet}_{\Tv}(i_{\lambda}^! \cS_G(V))$ has associated graded $H^{\bullet}(i_{\lambda}^! \cS_G(V))$.

Finally, using isomorphisms \eqref{eqn:localization-map-specialized} and \eqref{eqn:isom-cohomology-T}, we have constructed for any $h \in \ftv \smallsetminus \cup_{\alv} \ker(\alv)$ a canonical filtration on the vector space $V(\lambda) \cong H^{\lan \lambda,2\rhov\ran}(\fT_{\lambda}, t_{\lambda}^! \cS_G(V))$, denoted $\mathrm{F}^{\geom}_{\bullet}$, with associated graded $H^{\bullet}(i_{\lambda}^! \cS_G(V))$. (Note that this filtration depends on $h$, although we do not indicate this in the notation, for simplicity.)

The following result is due to Ginzburg (see \cite[Proposition 5.5.2]{g2}). As our point of view is different from that of Ginzburg, we include a proof in \S \ref{ss:proof-filtrations}. This proof is completely different from the one given by Ginzburg, and more in the spirit of \cite{mv} and \cite{yz}. It is independent of the rest of the paper.

\begin{thm}
\label{thm:filtrations}

There exists an explicit choice of $h \in \ftv$ such that the filtration $\mathrm{F}^{\geom}_{\bullet}$ on $V(\lambda)$ coincides with the filtration $\FBK_{\bullet}$ up to a shift. More precisely, for this choice of $h$, for any $i$ we have
\[
\FBK_i \, V(\lambda)  \ = \ \mathrm{F}^{\geom}_{2i+\lan \lambda,2\rhov \ran} \, V(\lambda) \ = \ \mathrm{F}^{\geom}_{2i+1+\lan \lambda,2\rhov \ran} \, V(\lambda).
\]

\end{thm}

\begin{rmk}
\label{rmk:localization-thm}
Here the particular case of the localization theorem we use in \eqref{eqn:localization-map} is very easy to prove directly. Indeed, consider the object $N:=t_{\lambda}^! \cS_G(V)$, an object of the $\Tv$-equivariant derived category $\cDb_{\eq{\Tv}}(\fT_{\lambda})$. The variety $\fT_{\lambda}$ has a natural action of $\Uv^-(\C[t^{-1}])$ (which is compatible with the $\Tv$-action in the natural way), and there exists a subgroup $K \subset \Uv^-(\C[t^{-1}])$, normalized by $\Tv$, such that the quotient $K \backslash \fT_{\lambda}$ is a finite dimensional affine space, with a linear $\Tv$-action (whose weights are in $-\Rv^+$), and such that the natural quotient map $\fT_{\lambda} \to K \backslash \fT_{\lambda}$ restricts to a closed embedding on the support of $N$. Then we can consider $N$ as an object of $\cDb_{\eq{\Tv}}(K \backslash \fT_{\lambda})$.

Consider the inclusions
\[
\xymatrix{
\{L_{\lambda}\} \ar@{^{(}->}[r]^-{a_{\lambda}} & K \backslash \fT_{\lambda} & \ (K \backslash \fT_{\lambda}) \smallsetminus \{L_{\lambda}\} \ar@{_{(}->}_-{j_{\lambda}}[l],
}
\]
and the distinguished triangle
\[
(a_{\lambda})_! (a_{\lambda})^! N \to N \to (j_{\lambda})_* (j_{\lambda})^* N \xrightarrow{+1}.
\]
It is easy to check that $H^{\bullet}_{\Tv}((j_{\lambda})^* N)$ is anihilated by $\prod_{{\check \alpha} \in -\Rv^+} {\check \alpha}$ (see e.g.~\cite[Lemma 3.3]{fw}). Hence morphism \eqref{eqn:localization-map} becomes an isomorphism after inverting all ${\check \alpha} \in \Rv$. Moreover, as $H_{\Tv}^{\bullet}(i_{\lambda}^! \cS_G(V))$ is free over $\mathrm{S}(\ft)$, this morphism is injective.
\end{rmk}

\subsection{Hyperbolic localization and semisimplicity of Frobenius}

Fix a standard Levi $\Lv \subset \Gv$, and recall the notation of \S \ref{ss:hl-restriction}. We will denote by the same symbol the morphisms similar to $i,j,p,q$ but defined over $\F_p$. We define the functor $\Theta_L^{\mix}$ so that it is a mixed version of $\Theta_L$ and that it sends pure objects of weight $0$ to pure objects of weight $0$. More precisely, using again the notation of \S \ref{ss:hl-restriction}, this functor sends an object $M$ to
\[
\Theta_L^{\mix}(M) \ := \
\bigoplus_{\chi \in X^*(Z(L))} \, M_{\chi} [\lan \chi , 2\rho_{\Gv} - 2 \rho_{\Lv} \ran] \bigl< - \lan \chi , 2\rho_{\Gv} - 2 \rho_{\Lv} \ran \bigr>.
\]
Then for any $\lambda \in \bX^+$ one can consider the object
\[
\Theta_L^{\mix} \circ p_! i^* \, \ICm_{\lambda} \ \cong \ \Theta_L^{\mix} \circ q_* j^! \ICm_{\lambda},
\]
an $\Lv(\F_p[[x]])$-equivariant perverse sheaf on $\Gr_{\Lv,\F_p}$. (Note that the isomorphism provided by \cite[Theorem 1]{br} also holds over $\F_p$, see \cite[Section 5]{br}.)

\begin{prop}
\label{prop:hl-Frobenius}

The perverse sheaf $\Theta_L^{\mix} \circ p_! i^* \, \ICm_{\lambda}$ is an object of the category $\Perv^0_{\mon{\LvO}}(\Gr_{\Lv})$.

\end{prop}

\begin{proof}
By \cite[Theorem 8]{br}, this perverse sheaf is pure of weight $0$. Hence we only have to show that it is semisimple. The case $L=T$ is contained in \cite[Th{\'e}or{\`e}me 3.1]{np}. Now we deduce the general case. To avoid confusion, we add a subscript ``${}_{L \subset G}$" to the morphisms $i$ and $p$ and to $\Theta^{\mix}$ relative to the inclusion $L \subset G$, and similarly for the other inclusions. 

First, by base change there is an isomorphism of functors
\begin{multline}
\label{eqn:hl-composition}
\bigl( \Theta_{T \subset L}^{\mix} \circ (p_{T \subset L})_! (i_{T \subset L})^* \bigr) \circ \bigl( \Theta_{L \subset G}^{\mix} \circ (p_{L \subset G})_! (i_{L \subset G})^* \bigr)\\ \cong \ \Theta_{T \subset G}^{\mix} \circ (p_{T \subset G})_! (i_{T \subset G})^*.
\end{multline}
Consider the perverse sheaf $\Theta_{L \subset G}^{\mix} \circ (p_{L \subset G})_! (i_{L \subset G})^* \, \ICm_{\lambda}$. Choose a decomposition into a sum of indecomposable pure perverse sheaves on $\Gr_{\Lv,\F_p}$:
\[
\Theta_{L \subset G}^{\mix} \circ (p_{L \subset G})_! (i_{L \subset G})^* \, \ICm_{\lambda} \ \cong \ \bigoplus_i \, M_i.
\]
By \cite[Proposition 5.3.9]{bbd}, each $M_i$ can be written as $S_i \otimes_{\Qlb} V_i$, where $S_i$ is a simple perverse sheaf, and $V_i$ is a $\Qlb$-vector space endowed with an indecomposable unipotent action of the Frobenius. Assume that $V_{i_0} \neq \Qlb$ for some $i_0$. Then, using isomorphism \eqref{eqn:hl-composition}, we obtain that $\Theta_{T \subset G}^{\mix} \circ (p_{T \subset G})_! (i_{T \subset G})^* \, \ICm_{\lambda}$ has a direct summand which is indecomposable but not simple. This is absurd since we know already the result for $T$. This concludes the proof.
\end{proof}

\subsection{Mixed version of $\fR^G_L$}
\label{ss:mixed-version-hl}

Consider the subcategory $\cD^{\Weil}_{\mon{\Iv}}(\Gr_{\Gv,\F_p})$ of the derived category of constructible sheaves on $\Gr_{\Gv,\F_p}$ generated by the simple objects $\IC(Y)\lan j \ran$, for $Y$ an $\Iv_{\F_p}$-orbit on $\Gr_{\Gv,\F_p}$ and $j \in \Z$. We use the same notation for $\GvO$-monodromic complexes. As in \S \ref{ss:equivalence-mix}, we also denote by $\Pervw_{\mon{\Iv}}(\Gr_{\Gv,\F_p})$ the abelian subcategory of perverse sheaves. As in \cite[\S 7.2]{ar}, we denote by $\iota$ the composition
\[
\cDb \Pervm_{\mon{\Iv}}(\Gr_{\Gv}) \, \to \, \cDb \Pervw_{\mon{\Iv}}(\Gr_{\Gv,\F_p}) \, \xrightarrow{\mathsf{real}} \, \cD^{\Weil}_{\mon{\Iv}}(\Gr_{\Gv,\F_p}),
\]
where the first arrow is the derived functor of the embedding $\Pervm_{\mon{\Iv}}(\Gr_{\Gv,\F_p}) \hookrightarrow \Pervw_{\mon{\Iv}}(\Gr_{\Gv,\F_p})$, and the second arrow is the realization functor. We also use the same notation for the functor
\[
\cD^{\mix}_{\mon{\GvO}}(\Gr_{\Gv}) \, \to \, \cD^{\Weil}_{\mon{\GvO}}(\Gr_{\Gv,\F_p})
\]
obtained by restriction. We denote by
\[
\varkappa : \cD^{\Weil}_{\mon{\GvO}}(\Gr_{\Gv,\F_p}) \, \to \, \cDb_{\mon{\GvO}}(\Gr_{\Gv})
\]
the composition of the extension of scalars from $\F_p$ to $\overline{\F_p}$, followed by the equivalence obtained by restriction of the first equivalence of Lemma \ref{lem:change-of-field}. With this notation, by construction we have $\For = \varkappa \circ \iota$.

It follows in particular from Proposition \ref{prop:hl-Frobenius} that the functor $\Theta_L^{\mix} \circ p_! i^*$ restricts to a functor
\[
\cD^{\Weil}_{\mon{\GvO}}(\Gr_{\Gv,\F_p}) \to \cD^{\Weil}_{\mon{\LvO}}(\Gr_{\Lv,\F_p}).
\]
Hence it defines a \emph{geometric} functor in the sense of \cite[Definition 6.6]{ar}. Recall the notation $\Pure(-)$ introduced in \cite[\S 6.4]{ar}. Then our functor restricts to a homogeneous functor
\[
\Pure_{\mon{\GvO}}(\Gr_{\Gv}) \ \to \ \Pure_{\mon{\LvO}}(\Gr_{\Lv})
\]
in the sense of \cite[Definition 4.1]{ar}. Hence by \cite[Proposition 9.1]{ar} we have the following existence result.

\begin{prop}
\label{prop:mixed-version-RGL}

There exists a functor $\fR^{G,\mix}_L$ which makes the diagram
\[
\xymatrix@R=0.5cm{
\cD^{\mix}_{\mon{\GvO}}(\Gr_{\Gv}) \ar[rr]^-{\iota} \ar[rd]_{\For} \ar[dd]_-{\fR^{G,\mix}_L} & & \cD^{\Weil}_{\mon{\GvO}}(\Gr_{\Gv,\F_p}) \ar[ld]^-{\varkappa} \ar[dd]^-{\Theta_L^{\mix} \circ p_! i^*} & \\
& \cDb_{\mon{\GvO}}(\Gr_{\Gv}) \ar[dd]^(.3){\fR^G_L} & \\
\cD^{\mix}_{\mon{\LvO}}(\Gr_{\Lv}) \ar'[r][rr]^(-.3){\iota} \ar[rd]_-{\For} & & \cD^{\Weil}_{\mon{\LvO}}(\Gr_{\Lv,\F_p}) \ar[ld]^-{\varkappa} \\
& \cDb_{\mon{\LvO}}(\Gr_{\Lv}) & \\
}
\]
commutative up to isomorphisms of functors.\qed

\end{prop}

\begin{rmk}
\label{rmk:hl-restriction-mix}
It follows in particular from Proposition \ref{prop:mixed-version-RGL} that the following diagram commutes:
\[
\xymatrix@C=1.5cm{
\Perv^0_{\mon{\GvO}}(\Gr_{\Gv}) \ar[r]^-{\Phi_G}_-{\sim} \ar[d]_-{\fR^{G,\mix}_L} & \Perv_{\mon{\GvO}}(\Gr_{\Gv}) \ar[d]^-{\fR^G_L} \\
\Perv^0_{\mon{\LvO}}(\Gr_{\Lv}) \ar[r]^-{\Phi_L}_-{\sim} & \Perv_{\mon{\LvO}}(\Gr_{\Lv}).
}
\]
Hence we deduce from Theorem \ref{thm:hl-restriction-classical} an isomorphism of functors
\[
\fR^{G,\mix}_L \circ \cS_G^0 \ \cong \ \cS_L^0 \circ \mathrm{Res}^G_L.
\]
\end{rmk}

\subsection{Action of the functors on morphisms}
\label{ss:functors-morphisms}

Consider the triangulated functors
\begin{multline}
\label{eqn:functors-hl-restriction}
(i_L^G)_{\mix}, \ F_L^{\mix} \circ \fR^{G,\mix}_L \circ (F_G^{\mix})^{-1} : \cDb_{\free} \Coh^{G \times \Gm}(\cN_G) \\ \to \cDb_{\free} \Coh^{L \times \Gm}(\cN_L).
\end{multline}
By Lemma \ref{lem:orlov-coherent}, there are equivalences of categories
\begin{align*}
\Kb\bigl( \Coh_{\free}^{G \times \Gm}(\cN_G) \bigr) \ & \cong \ \cDb_{\free} \Coh^{G \times \Gm}(\cN_G), \\ \Kb\bigl( \Coh_{\free}^{L \times \Gm}(\cN_L) \bigr) \ & \cong \ \cDb_{\free} \Coh^{L \times \Gm}(\cN_L).
\end{align*}
We claim that both functors in \eqref{eqn:functors-hl-restriction} send the subcategory $\Coh_{\free}^{G \times \Gm}(\cN_G)$ to the subcategory $\Coh_{\free}^{L \times \Gm}(\cN_L)$. This is obvious for the first functor. For the second one, we have natural isomorphisms, for any $V$ in $\Rep(G)$ and $n \in \Z$,
\begin{align*}
F_L^{\mix} \circ \fR^{G,\mix}_L \circ (F_G^{\mix})^{-1}(V \otimes \cO_{\cN_G} \lan n \ran) \ & \cong \ F_L^{\mix} \circ \fR^{G,\mix}_L(\cS_G^0(V) \lan n \ran [-n]) \\
& \cong \ F_L^{\mix} \bigl(\cS_L^0(\mathrm{Res}^G_L(V)) \lan n \ran [-n] \bigr) \\
& \cong \ \mathrm{Res}^G_L(V) \otimes \cO_{\cN_L} \lan n \ran,
\end{align*}
where we have used Lemma \ref{lem:image-IC-mix} and Remark \ref{rmk:hl-restriction-mix}. The claim follows.

\begin{prop}
\label{prop:isom-functors-Orlov-categories}

There exists an isomorphism of additive functors from the category $\Coh_{\free}^{G \times \Gm}(\cN_G)$ to $\Coh_{\free}^{L \times \Gm}(\cN_L)$:
\[
(i_L^G)_{\mix}{}_{|\Coh_{\free}^{G \times \Gm}(\cN_G)} \ \cong \ \bigl( F_L^{\mix} \circ \fR^{G,\mix}_L \circ (F_G^{\mix})^{-1} \bigr){}_{|\Coh_{\free}^{G \times \Gm}(\cN_G)}.
\]

\end{prop}

\begin{proof}
Recall the category $\sA_G$ and the equivalence $\sA_G \cong \Coh_{\free}^{G \times \Gm}(\cN_G)$ , see \S \ref{ss:orlov-category}. We have already constructed isomorphisms
\begin{multline*}
(i_L^G)_{\mix}(V \otimes \cO_{\cN_G} \lan n \ran) 
\ \cong \ \mathrm{Res^G_L}(V) \otimes \cO_{\cN_L} \lan n \ran 
\\ \cong \ \bigl( F_L^{\mix} \circ \fR^{G,\mix}_L \circ (F_G^{\mix})^{-1} \bigr)(V \otimes \cO_{\cN_G} \lan n \ran)
\end{multline*}
for any $V$ in $\Rep(G)$ and any $n \in \Z$. Hence to prove the proposition is enough to prove that for any $V,V'$ in $\Rep(G)$ and $n \in \Z$, the morphisms
\begin{multline*}
\bigoplus_{n \in \Z} \, \Hom_{\Coh^{G \times \Gm}(\cN_G)} \bigl( V \otimes \cO_{\cN_G},V' \otimes \cO_{\cN_G} \lan n \ran \bigr) \ \to \\
\bigoplus_{n \in \Z} \, \Hom_{\Coh^{L \times \Gm}(\cN_L)} \bigl( \mathrm{Res^G_L}(V) \otimes \cO_{\cN_L},\mathrm{Res^G_L}(V') \otimes \cO_{\cN_L} \lan n \ran \bigr)
\end{multline*}
induced by our two functors coincide. Moreover, the direct sums of morphisms spaces considered here can be expressed in terms of morphisms in the category $\DGCGf(\cN_G)$. Hence Proposition \ref{prop:isom-functors-Orlov-categories} follows from Proposition \ref{prop:action-functors-morphisms} below.
\end{proof}

\begin{prop}
\label{prop:action-functors-morphisms}

For any $V,V'$ in $\Rep(G)$, the morphisms
\begin{multline*}
\bigoplus_{n \in \Z} \, \Hom_{\DGCGf(\cN_G)}^n \bigl( V \otimes \C[\cN_G],V' \otimes \C[\cN_G]) \ \to \\
\bigoplus_{n \in \Z} \, \Hom_{\DGC_{\fr}^{L}(\cN_L)}^n \bigl( \mathrm{Res^G_L}(V) \otimes \C[\cN_L],\mathrm{Res^G_L}(V') \otimes \C[\cN_L]\bigr)
\end{multline*}
induced by the functors $(i_L^G)^*$ and $F_L \circ \fR^{G}_L \circ (F_G)^{-1}$ coincide.

\end{prop}

\begin{proof}
As in the proof of Proposition \ref{prop:isom-morphisms}, using adjunction (see Lemma \ref{lem:adjunction}), one can assume that $V=\C$. Then for simplicity we replace $V'$ by $V$ in the notation. What we have to prove is the commutativity of the following diagram:
\begin{equation}
\label{eqn:diagram-morphisms-Orlov-categories}
\vcenter{
\xymatrix{
\Hom_{\cDb_{\mon{\GvO}}(\Gr_{\Gv})}^{\bullet}(1_G,\cS_G(V)) \ar[d]_-{\fR^G_L} \ar[r]^-{\sim} & V^{G^{e_G}} &\Hom_G(\cN_G,V) \ar[l]_-{\sim} \ar[d]^-{(-) \circ i_L^G} \\
\Hom_{\cDb_{\mon{\LvO}}(\Gr_{\Lv})}^{\bullet}(1_L,\cS_L(\mathrm{Res}^G_L V)) \ar[r]^-{\sim} & V^{L^{e_L}} & \Hom_L(\cN_L,\mathrm{Res}^G_L V), \ar[l]_-{\sim}
}
}
\end{equation}
where the first horizontal isomorphisms are given by \eqref{eqn:prop-isom-morphisms-1} and \eqref{eqn:prop-isom-morphisms-2}, and the second horizontal isomorphisms are given by restriction to $e_G \in \cN_G$, respectively $e_L \in \cN_L$. (Here, $\Hom_G(\cN_G,V)$ denotes the space of morphisms of $G$-varieties from $\cN_G$ to $V$, and similarly for $L$.)

By adjunction for the pair $((i_0)_!,i_0^!)$ we have
\[
\Hom_{\cDb_{\mon{\GvO}}(\Gr_{\Gv})}^{\bullet}(1_G,\cS_G(V)) \ \cong \ H^{\bullet}(i_0^! \cS_G(V)).
\]
Hence, by Theorem \ref{thm:filtrations}, this graded vector space is (up to regrading) the associated graded of the Brylinski--Kostant filtration on $V^T$, for the group $G$ and the nilpotent $e_G$, denoted by $F^{\mathrm{BK},G}_{\bullet}$. Similarly, $\Hom_{\cDb_{\mon{\LvO}}(\Gr_{\Lv})}^{\bullet}(1_L,\cS_L(\mathrm{Res}^G_L V))$ is the associated graded of the Brylinski--Kostant filtration on $V^T$ for the group $L$ and the nilpotent $e_L$, denoted by $F^{\mathrm{BK},L}_{\bullet}$. Moreover, by base change we have an isomorphism $\fR^G_L \circ \fR^L_T = \fR^G_T$ of functors on $\cDb_{\mon{\GvO}}(\Gr_{\Gv})$. Hence the morphism $V^T \to V^T$ induced by $\fR^G_L$ is the identity.

By the description of $e_G$ in \S \ref{ss:reminder-cohomology} and the fact that $\prod_{\alpha \in \Delta} \alpha : T \to (\C^{\times})^{\# \Delta}$ is surjective, there exists a cocharacter $\chi : \C^{\times} \to T$ such that 
\begin{equation}
\label{eqn:limite-e_G}
e_L=\lim_{t \to 0} \chi(t) \cdot e_G.
\end{equation}
It follows that we have inclusions
\[
V^T \cap \ker(e_G^n) \subset V^T \cap \ker(e_L^n),
\]
which induce a morphism
\[
\vartheta : \mathrm{gr}^{\mathrm{BK},G}_{\bullet}(V^T) \to \mathrm{gr}^{\mathrm{BK},L}_{\bullet}(V^T).
\]
Via the isomorphisms
\[
\mathrm{gr}^{\mathrm{BK},G}_{\bullet}(V^T) \cong V^{G^{e_G}}, \quad \mathrm{gr}^{\mathrm{BK},L}_{\bullet}(V^T) \cong V^{L^{e_G}}
\]
which send $v \in \mathrm{gr}^{\mathrm{BK},G}_{i}(V^T)$ to $e_G^i \cdot v$ (and a similar formula for $L$), see \cite[Corollary 2.7]{bry}, one can complete diagram \eqref{eqn:diagram-morphisms-Orlov-categories} to the diagram
\[
\xymatrix{
\Hom_{\cDb_{\mon{\GvO}}(\Gr_{\Gv})}^{\bullet}(1_G,\cS_G(V)) \ar[d]_-{\fR^G_L} \ar[r]^-{\sim} & V^{G^{e_G}} \ar[d]^-{\vartheta} &\Hom^G(\cN_G,V) \ar[l]_-{\sim} \ar[d]^-{(-) \circ i_L^G} \\
\Hom_{\cDb_{\mon{\LvO}}(\Gr_{\Lv})}^{\bullet}(1_L,\cS_L(\mathrm{Res}^G_L V)) \ar[r]^-{\sim} & V^{L^{e_L}} & \Hom^L(\cN_L,\mathrm{Res}^G_L V), \ar[l]_-{\sim} \\
}
\]
in which the left square commutes. The commutativity of the right square follows again from \eqref{eqn:limite-e_G}. This finishes the proof.\end{proof}

\subsection{Proof of Theorem \ref{thm:hl-restriction}}
\label{ss:proof-thm-hl-restriction}

Now we can finish the proof of Theorem \ref{thm:hl-restriction}. Na\-mely, we are in the situation of \cite[Theorem 4.7]{ar}: we have two triangulated functors (see Equation \eqref{eqn:functors-hl-restriction}) between bounded homotopy categories of Orlov categories (see Lemma \ref{lem:orlov-coherent}), which induce homogeneous functors between these Orlov categories (see \S \ref{ss:functors-morphisms}), and an isomorphism of additive functors between their restrictions (see Proposition \ref{prop:isom-functors-Orlov-categories}). Hence, by \cite[Theorem 4.7]{ar}, we get an isomorphism
\[
(i_L^G)_{\mix} \ \cong \ F_L^{\mix} \circ \fR^{G,\mix}_L \circ (F_G^{\mix})^{-1}.
\]
Composing with $F_G^{\mix}$ gives the commutativity of the diagram of Theorem \ref{thm:hl-restriction}.

\subsection{Digression: $t$-structures}
\label{ss:digression}

We have explained in \cite[\S 5.1]{ar} that for any Orlov category $\scA$, there exists a natural bounded $t$-structure on $\Kb(\scA)$ whose heart is a finite length abelian category endowed with a mixed structure. The simple objects in this heart are the $S[\deg S]$, where $S$ runs over indecomposable objects of $\scA$.

By Lemma \ref{lem:orlov-coherent}, the category $\Coh_{\fr}^{G \times \Gm}(\cN_G)$ has the structure of an Orlov category, and its bounded homotopy category is
\[
\cDb_{\fr} \Coh^{G \times \Gm}(\cN_G) \ \cong \ \cD^{\mix}_{\mon{\GvO}}(\Gr_{\Gv}).
\]
Hence one gets an abelian category $\scC_G$ as the heart of a certain $t$-structure on these triangulated categories. With the definition of Lemma \ref{lem:orlov-coherent}, the simple objects of this heart are the objects
\[
V_{\lambda} \otimes \cO_{\cN_G} \lan i \ran [i] \ \text{ in } \ \cDb_{\fr} \Coh^{G \times \Gm}(\cN_G), \ \text{ or } \ \IC_{\lambda}^{\mix}\lan i \ran \ \text{ in } \ \cD^{\mix}_{\mon{\GvO}}(\Gr_{\Gv}),
\]
where $\lambda$ runs over $\bX^+$, and $i$ over $\Z$. The abelian category $\scC_G$ is semisimple and equivalent to the category $\Pervm_{\mon{\GvO}}(\Gr_{\Gv})$ of \S \ref{ss:Satake-mix}, so nothing new arises in this situation.

But the structure of $\Coh_{\fr}^{G \times \Gm}(\cN_G)$ as an Orlov category is not unique. In fact, one can take as a degree function
\[
\deg(V_{\lambda} \otimes \cO_{\cN_G} \lan i \ran) = \lfloor {\textstyle \frac{ki}{2}} \rfloor
\] 
for any $k \in \Z_{>0}$. If $k$ is even, then the situation is the same as above, and the heart is semisimple. But if $k$ is odd, then the simple objects are the
\[
V_{\lambda} \otimes \cO_{\cN_G} \lan i \ran [\lfloor {\textstyle \frac{k i}{2} } \rfloor] \ \text{ in } \ \cDb_{\fr} \Coh^{G \times \Gm}(\cN_G), \] 
or the 
\[ 
\IC_{\lambda}^{\mix}\lan i \ran [\lfloor {\textstyle \frac{ki}{2} } \rfloor - i] \ \text{ in } \ \cD^{\mix}_{\mon{\GvO}}(\Gr_{\Gv}).
\]
In particular, this heart is not semisimple.

As Koszul duality is ubiquitous in this geometric context, we expect that our Orlov category is Koszulescent (in the sense of \cite[\S 5.2]{ar}) in this case, but we were not able to prove it.

\subsection{Reminder on equivariant cohomology of $\Gr_{\Gv}$}
\label{ss:reminder-cohomology-equivariant}

In the rest of this section we give a proof of Theorem \ref{thm:filtrations}. 

Consider the $\Tv$-equivariant cohomology $H^{\bullet}_{\Tv}(\Gr_{\Gv})$. It is naturally a Hopf algebra. By the same arguments as in \S \ref{ss:reminder-cohomology}, any primitive element $c$ in $H^{\bullet}_{\Tv}(\Gr_{\Gv})$ defines an element $\psi^{\Tv}(c) \in \fg \otimes_{\C} H^{\bullet}_{\Tv}(\pt)$ (see \cite[\S 5.3]{yz} for details). In particular, we can assume that the line bundle $\cL_{\det}$ has a $\Tv$-equivariant structure, replacing it by a sufficiently large power if necessary (see \cite[Lemma 4.2]{yz}). Hence one can consider the equivariant first Chern class $c_1^{\Tv}:=c_1^{\Tv}(\cL_{\det}) \in H^2_{\Tv}(\Gr_{\Gv})$. By \cite[Lemma 5.1]{yz} this element is primitive. Hence one can define
\[
e_G^{\Tv}:=\psi^{\Tv}(c_1^{\Tv}) \ \in \fg \otimes_{\C} H^{\bullet}_{\Tv}(\pt).
\]

The element $e_G^{\Tv}$ is described very explicitly in \cite[Propositions 5.6 and 5.7]{yz} (in the case $\cL_{\det}$ is the determinant line bundle). Let us recall the parts of this description that we will need. First, it is easy to show (see \cite[Lemma 5.5]{yz}) that $e_G^{\Tv} \in \fg \oplus (\ft \otimes_{\C} H^2_{\Tv}(\pt))$. By construction, the component on $\fg$ is the element $e_G$ of \S \ref{ss:reminder-cohomology}. Hence one can write
\[
e_G^{\Tv} = e_G + f_G
\]
where $f_G \in \ft \otimes_{\C} H^2_{\Tv}(\pt) \cong \ft \otimes_{\C} \ft$. One can view $f_G$ as a bilinear form on $\ftv \cong \ft^*$. Decomposing the basis $\Delta$ into connected components, one obtains a direct sum decomposition $\fgv = \mathfrak{z}(\fgv) \oplus \bigl( \oplus_i \, \fgv_i \bigr)$ where each $\fgv_i$ is a simple Lie algebra, and $\mathfrak{z}(\fgv)$ is the center of $\fgv$. Accordingly we have a decomposition $\ftv=\mathfrak{z}(\fgv) \oplus \bigl( \oplus_i \, \ftv_i \bigr)$. On each $\ftv_i$ one can consider the restriction $\kappa_i$ of the Killing form of $\fgv_i$, and use the direct sum decomposition to extend it to $\ftv$. By \cite[Proposition 5.7]{yz} we have the following:
\[
f_G \text{ is a linear combination with non-zero coefficients of the } \kappa_i\text{'s}.
\]
This bilinear form defines a morphism $\ftv \to \ftv^*\cong \ft$. Let us fix an element $h \in \ftv$ such that
\begin{equation}
\label{eqn:choice-h}
f_G (-,h) = 2\rhov \in \ft.
\end{equation}
Such a choice is possible, and is unique up to adding an element of $\mathfrak{z}(\fgv)$. This will be our choice of $h$ in Theorem \ref{thm:filtrations}.

\subsection{Proof of Theorem \ref{thm:filtrations}}
\label{ss:proof-filtrations}

Fix $V$ in $\Rep(G)$, and let $M=\cS_G(V)$. Recall that we have
\[
\overline{\fT_{\lambda}} \ = \ \bigsqcup_{\mu \geq \lambda} \fT_{\mu}
\]
(see \cite[Proposition 3.1]{mv}). Here, ``$\geq$" is the partial order on $\bX$ determined by our choice of $R^+$. Recall also that we denote by $t_{\lambda} : \fT_{\lambda} \hookrightarrow \Gr_{\Gv}$ the inclusion. Similarly, we denote the inclusion of the closure by $\overline{t_{\lambda}} : \overline{\fT_{\lambda}} \hookrightarrow \Gr_{\Gv}$. As $\overline{\fT_{\lambda}}$ is closed in $\Gr_{\Gv}$ and $\fT_{\lambda}$ is open in $\overline{\fT_{\lambda}}$, there are natural morphisms
\begin{equation}
\label{eqn:diagram-cohomology}
\vcenter{
\xymatrix{
H^{\bullet}_{\Tv}(i_{\lambda}^! M) \ar[r]^-{\eta} \ar[rd]_-{\xi} & H^{\bullet}_{\Tv}(\overline{\fT_{\lambda}}, \overline{t_{\lambda}}^! M) \ar[r]^-{\psi} \ar[d]^-{\phi} & H^{\bullet}_{\Tv}(M) \\
& H^{\bullet}_{\Tv}(\fT_{\lambda}, t_{\lambda}^! M). &
}
}
\end{equation}
Here, $\xi$ is the morphism \eqref{eqn:localization-map}.

The arguments for the proof of the following lemma are adapted from \cite[Proof of Lemma 2.2]{yz}.

\begin{lem}
\label{lem:equivariant-cohomology}

\begin{enumerate}
\item The morphism $\phi$ has a canonical splitting.
\item There is a canonical isomorphism
\[
H_{\Tv}^{\bullet}(\overline{\fT_{\lambda}},\overline{t_{\lambda}}^! M) \ \cong \ \bigoplus_{\mu \geq \lambda} \, \bigl( H^{\lan \mu,2\rhov \ran}(\fT_{\mu},t_{\mu}^! M) \otimes_{\C} H_{\Tv}^{\bullet}(\pt) \bigr) [-\lan \mu,2\rhov \ran].
\]
\item There is a canonical isomorphism
\[
H_{\Tv}^{\bullet}(M) \ \cong \ \bigoplus_{\mu \in \bX} \, \bigl( H^{\lan \mu, 2\rhov \ran}(\fT_{\mu},t_{\mu}^! M) \otimes_{\C} H_{\Tv}^{\bullet}(\pt) \bigr) [- \lan \mu, 2\rhov \ran].
\]
\item Under the isomorphisms of {\rm (2)} and {\rm (3)}, and {\rm \eqref{eqn:isom-equivariant-cohomology}}, the morphism $\psi$, respectively $\phi$, is the inclusion of, respectively projection on, the corresponding direct summands.
\end{enumerate}

\end{lem}

\begin{proof}
(1) It is easy to show using degree arguments that there is a canonical isomorphism
\[
H^{\bullet}(\overline{\fT_{\lambda}},\overline{t_{\lambda}}^! M) \ \cong \ \bigoplus_{\mu \geq \lambda} H^{\lan \mu, 2\rhov \ran}(\fT_{\mu},t_{\mu}^! M)[- \lan \mu,2\rhov \ran].
\]
For any $\mu \geq \lambda$, the parity of $\lan \mu,2\rhov \ran$ is the same as that of $\lan \lambda,2\rhov \ran$. Hence, by the same argument using the Leray spectral sequence as in \S \ref{ss:reminder-BK-filtration}, there exists an \emph{a priori} non-canonical isomorphism as in (2). In particular, the lowest non-zero degree in $H_{\Tv}^{\bullet}(\overline{\fT_{\lambda}},\overline{t_{\lambda}}^! M)$ is $\lan \lambda, 2\rhov \ran$, and we have
\[
H_{\Tv}^{\lan \lambda, 2\rhov \ran}(\overline{\fT_{\lambda}},\overline{t_{\lambda}}^! M) \ \cong \ H^{\lan \lambda, 2\rhov \ran}(\fT_{\lambda},t_{\lambda}^! M).
\]
In particular, there is a canonical morphism
\[
\bigl( H^{\lan \lambda, 2\rhov \ran}(\fT_{\lambda},t_{\lambda}^! M) \otimes H_{\Tv}^{\bullet}(\pt) \bigr) [-\lan \lambda,2\rho \ran] \to H_{\Tv}^{\bullet}(\overline{\fT_{\lambda}},\overline{t_{\lambda}}^! M).
\]
Using isomorphism \eqref{eqn:isom-equivariant-cohomology}, one sees that this morphism is a splitting for $\phi$.

(2) By (1), $\phi$ is a canonically split surjection. Its kernel is
\[
H^{\bullet}_{\Tv}(\overline{\fT_{\lambda}} \smallsetminus \fT_{\lambda}, r_{\lambda}^! M),
\]
where $r_{\lambda} : \overline{\fT_{\lambda}} \smallsetminus \fT_{\lambda} \hookrightarrow \Gr_{\Gv}$ is the (closed) inclusion. By the same arguments, the restriction morphism to
\[
\bigoplus_{\alpha \in \Delta} H^{\bullet}_{\Tv}(\fT_{\lambda+\alpha}, t_{\lambda+\alpha}^! M)
\]
is also a canonically split surjection. Repeating this argument again and again, one obtains the isomorphism.

(3) Using the splitting in (1) for any $\mu \in \bX$, one obtains a canonical morphism from the right-hand side to the left-hand side. By the same Leray spectral sequence argument as above, one proves that this morphism is an isomorphism.

Property (4) is clear by construction.\end{proof}

\begin{prop}
\label{prop:inclusion-filtration}

Assume $h \in \ftv$ is such that \eqref{eqn:choice-h} holds. Then for any $i \in \Z$ we have
\[
\mathrm{F}^{\geom}_{2i+\lan \lambda,2\rhov \ran} \, V(\lambda) \ = \ \mathrm{F}^{\geom}_{2i+1+\lan \lambda,2\rhov \ran} \, V(\lambda) \ \subset \ \FBK_i \, V(\lambda).
\]

\end{prop}

\begin{proof}
The equality follows directly from the fact that $H^{\bullet}(i_{\lambda}^! M)$ is concentrated in degrees of the same parity as $\lan \lambda,2\rhov \ran$, see e.g.~\cite[Corollaire 2.10]{sp2}.

Consider some $c \in H^{2i+\lan \lambda,2\rhov \ran}_{\Tv}(i_{\lambda}^! M)$. Its image under the morphism $\eta$ of \eqref{eqn:diagram-cohomology} decomposes according to Lemma \ref{lem:equivariant-cohomology} as:
\[
\eta(c) \ = \ x_{\lambda} + \bigl( \sum_{0 < j \leq i} \sum_{\gamma \in j \Delta} x_{\lambda + \gamma} \bigr)
\]
where for any relevant $\mu$ 
\[
x_{\mu} \in H^{\lan \mu,2\rhov \ran}(\fT_{\mu},t_{\mu}^! M) \otimes_{\C} H^{2i + \lan \lambda-\mu,2\rhov \ran}_{\Tv}(\pt).
\] 
Here, $j \Delta \subset \bX$ is the set of sums of $j$ (not necessarily distinct) elements of $\Delta$. We have $\xi(c)=x_{\lambda}$. Hence what we have to show is that $e_G^{i+1} \cdot (x_{\lambda})_{h}=0$, where $(x_{\lambda})_{h}$ is the image of $x_{\lambda}$ in $V(\lambda)$, under the isomorphism \eqref{eqn:isom-cohomology-T}.

The morphism $\eta$ is compatible with cup products, and the element $c_1^{\Tv}$ of \S \ref{ss:reminder-cohomology-equivariant} acts on $H^{\bullet}_{\Tv}(i_{\lambda}^! M)$ via its restriction to $\{L_{\lambda}\}$, which is $f_G(\lambda,-) \in \ft$ (see \cite[Proof of Proposition 5.7]{yz}). (Here we view $\lambda$ in $\ft^* \cong \ftv$.) Hence
\[
c_1^{\Tv} \cup \bigl(x_{\lambda} + \sum_{0 < j \leq i} \sum_{\gamma \in j \Delta} x_{\lambda + \gamma} \bigr)_h = \lan \lambda,2\rhov \ran \cdot \bigl(x_{\lambda} + \sum_{0 < j \leq i} \sum_{\gamma \in j \Delta} x_{\lambda + \gamma} \bigr)_h,
\]
by our choice of $h$. On the other hand,
\[
c_1^{\Tv} \cup \bigl(x_{\lambda} + \sum_{0 < j \leq i} \sum_{\gamma \in j \Delta} x_{\lambda + \gamma} \bigr)_h =
(e_G + f_G(-,h)) \cdot \bigl(x_{\lambda} + \sum_{0 < j \leq i} \sum_{\gamma \in j \Delta} x_{\lambda + \gamma} \bigr)_h,
\]
where on the right-hand side we consider the action in $V$. We deduce that
\[
e_G \cdot (x_{\lambda})_h = - 2(\sum_{\gamma \in \Delta} x_{\lambda+\gamma})_h, \quad e_G \cdot \bigl( \sum_{\gamma \in \Delta} x_{\lambda+\gamma} \bigr)_h = -4 (\sum_{\delta \in 2\Delta} x_{\lambda+\delta})_h,
\]
and so on until
\[
e_G \cdot \bigl( \sum_{\gamma \in i\Delta} x_{\lambda+\gamma} \bigr)_h = 0.
\]
Taking all these equations into account, we obtain finally that $e_G^{i+1} \cdot (x_{\lambda})_h=0$.
\end{proof}

Now we can finish our proof of Theorem \ref{thm:filtrations}.

\begin{proof}[Proof of Theorem {\rm \ref{thm:filtrations}}]
By Proposition \ref{prop:inclusion-filtration}, it is enough to prove that the Laurent polynomials in $q$
\begin{equation}
\label{eqn:polynomials-geom-BK}
\sum_{i \in \Z} \, \dim \bigl( \mathrm{gr}^{\geom}_{\lan \lambda,2\rhov \ran + 2i} V(\lambda) \bigr) \cdot q^i \quad \text{and} \quad \sum_{i \in \Z} \, \dim \bigl( \mathrm{gr}^{\mathrm{BK}}_i V(\lambda) \bigr) \cdot q^i
\end{equation}
coincide. However, both of these polynomials are known explicitly, as follows. One can assume that $V$ is a simple module: $V=V_{\nu}$ for some $\nu \in \bXp$. Moreover, $V$ restricts to a simple module of the derived subgroup of $G$; hence one can assume that $G$ is semisimple. For simplicity, we write $P^{\geom}_{\nu,\lambda}(q)$ for the left-hand side of \eqref{eqn:polynomials-geom-BK}, and $P^{\mathrm{BK}}_{\nu,\lambda}(q)$ for the right-hand side.

First, consider the left-hand side. The associated graded of the geometric filtration is known by construction, hence we have
\[
P^{\geom}_{\nu,\lambda}(q) = \sum_{j \in \Z} \, \dim \bigl( H^{\lan \lambda,2\rhov \ran+ 2j}(i_{\lambda}^! \IC_{\nu}) \bigr) \cdot q^j.
\]
Let $w \in W$ be such that $w(\lambda)$ is dominant. Then we have by $\GvO$-equivariance of $\IC_{\nu}$
\[
P^{\geom}_{\nu,\lambda}(q) = q^{\lan w(\lambda)-\lambda,\rhov \ran} \cdot P^{\geom}_{\nu,w(\lambda)}(q).
\]
Now, this polynomial can be expressed in terms of Kazhdan--Lusztig polynomials (see \cite{kl, sp2, lu}). Using also the formula provided by \cite[Theorem 1.8]{ka}\footnote{This formula was conjectured by Lusztig in \cite[Equation (9.4)]{lu}.}, one obtains that $P^{\geom}_{\nu,w(\lambda)}(q)$ is equal to a polynomial $m^{w(\lambda)}_{\nu}(q)$ defined combinatorially and nowadays known as Lusztig's $q$-analog of the weight multiplicity (see e.g.~\cite[\S 2.3]{jlz} or \cite[Equation (3.3)]{bry} for a definition).

Now, consider the right-hand side of \eqref{eqn:polynomials-geom-BK}. By \cite[Theorem 7.6]{jlz}, we have
\begin{equation}
\label{eqn:Polynomial-BK}
P^{\mathrm{BK}}_{\nu,\lambda}(q) = q^{\lan w(\lambda)-\lambda,\rhov \ran} \cdot m^{w(\lambda)}_{\nu}(q).
\end{equation}
The result follows.
\end{proof}

\begin{rmk}
In this paper we use Theorem \ref{thm:filtrations} only in the case $\lambda=0$. In this case (and more generally in the case where $\lambda$ is dominant), \eqref{eqn:Polynomial-BK} is proved in \cite[Theorem 3.4]{bry} (under a cohomology vanishing assumption later proved in \cite{bro}).
\end{rmk}

\section{Mixed equivalence and convolution}
\label{sect:mix-convolution}

In this section we prove Proposition \ref{prop:FGmix-convolution}.

\subsection{Convolution with mixed perverse sheaves}

Recall the notation of \S \ref{ss:mixed-version-hl}. By definition, any object of $\Perv^0_{\mon{\GvO}}(\Gr_{\Gv})$ has a natural $\Gv(\F_p[[x]])$-equivariant structure. Let $M$ be in $\Perv^0_{\mon{\GvO}}(\Gr_{\Gv})$. It follows from Proposition \ref{prop:convolution-mix} that the convolution with $M$ defines a functor
\[
(-) \star M : \cD^{\Weil}_{\mon{\GvO}}(\Gr_{\Gv}) \to \cD^{\Weil}_{\mon{\GvO}}(\Gr_{\Gv}).
\]
Moreover, this functor preserves the additive subcategory $\Pure_{\mon{\GvO}}(\Gr_{\Gv})$, and induces a homogeneous functor $\Pure_{\mon{\GvO}}(\Gr_{\Gv}) \to \Pure_{\mon{\GvO}}(\Gr_{\Gv})$ in the sense of \cite[Definition 4.1]{ar}. Hence by \cite[Proposition 9.1]{ar} we have the following existence result.

\begin{prop}
\label{prop:mixed-version-convolution}

For any $M$ in $\Perv^0_{\mon{\GvO}}(\Gr_{\Gv})$, there exists a functor 
\[
(-) \star M : \cD^{\mix}_{\mon{\GvO}}(\Gr_{\Gv}) \to \cD^{\mix}_{\mon{\GvO}}(\Gr_{\Gv})
\]
which makes the diagram
\[
\xymatrix@R=0.5cm{
\cD^{\mix}_{\mon{\GvO}}(\Gr_{\Gv}) \ar[rr]^-{\iota} \ar[rd]_{\For} \ar[dd]_-{(-) \star M} & & \cD^{\Weil}_{\mon{\GvO}}(\Gr_{\Gv}) \ar[ld]^-{\varkappa} \ar[dd]^-{(-) \star M} & \\
& \cDb_{\mon{\GvO}}(\Gr_{\Gv}) \ar[dd]^(.3){(-) \star \Phi_G(M)} & \\
\cD^{\mix}_{\mon{\LvO}}(\Gr_{\Lv}) \ar'[r][rr]^(-.3){\iota} \ar[rd]_-{\For} & & \cD^{\Weil}_{\mon{\LvO}}(\Gr_{\Lv}) \ar[ld]^-{\varkappa} \\
& \cDb_{\mon{\LvO}}(\Gr_{\Lv}) & \\
}
\]
commutative up to isomorphisms of functors.\qed

\end{prop}

\subsection{Compatibility of $F_G^{\mix}$ with convolution}
\label{ss:compatibility}

Fix some $V$ in $\Rep(G)$. Consider the functors
\[
A_1, A_2 : \cDb_{\free} \Coh^{G \times \Gm}(\cN_G) \to \cDb_{\free} \Coh^{G \times \Gm}(\cN_G)
\]
defined as follows: 
\[
A_1(M)=M \otimes V \quad \text{ and } \quad
A_2(M) = F^{\mix}_G \bigl( (F^{\mix}_G)^{-1}(M) \star \cS_G^0(V) \bigr),
\]
where the object $(F^{\mix}_G)^{-1}(M) \star \cS_G^0(V)$ is defined in Proposition \ref{prop:mixed-version-convolution}.

Under the equivalence
\[
\Kb\bigl( \Coh_{\free}^{G \times \Gm}(\cN_G) \bigr) \ \cong \ \cDb_{\free} \Coh^{G \times \Gm}(\cN_G)
\]
of Lemma \ref{lem:orlov-coherent}, both functors $A_1$ and $A_2$ stabilize the subcategory $\Coh_{\free}^{G \times \Gm}(\cN_G)$. Moreover, it is easy to construct an isomorphism of additive functors
\[
(A_1){}_{|\Coh_{\free}^{G \times \Gm}(\cN_G)} \ \cong \ (A_2){}_{|\Coh_{\free}^{G \times \Gm}(\cN_G)}
\]
using Lemma \ref{lem:image-IC}(2). By \cite[Theorem 4.7]{ar}, we deduce that there exists an isomorphism of functors
\[
A_1 \ \cong \ A_2,
\]
which proves Proposition \ref{prop:FGmix-convolution}.

\section{Convolution and hyperbolic localization}
\label{sect:hl-convolution}

In this section we prove Proposition \ref{prop:convolution-hl}. Our arguments are independent of the rest of the paper.

\subsection{Reminder on nearby cycles}
\label{ss:nearby-cycles}

Recall the definition of the nearby cycles functor (see \cite[p.~98]{re} for more details and references). Let $X$ be a variety, and let $f: X \to \C$ be a morphism. Let $X_0:=f^{-1}(0)$, $X_U:=f^{-1}(\C^{\times})$, and $\widetilde{X}_U:= X_U \times_{\C^{\times}} \widetilde{\C^{\times}}$, where $\widetilde{\C^{\times}}$ is the universal cover of $\C^{\times}$. We have the following diagram, where the morphisms are the natural ones, and all squares are cartesian: 
\[
\xymatrix{
X_0 \ar[d] \ar@{^{(}->}[r]^-{i} & X \ar[d]^-{f} & X_U \ar@{_{(}->}[l]_-{j} \ar[d] & \ar[l]_-{v} \widetilde{X}_U \ar[d] \\ \{0\} \ar@{^{(}->}[r] & \C & \C^{\times} \ar@{_{(}->}[l] & \widetilde{\C^{\times}}. \ar[l]
}
\]
Then the nearby cycles functor is defined by
\[
\Psi_f := i^* j_* v_* v^* [-1] : \cDb_c(X_U) \to \cDb_c(X_0).
\]
(Here, $\cDb_c(-)$ is the bounded derived category of constructible sheaves.)

We will need some functoriality properties of this construction. Let $g: Y \to \C$ be another variety over $\C$, and consider a commutative diagram 
\[
\xymatrix{
Y \ar[d]_-{g} \ar[r]^-{\pi} & X \ar[d]^-{f} \\ \C \ar@{=}[r] & \C.
}
\]
We let $\pi_0 : Y_0 \to X_0$ and $\pi_U : Y_U \to X_U$ be the restrictions of $\pi$, and $\widetilde{\pi}_U : \widetilde{Y}_U \to \widetilde{X}_U$ be the morphism obtained from $\pi_U$ by base change. Then there are morphisms of functors 
\begin{align}
\label{eqn:Psi-inverse-*}
(\pi_0)^* \circ \Psi_f \ & \to \ \Psi_g \circ (\pi_U)^*, \\
\label{eqn:Psi-direct-!}
(\pi_0)_! \circ \Psi_g \ & \to \Psi_f \circ (\pi_U)_!, \\
\label{eqn:Psi-direct-*}
\Psi_f \circ (\pi_U)_* \ & \to \ (\pi_0)_* \circ \Psi_g, \\
\label{eqn:Psi-inverse-!}
\Psi_g \circ (\pi_U)^! \ & \to \ (\pi_0)^! \circ \Psi_f.
\end{align}
For example, let us construct morphism \eqref{eqn:Psi-inverse-*}. The other morphisms are constructed similarly. We use the notation above, adding indices ``$X$'' and ``$Y$'' to distinguish the morphisms associated to $f$ and $g$. Hence we have the following diagram, where all squares are cartesian:
\[
\xymatrix{
Y_0 \ar@{^{(}->}[r]^-{i_Y} \ar[d]_-{\pi_0} & Y \ar[d]^-{\pi} & Y_U \ar@{_{(}->}[l]_-{j_Y} \ar[d]^-{\pi_U} & \widetilde{Y}_U \ar[d]^{\widetilde{\pi}_U} \ar[l]_-{v_Y} \\ 
X_0 \ar@{^{(}->}[r]^-{i_X} & X & X_U \ar@{_{(}->}[l]_-{j_X} & \widetilde{X}_U. \ar[l]_-{v_X}
}
\]
By definition, we have 
\[
(\pi_0)^* \circ \Psi_f \ = \ (\pi_0)^* (i_X)^* (j_X)_* (v_X)_* (v_X)^* \ \cong \ (i_Y)^* \pi^* (j_X)_* (v_X)_* (v_X)^*.
\]
Here, the isomorphism follows from the identity $i_X \circ \pi_0 = i_Y \circ \pi$. Next, there are natural morphisms of functors 
\[
\pi^* (j_X)_* \to (j_Y)_* (\pi_U)^* \quad \text{and} \quad (\pi_U)^* (v_X)_* \to (v_Y)_* (\widetilde{\pi}_U)^*,
\]
induced by the adjunctions $\bigl( (\pi_U)^*, (\pi_U)_* \bigr)$ and $\bigl( (\widetilde{\pi}_U)^*, (\widetilde{\pi}_U)_* \bigr)$. Hence we obtain a morphism 
\[
(\pi_0)^* \circ \Psi_f \ \to \ (i_Y)^* (j_Y)_* (v_Y)_* (\widetilde{\pi}_U)^* (v_X)^*.
\]
Finally, using the equality $v_X \circ \widetilde{\pi}_U = \pi_U \circ v_Y$ and the definition of $\Psi_g$, one obtains morphism \eqref{eqn:Psi-inverse-*}.

\subsection{Convolution via nearby cycles}

Consider the convolution bifunctor 
\[
(- \star -) : \cDb_c(\Gr_{\Gv}) \times \cDb_{\eq{\GvO}}(\Gr_{\Gv}) \to \cDb_c(\Gr_{\Gv}).
\]
It induces, via the forgetful functor, a bifunctor denoted similarly
\[
(- \star -) : \cDb_c(\Gr_{\Gv}) \times \cDb_{\eq{\GvO \rtimes \mathrm{Aut}}}(\Gr_{\Gv}) \to \cDb_c(\Gr_{\Gv}).
\]
In this subsection we recall a construction of Gaitsgory (\cite{ga}) which allows to give an equivalent definition of this bifunctor in terms of nearby cycles.

In what follows, we set $X:=\C$ and $x=0 \in X$. For an algebraic group\footnote{In \cite{ga}, the author works with reductive groups. However, some of his constructions generalize to an arbitrary algebraic group, see e.g.~\cite[\S 4.5.1--4.5.2]{bd} and \cite[\S A.5]{ga}. We freely use these extensions.} $H$, the ind-scheme $\Gr_{H,X}'$ is defined in \cite[\S 3.1.1]{ga} as the ind-scheme which represents the functor sending an affine scheme $S$ to the set of triples $(y,\cF_{H},\beta')$, where $y$ is an $S$-point of $X$, $\cF_{H}$ is an $H$-bundle over $X \times S$, and $\beta'$ is a trivialization of $\cF_{H}$ off the divisor $\Gamma_y \cup (x \times S)$. Here, $\Gamma_y$ is the graph of $y$. Let $\Gr_{H,X \smallsetminus x}'$, resp. $\Gr_{H,x}'$, be the restriction of $\Gr_{H,X}'$ to $X \smallsetminus \{x\}$, resp. $x$.

Following \cite[Lemma 3]{ga}, we also define the ind-scheme $\Gr_{H,X}$, which represents the functor sending an affine scheme $S$ to the set of triples $(y,\cF_{H},\beta)$, where $y$ is an $S$-point of $X$, $\cF_{H}$ is an $H$-bundle over $X \times S$, and $\beta$ is a trivialization of $\cF_{H}$ off the divisor $\Gamma_y$. As above, let $\Gr_{H,X \smallsetminus x}$, resp. $\Gr_{H,x}$, be the restriction of $\Gr_{H,X}$ to $X \smallsetminus \{x\}$, resp. $x$. By \cite[Lemma 3]{ga}, there is an isomorphism of ind-schemes
\[
\Gr_{H,X} \ \cong \ \mathfrak{X} \times^{\mathrm{Aut}} \Gr_H,
\]
where $\mathfrak{X}$ is the canonical $\mathrm{Aut}$-torsor over $X$, see \cite[\S 2.1.2]{ga}. (Here, as in the preceding sections, $\Gr_H$ is the affine Grassmannian of $H$, see \cite[\S 4.5]{bd}, endowed with the natural action of $\mathrm{Aut}$.)

By \cite[Proposition 5]{ga}, there are natural isomorphisms 
\begin{equation}
\label{eqn:isomorphisms-Gr-Gr'} \Gr'_{H,X \smallsetminus x} \cong \Gr_{H, X \smallsetminus x} \times \Gr_{H}, \quad \Gr'_{H,x} \cong \Gr_H.
\end{equation}

There is a functor
\[
\cDb_{\eq{\mathrm{Aut}}}(\Gr_H) \to \cDb_c(\Gr_{H,X}),
\]
denoted $M \mapsto M_X$, where $M_X$ is the twisted external product $\uC{}_X[1] \tsqtimes M$, an object of the derived category of sheaves on $\mathfrak{X} \times^{\mathrm{Aut}} \Gr_H \cong \Gr_{H,X}$. We denote by $M_{X \smallsetminus x}$ the restriction of $M_X$ to $\Gr_{H,X \smallsetminus x}$.

Finally we can explain the construction of the bifunctor $\cC_H(-,-)$ of \cite[\S 3.2]{ga}. Starting from $M$ in $\cDb_c(\Gr_{H})$ and $N$ in $\cDb_{\eq{\mathrm{Aut}}}(\Gr_{H})$, one can consider the external product
\[
N_{X \smallsetminus x} \boxtimes M
\]
in $\cDb_{\mathrm{const}}(\Gr'_{H, X \smallsetminus x})$. (Here we have used the first isomorphism in \eqref{eqn:isomorphisms-Gr-Gr'}.) Let
\[
\Psi^H : \cDb_c(\Gr'_{H, X \smallsetminus x}) \to \cDb_c(\Gr_{H}),
\]
be the nearby cycle functor associated to the natural map $\Gr'_{H,X} \to X=\C$ (see \S \ref{ss:nearby-cycles}). (Here we have used the second isomorphism in \eqref{eqn:isomorphisms-Gr-Gr'}.) Composing these two operations, one obtains a bifunctor
\[
\cC_H(-,-) : \cDb_c(\Gr_{H}) \times \cDb_{\eq{\mathrm{Aut}}}(\Gr_{H}) \to \cDb_c(\Gr_{H}),
\]
such that
\[
\cC_H(M,N)= \Psi^H(N_{X \smallsetminus x} \boxtimes M).
\]
The following result is part of \cite[Proposition 6(b)]{ga}.

\begin{prop} 
\label{prop:gaitsgory}

Assume $H$ is reductive.

For $M$ in $\cDb_c(\Gr_{H})$ and $N$ in $\cDb_{\eq{H(\cO) \rtimes \mathrm{Aut}}}(\Gr_{H})$, there exists a bifunctorial isomorphism
\[
\cC_H(M,N) \ \cong \ M \star N.
\]

\end{prop}

We will also need the following technical result.

\begin{lem}
\label{lem:globalization-duality}

The functor $M \mapsto M_X$ commutes with Verdier duality. In other words, for $M$ in $\cDb_{\eq{\mathrm{Aut}}}(\Gr_H)$ there is a functorial isomorphism
\[
\D_{\Gr_{H,X}}(M_X) \ \cong \ (\D_{\Gr_H}(M))_X.
\]

\end{lem}

\begin{proof}
For simplicity, in the proof we omit the subscript ``$H$.'' By definition, $M$ is supported on a finite-dimensional $\mathrm{Aut}$-stable closed subvariety $Y \subset \Gr$. (For simplicity again, we omit direct image functors under closed embeddings.) Moreover, the action of $\mathrm{Aut}$ on $Y$ factors through a finite dimensional quotient $\mathrm{Aut}_0=\mathrm{Aut}/K$ of $\mathrm{Aut}$ such that $M$ is an $\mathrm{Aut}_0$-equivariant complex on $Y$. Consider the morphisms
\[
\xymatrix{
\fX \times^{\mathrm{Aut}} Y \, = \, (\fX/K) \times^{\mathrm{Aut}_0} Y & (\fX/K) \times Y \ar[l]_-{a} \ar[r]^-{b} & X \times Y.
}
\]
By definition, $\D_{\Gr}(M) = \D_{Y}(M)$. And $(\D_{\Gr}(M))_X$ is the unique object of the category $\cDb_c(\fX \times^{\mathrm{Aut}} Y)$ such that
\[
a^*\bigl( (\D_{\Gr}(M))_X \bigr) \ \cong \ b^* \bigl( \underline{\C}_X[1] \boxtimes \D_{Y}(M) \bigr).
\]
Hence, to prove the lemma, we only have to prove that $\D_{\Gr_{X}}(M_X) = \D_{\fX \times^{\mathrm{Aut}} Y}(M_X)$ satisfies this condition. However, as $a$ is a smooth map of relative dimension $r:=\dim(\mathrm{Aut}_0)$ we have
\[
a^*(\D_{\fX \times^{\mathrm{Aut}} Y}(M_X)) \, \cong \, \D_{(\fX/K) \times Y}(a^! (M_X)) \, \cong \, \D_{(\fX/K) \times Y}(a^* (M_X) [r]).
\]
By definition of $M_X$ we have $a^* (M_X) \cong b^* (\underline{\C}_X[1] \boxtimes M)$. As $b$ is also smooth of relative dimension $r$ we obtain
\begin{multline*}
a^*(\D_{\fX \times^{\mathrm{Aut}} Y}(M_X)) \, \cong \, \D_{(\fX/K) \times Y}(b^* (\underline{\C}_X[1] \boxtimes M) [r]) \, \cong \, \D_{(\fX/K) \times Y}(b^! (\underline{\C}_X[1] \boxtimes M)) \\ \cong \, b^* \D_{X \times Y}(\underline{\C}_X[1] \boxtimes M) \, \cong \, b^* (\underline{\C}_X[1] \boxtimes \D_Y(M)).
\end{multline*}
This concludes the proof.
\end{proof}

\subsection{Convolution and hyperbolic localization}
\label{ss:convolution-hl}

Finally we can prove Proposition \ref{prop:convolution-hl}.

Recall the notation of \S \ref{ss:hl-restriction}. The morphisms $\Pv \hookrightarrow \Gv$, $\Pv^- \hookrightarrow \Gv$, $\Pv \twoheadrightarrow \Lv$, $\Pv^- \twoheadrightarrow \Lv$ induce morphisms of ind-schemes 
\begin{align*}
i': \Gr_{\Pv,X}' \to \Gr_{\Gv,X}', & \quad j': \Gr_{\Pv^-,X}' \to \Gr_{\Gv,X}', \\
p': \Gr_{\Pv,X}' \to \Gr_{\Lv,X}', & \quad q': \Gr_{\Pv^-,X}' \to \Gr_{\Lv,X}'
\end{align*}
(via induction of $\Pv$-bundles or $\Pv^-$-bundles to $\Gv$-bundles, and quotient of $\Pv$-bundles or $\Pv^-$-bundles to $\Lv$-bundles). By definition, and using the identifications given by the second isomorphism in \eqref{eqn:isomorphisms-Gr-Gr'}, the hyperbolic localization functor is
\[
h_L^{!*}:=(p_0')_! (i_0')^* : \cDb_c(\Gr_{\Gv}) \to \cDb_c(\Gr_{\Lv}).
\]
By \cite[Theorem 1]{br}, on the category $\cDb_{\mon{\lambda_L(\C^{\times})}}(\Gr_{\Gv})$ there is an isomorphism of functors
\begin{equation}
\label{eqn:isom-braden}
(p_0')_! (i_0')^* \ \cong \ (q_0')_* (j_0')^! \ : \cDb_{\mon{\lambda_L(\C^{\times})}}(\Gr_{\Gv}) \to \cDb_c(\Gr_{\Lv}).
\end{equation}

By \eqref{eqn:Psi-direct-!} and \eqref{eqn:Psi-inverse-*}, respectively by \eqref{eqn:Psi-inverse-!} and \eqref{eqn:Psi-direct-*}, there are natural morphisms of functors
\begin{equation}
\label{eqn:morphism-hl-nearby}
(p_0')_! (i_0')^* \circ \Psi^G \to \Psi^L \circ (p_U')_! (i_U')^*, \quad \Psi^L \circ (q_U')_* (j_U')^! \to (q_0')_* (j_0')^! \circ \Psi^G.
\end{equation}

Now, there are also natural morphisms 
\begin{align*}
i'': \Gr_{\Pv,X} \to \Gr_{\Gv,X}, \quad j'': \Gr_{\Pv^-,X} \to \Gr_{\Gv,X}, \\
p'': \Gr_{\Pv,X} \to \Gr_{\Lv,X}, \quad q'': \Gr_{\Pv^-,X} \to \Gr_{\Lv,X}.
\end{align*}
Note that, via the identifications \eqref{eqn:isomorphisms-Gr-Gr'}, we have 
\begin{align*}
i_0'=i_0'', \quad j_0'=j_0'', & \quad p_0'=p_0'', \quad q_0'=q_0'', \\
i_U' = (i_U'' \times i_0'), \quad j_U' = (j_U'' \times j_0'), & \quad p_U' = (p_U'' \times p_0'), \quad q_U' = (q_U'' \times q_0').
\end{align*}
In particular, we get isomorphisms of bifunctors, for $M_1$ in $\cDb_c(\Gr_{\Gv,X \smallsetminus x})$ and $M_2$ in $\cDb_c(\Gr_{\Gv})$,
\begin{align}
\label{eqn:isomorphism-hl}
(p_U')_! (i_U')^*(M_1 \boxtimes M_2) \ & \cong \ \bigl( (p_U'')_! (i_U'')^* M_1 \bigr) \boxtimes ( (p_0')_! (i_0')^* M_2), \\
\label{eqn:isomorphism-hl-2}
(q_U')_* (j_U')^!(M_1 \boxtimes M_2) \ & \cong \ \bigl( (q_U'')_* (j_U'')^! M_1 \bigr) \boxtimes ( (q_0')_* (j_0')^! M_2).
\end{align}

We claim that for $M$ in $\cDb_{\eq{\mathrm{Aut}}}(\Gr_{\Gv})$, there is a functorial isomorphism
\begin{equation}
\label{eqn:hl-globalization}
(p_U'')_! (i_U'')^* (M_{X \smallsetminus x}) \ \cong \ ((p_0')_! (i_0')^* M)_{X \smallsetminus x}.
\end{equation}
Using base change for the diagram
\[
\xymatrix{
\Gr_{\Pv,X \smallsetminus x} \ar@{^{(}->}[r] \ar[d]_-{p_U''} & \Gr_{\Pv,X} \ar[d]^-{p''} \\
\Gr_{\Lv,X \smallsetminus x} \ar@{^{(}->}[r] & \Gr_{\Lv,X},
}
\]
the left-hand side is isomorphic to $\bigl( (p'')_! (i'')^* (M_X)
\bigr)_{|\Gr_{\Lv, X \smallsetminus x}}$. Hence isomorphism \eqref{eqn:hl-globalization} would follow from an isomorphism
\begin{equation}
\label{eqn:hl-globalization'}
(p'')_! (i'')^* (M_{X}) \ \cong \ ((p_0')_! (i_0')^* M)_{X}.
\end{equation}
However, it follows directly from the definition that
\begin{equation}
\label{eqn:globalization-inverse-*}
(i'')^* M_X \cong ((i_0')^* M)_X.
\end{equation}
Now, consider the cartesian diagram
\[
\xymatrix{
\mathfrak{X} \times \Gr_{\Pv} \ar[r]^-{\id \times p_0'} \ar[d]_-{\pi_{\Pv}} & \mathfrak{X} \times \Gr_{\Lv} \ar[d]_-{\pi_{\Lv}} \\
\mathfrak{X} \times^{\mathrm{Aut}} \Gr_{\Pv} \ar[r]^-{p''} & \mathfrak{X} \times^{\mathrm{Aut}} \Gr_{\Lv}.
}
\]
The base change theorem gives an isomorphism
\[
(\pi_{\Lv})^* (p'')_! \ \cong \ (\id \times p_0')_! (\pi_{\Pv})^*.
\]
By definition, for $N$ in $\cDb_{\eq{\mathrm{Aut}}}(\Gr_{\Lv})$, $N_X$ is the only object of $\cDb_c(\Gr_{\Lv,X})$ such that $(\pi_{\Lv})^* N_X \cong \uC{}_{\mathfrak{X}} [1] \boxtimes N$, and the same is true for $\Gr_{\Pv}$. Hence we have a functorial isomorphism
\begin{equation}
\label{eqn:globalization-direct-!}
\bigl( (p_0')_! \, N \bigr)_X \ \cong \ (p'')_! \, N_X
\end{equation}
for $N$ in $\cDb_{\eq{\mathrm{Aut}}}(\Gr_{\Pv})$. Combining \eqref{eqn:globalization-inverse-*} and \eqref{eqn:globalization-direct-!}, one gets \eqref{eqn:hl-globalization'}, hence also \eqref{eqn:hl-globalization}.

Now, if $M$ is in $\cDb_{\eq{\mathrm{Aut}}}(\Gr_{\Gv})$, we claim that there is a functorial isomorphism
\begin{equation}
\label{eqn:hl-globalization-2}
(q''_U)_* (j''_U)^! (M_{X \smallsetminus x}) \ \cong \ ((q_0')_* (j_0')^! M)_{X \smallsetminus x}.
\end{equation}
Indeed, applying \eqref{eqn:hl-globalization'} to $\Pv^-$ instead of $\Pv$, and using the fact that $M \mapsto M_X$ commutes with Verdier duality (see Lemma \ref{lem:globalization-duality}), we obtain a functorial isomorphism
\[
(q'')_* (j'')^! (M_X) \ \cong \ ((q_0')_* (j_0')^! M)_{X}.
\]
Restricting to $\Gr_{\Lv,X \smallsetminus x}$ gives \eqref{eqn:hl-globalization-2}.

Combining the first morphism of functors in \eqref{eqn:morphism-hl-nearby} and isomorphisms \eqref{eqn:isomorphism-hl} and \eqref{eqn:hl-globalization}, one obtains for $M$ in $\cDb_{\mon{\lambda_L(\C^{\times})}}(\Gr_{\Gv})$ and $N$ in $\cDb_{\eq{\GvO \rtimes \mathrm{Aut}}}(\Gr_{\Gv})$ a bifunctorial morphism
\[
h_L^{!*} \circ \cC_{\Gv}(M,N) \ \to \ \cC_{\Lv}(h_L^{!*}(M),h_L^{!*}(N)),
\]
hence, using Proposition \ref{prop:gaitsgory} and the definition of $\fR^G_L$, a bifunctorial morphism
\[
\fR^G_L (M \star N) \ \to \ \fR^G_L(M) \star \fR^G_L(N).
\]
Similarly, using isomorphisms \eqref{eqn:isom-braden}, \eqref{eqn:isomorphism-hl-2} and \eqref{eqn:hl-globalization-2}, the second morphism of functors in \eqref{eqn:morphism-hl-nearby} provides a bifunctorial morphism
\[
\fR^G_L(M) \star \fR^G_L(N) \ \to \ \fR^G_L (M \star N).
\]
One can check that these two morphisms are inverse to each other, which concludes the proof of Proposition \ref{prop:convolution-hl}.

\section{Example: $\Gv=\mathrm{SL}(2)$}
\label{sect:example}

In this section we concentrate on the case $\Gv=\mathrm{SL}(2)$. We choose as $\Tv$ the subgroup of diagonal matrices, and as $\Bv$ the subgroup of upper triangular matrices. Then we have $G=\mathrm{PSL}(2)$. There is a natural isomorphism $\bX \cong 2\Z$, which matches $2\rho$ with $2$. We denote by $X_n$ the $\Iv$-orbit of $L_n$, so that we have inclusions
\[
X_0=\{L_0\} \subset \overline{X_{-2}} \subset \overline{X_{2}} \subset \overline{X_{-4}} \subset \cdots
\]
and we have
\[
\dim(X_n)=\left\{ 
\begin{array}{cl}
n & \text{if } n \geq 0; \\ 
-n-1 & \text{if } n<0.
\end{array}
\right.
\]

For any $n \in 2\Z$, we denote by $j_n : X_n \hookrightarrow \Gr_{\Gv}$ the inclusion, and by $\IC_n$ the simple perverse sheaf $\IC(\overline{X_n})$.

\subsection{Simple objects}

It is well known that in this case all closures of $\Iv$-orbits are rationally smooth. For the reader's convenience, we include a simple proof of this fact.

\begin{prop}
\label{prop:simples-SL2}

For any $n \in 2\Z$ we have
\[
\IC_n \ \cong \ \underline{\C}_{\overline{X_n}}[\dim X_n].
\]

\end{prop}

\begin{proof}
We proceed by induction on the dimension of orbits. The case $n=0$ is obvious. To fix notation, we assume that $n \geq 0$ and that the result is known for $X_n$, and we prove it for $X_{-n-2}$.

Let ${\check Q}$ be the (parahoric) subgroup of $\Gv(\fK)$ generated by $\Iv$ and by the matrix
\[
\left(
\begin{array}{cc}
0 & x^{-1} \\
-x & 0
\end{array}
\right).
\]
Then we have ${\check Q}/\Iv \cong \bP^1$. Consider the morphism
\[
\pi : {\check Q} \times^{\Iv} \overline{X_n} \to \Gr_{\Gv}
\]
induced by the $\Gv(\fK)$-action on $\Gr_{\Gv}$ by left multiplication. It is well known that its image is $\overline{X_{-n-2}}$. By induction, the shifted constant sheaf $\underline{\C}_{{\check Q} \times^{\Iv} \overline{X_n}}[n+1]$ is a simple perverse sheaf. The decomposition theorem, and the fact that $\pi$ is a semismall map, imply that the direct image
\[
\cK := \pi_* \bigl( \underline{\C}_{{\check Q} \times^{\Iv} \overline{X_n}}[n+1] \bigr)
\]
is a semisimple perverse sheaf.

The morphism $\pi$ is an isomorphism over $X_{-n-2} \cup X_n$, and its fibers over $\overline{X_{-n}}$ are isomorphic to $\bP^1$. Hence the cohomology of the stalks of $\cK$ are as follows:\medskip
\[
\begin{tabular}[c]{|c|c|c|c|c|}
\hline $\dim$ & orbit & $-n-1$ & $-n$ & $-n+1$ \\
\hline $n+1$ & $X_{-n-2}$ & $\C$ & $0$ & $0$ \\
\hline $n$ & $X_n$ & $\C$ & $0$ & $0$ \\
\hline $n-1$ & $X_{-n}$ & $\C$ & $0$ & $\C$ \\
\hline $\vdots$ & & $\vdots$ & & $\vdots$ \\
\hline $0$ & $X_0$ & $\C$ & $0$ & $\C$ \\ 
\hline
\end{tabular}
\] \medskip
From this table and the fact that the $\Iv$-orbits are simply connected, we deduce (looking at the diagonal) that 
\[
\cK \ \cong \ \IC_{-n-2} \oplus \IC_{-n}.
\]
By induction, we know that $\IC_{-n} \cong \underline{\C}_{\overline{X_{-n}}}[n-1]$. We deduce again from the table that $\IC_{-n-2} \cong \underline{\C}_{\overline{X_{-n-2}}}[n+1]$ (see \cite[Proposition 1.4]{bm}).
\end{proof}

\subsection{Standard and projective objects}

\begin{prop}
\label{prop:standard-sl2}

For $n \in 2\Z$, the composition factors of the standard objects $\Delta_n:=(j_n)_! \C[\dim X_n]$ are given as follows:
\begin{itemize}
\item $\Delta_0=\IC_0$;
\item for $n>0$,
\[
\Delta_n = 
\begin{tabular}{|c|} 
\hline $\IC_n$ \\
\hline $\IC_{-n}$ \\ 
\hline
\end{tabular} \ ;
\]
\item for $n<0$,
\[
\Delta_{n} = 
\begin{tabular}{|c|} 
\hline $\IC_{n}$ \\ 
\hline $\IC_{-n-2}$ \\ 
\hline
\end{tabular} \ .
\]
\end{itemize}

\end{prop}

\begin{proof}
For $n>0$, the natural exact sequence of sheaves
\[
(j_n)_! \uC_{X_n} \hookrightarrow \uC_{\overline{X_n}} \twoheadrightarrow \uC_{\overline{X_{-n}}}
\]
induces an exact sequence of perverse sheaves
\[
\IC_{-n} \hookrightarrow \Delta_n \twoheadrightarrow \IC_n.
\]
(Here we use Proposition \ref{prop:simples-SL2}.) The case $n<0$ is similar.\end{proof}

By duality, we obtain the following.

\begin{prop}
\label{prop:costandard-sl2}

For $n \in 2\Z$, the composition factors of the costandard objects $\nabla_n:=(j_n)_* \C[\dim X_n]$ are given as follows:
\begin{itemize}
\item $\nabla_0=\IC_0$;
\item for $n>0$,
\[
\nabla_n = 
\begin{tabular}{|c|} 
\hline $\IC_{-n}$ \\
\hline $\IC_n$ \\ 
\hline
\end{tabular} \ ;
\]
\item for $n<0$,
\[
\nabla_{n} = 
\begin{tabular}{|c|} 
\hline $\IC_{-n-2}$ \\ 
\hline $\IC_{n}$ \\ 
\hline
\end{tabular} \ .
\]
\end{itemize}

\end{prop}

\begin{cor}
\label{cor:projectives-sl2}

Let $n \in 2\Z$.

\begin{itemize}
\item For $n \geq 0$, the projective cover of $\IC_n$ in the category $\Perv_{\mon{\Iv}}(\overline{X_{-n-2}})$ is still projective in the category $\Perv_{\mon{\Iv}}(\overline{X})$ for any $\Iv$-orbit $X$ whose closure contains $X_{-n-2}$, hence also in the category $\Perv_{\mon{\Iv}}(\Gr_{\Gv})$. Its standard flag and Jordan--H{\"o}lder series are given as follows:
\[
P_n = 
\begin{tabular}{|c|} 
\hline $\Delta_n$ \\ 
\hline $\Delta_{-n-2}$ \\ 
\hline
\end{tabular} \ = \
\begin{tabular}{|c|} 
\hline $\IC_n$ \\ 
\hline $\IC_{-n} \oplus \IC_{-n-2}$ \\ 
\hline $\IC_n$ \\ 
\hline
\end{tabular} \
\]
if $n>0$, and 
\[
P_0 = 
\begin{tabular}{|c|} 
\hline $\Delta_0$ \\ 
\hline $\Delta_{-2}$ \\ 
\hline
\end{tabular} \ = \
\begin{tabular}{|c|} 
\hline $\IC_0$ \\ 
\hline $\IC_{-2}$ \\ 
\hline $\IC_0$ \\ 
\hline
\end{tabular} \ .
\]

\item For $n<0$, the projective cover of $\IC_n$ in the category $\Perv_{\mon{\Iv}}(\overline{X_{-n}})$ is still projective in the category $\Perv_{\mon{\Iv}}(\overline{X})$ for any $\Iv$-orbit $X$ whose closure contains $X_{-n}$, hence also in the category $\Perv_{\mon{\Iv}}(\Gr_{\Gv})$. Its standard flag and Jordan--H{\"o}lder series are given as follows:
\[
P_n = 
\begin{tabular}{|c|} 
\hline $\Delta_n$ \\ 
\hline $\Delta_{-n}$ \\ 
\hline
\end{tabular} \ = \
\begin{tabular}{|c|} 
\hline $\IC_{n}$ \\ 
\hline $\IC_{-n-2} \oplus \IC_{-n}$ \\ 
\hline $\IC_n$ \\ 
\hline
\end{tabular} \ .
\]

\end{itemize}

\end{cor}

\begin{proof}
This follows from Proposition \ref{prop:costandard-sl2} and the reciprocity formula, see \cite[Remark (1) after Theorem 3.2.1]{bgs}.
\end{proof}

In particular, it follows from this corollary that the category $\Perv_{\mon{\Iv}}(\Gr_{\Gv})$ has enough projectives, in accordance with Remark \ref{rmk:projectives}.

\subsection{Projective resolution of $1_G$}

In this case, the projective resolution of the object $1_G$ constructed in \S \ref{ss:projective-resolution} looks as follows:
\[
\cdots \to P_4 \to P_{-4} \to P_{2} \to P_{-2} \to P_0 \twoheadrightarrow \IC_0=1_G.
\]

\subsection{Convolution of simple perverse sheaves}

\begin{prop}
\label{prop:convolution-sl2}

Let $n,k \in 2\Z_{\geq 0}$.

\begin{itemize}

\item If $n \geq k$ we have
\[
\IC_n \star \IC_k \ \cong \ \IC_k \star \IC_n \ \cong \ \IC_{n+k} \oplus \IC_{n+k-2} \oplus \cdots \oplus \IC_{n-k}.
\]

\item If $n > k$ we have
\[
\IC_{-n} \star \IC_k \ \cong \ \IC_{-n-k} \oplus \IC_{-n-k+2} \oplus \cdots \oplus \IC_{-n+k}.
\]

\item If $0 < n \leq k$ we have
\[
\IC_{-n} \star \IC_k \ \cong \ \IC_{-n-k} \oplus \IC_{-n-k+2} \oplus \cdots \oplus \IC_{n-k-2}.
\]

\end{itemize}

\end{prop}

\begin{proof}
The first formula follows from representation theory of $\mathrm{PGL}(2)$, via the Satake equivalence. The second and third formulas in the case $n=2$ were proved in the course of the proof of Proposition \ref{prop:simples-SL2} (in a different language). The general case can be proved by induction on $n$ using the following two expressions:
\begin{align*}
\IC_{-2} \star \IC_n \star \IC_k \ & \cong \ \IC_{-2} \star \bigl(\IC_{n+k} \oplus \cdots \oplus \IC_{|n-k|} \bigr) \\
& \cong \ \bigl( \IC_{-n-2} \oplus \IC_{-n} \bigr) \star \IC_k,
\end{align*}
where in the first isomorphism we use the first formula, and in the second one we use the case ``$n=2$'' of the third formula.
\end{proof}

\subsection{Dg-algebra}

Using Corollary \ref{cor:projectives-sl2} and Proposition \ref{prop:convolution-sl2}, one can describe the dg-algebra $\Hom^{\bullet}(P^{\bullet},P^{\bullet} \star \cR_G)$ concretely.

For example, the dg-algebra
\[
\Hom^{\bullet}(P^{\bullet}, P^{\bullet})
\]
is isomorphic to the product of an infinite number of exact complexes
\[
\C \hookrightarrow \C \oplus \C \twoheadrightarrow \C
\]
(in degrees $-1$, $0$ and $1$) parametrized by $\Z_{<0}$ and one copy of the complex
\[
\C \hookrightarrow \C \oplus \C \to \{0\}
\]
(again in degrees $-1$, $0$ and $1$) with cohomology $\C$ in degree $0$.

Similarly, the complex
\[
\Hom^{\bullet}(P^{\bullet}, P^{\bullet} \star \IC_2)
\]
is isomorphic to the product of an infinite number of exact complexes
\[
\C \hookrightarrow \C^2 \to \C^2 \to \C^2 \to \C^2 \to \C^2 \twoheadrightarrow \C
\]
(in degrees between $-3$ and $3$) parametrized by $\Z_{<-2}$, one copy of the complex
\[
\C \hookrightarrow \C^2 \to \C^2 \to \C^2 \to \C^2 \to \C^2 \to 0
\]
(again in degrees between $-3$ and $3$) with cohomology $\C$ in degree $2$, one copy of the exact complex
\[
\C \hookrightarrow \C^2 \to \C^2 \to \C^2 \twoheadrightarrow \C
\]
(in degrees between $-3$ and $1$) and one copy of the exact complex
\[
\C \hookrightarrow \C^2 \twoheadrightarrow \C
\]
(in degrees between $-3$ and $-1$).

The description of
\[
\Hom^{\bullet}(P^{\bullet}, P^{\bullet} \star \IC_n)
\]
for a general $n \in 2\Z_{\geq 0}$ is similar.

\end{document}